\documentclass[11pt,bull-l]{amsart}
\usepackage[top = 0.9in, bottom = 0.9in, left =1in, right = 1in]{geometry}
\usepackage{geometry}
\usepackage[utf8]{inputenc}
\usepackage{hyperref}
\usepackage{listings}
\usepackage{multimedia} 
\usepackage{graphicx}
\usepackage{comment}
\usepackage[english]{babel}
\usepackage{amsfonts,amscd,amssymb,amsmath,mathrsfs}
\usepackage{wrapfig}
\usepackage{multirow}
\usepackage{verbatim,cite}
\usepackage{float}
\usepackage{cancel}
\usepackage{caption}
\usepackage{subcaption}
\usepackage{mathdots}
\usepackage{bbm}
\usepackage{tikz}
\usepackage{paper-defs}
\usepackage{xcolor}

\title[Computing the inverse spectral map for measures supported on disjoint intervals]{A Riemann--Hilbert approach to computing the inverse spectral map for measures supported on disjoint intervals}

\author{Cade Ballew}
\address{University of Washington, Seattle, WA}
\email{ballew@uw.edu}

\author{Thomas Trogdon}
\address{University of Washington, Seattle, WA}
\email{trogdon@uw.edu}

\date{}
\thanks{This material is based on work supported by the National Science Foundation under Grant No. DMS-1945652 (TT).
Any opinions, findings, and conclusions or recommendations expressed in this material are those of the authors and do not necessarily reflect the views of the National Science Foundation.}

\begin{document}

\maketitle

\begin{abstract}
We develop a numerical method for computing with orthogonal polynomials that are orthogonal on multiple, disjoint intervals for which analytical formulae are currently unknown. Our approach exploits the Fokas--Its--Kitaev Riemann--Hilbert representation of the orthogonal polynomials to produce an $\OO(N)$ method to compute the first $N$ recurrence coefficients. The method can also be used for pointwise evaluation of the polynomials and their Cauchy transforms throughout the complex plane. The method encodes the singularity behavior of weight functions using weighted Cauchy integrals of Chebyshev polynomials. This greatly improves the efficiency of the method, outperforming other available techniques.  We demonstrate the fast convergence of our method and present applications to integrable systems and approximation theory.
\end{abstract}

\section{Introduction}
Orthogonal polynomials are fundamental objects of study in computational mathematics with numerous applications ranging from quadrature \cite{Golub1969} and eigenvalue approximation \cite{Lanczos1950} to differential equations \cite{Olver2013a} and random matrix theory \cite{deift_2000}, and beyond \cite{DEIFT1985358, gautschi, olver_slevinsky_townsend_2020, Szeg1939}. The theory of orthogonal polynomials dates back to the work of Chebyshev, Jacobi, Lagrange, and others. Much of this theory is summarized in the book of Szeg\H{o} \cite{Szeg1939}. The modern theory of orthogonal polynomials has been intertwined with the theory of Riemann--Hilbert problems since the discovery of the Fokas, Its, and Kitaev (FIK) Riemann--Hilbert problem \cite{Fokas1992}. This characterization of orthogonal polynomials in terms of a solution of a Riemann--Hilbert problem, combined with the Deift--Zhou method of nonlinear steepest descent \cite{deift_zhou}, has significantly advanced the asymptotic knowledge of orthogonal polynomials. Much of this advancement is presented in the book of Deift \cite{deift_2000} which is fundamental to our work (see also \cite{Kuijlaars2003}).

Orthogonal polynomials are arguably best utilized via the coefficients that make up their three-term recurrence. These recurrence coefficients, in a sense, fully characterize their orthogonal polynomials, allowing one to compute their values, their zeroes, and their normalization coefficients \cite{gautschi}. Thus, orthogonal polynomials that are otherwise not characterized may be utilized numerically if their recurrence coefficients can be reliably computed. The Jacobi matrix of three-term recurrence coefficients has the orthogonality measure as its spectral measure \cite{deift_2000}, and so computing the three-term recurrence coefficients, given the measure, is often referred to as the computing the inverse spectral map.

Historically, given an orthogonality weight, the orthogonal polynomial recurrence coefficients have been computed via a discretization of the Gram--Schmidt procedure known as the Stieltjes procedure \cite{gautschi}. While reliable, this method, in a different setting, has been realized to require $\OO(N^3)$ arithmetic operations to compute the first $N$ recurrence coefficients \cite{tomjacobi}, leaving significant room for speedup through a Riemann--Hilbert approach. An $\OO(N^2)$ algorithm based on \cite{Gragg1984} is presented in \cite{gautschi} but is not without its own limitations. Obtaining numerical solutions to the FIK Riemann--Hilbert problem is a relatively recent area of research, and such methods are primarily based on techniques developed by Olver \cite{Olver2012}. These techniques were applied in \cite{tomjacobi} to compute recurrence coefficients for orthogonal polynomials with exponential weights. This work demonstrated that a Riemann--Hilbert approach requires only $\OO(N)$ arithmetic operations to compute the first $N$ recurrence coefficients. Furthermore, this approach allows such coefficients to be computed individually in $\OO(1)$ arithmetic operations so that methods based on this approach are optimal in a time complexity sense. It is also possible that an increasingly analytical approach could be employed \cite{Huybrechs2016} to accomplish our same goal, and we leave this for future work, but the complication here is that asymptotic formulae are given in terms of theta functions \cite{Ding2022,YATTSELEV201573} which are themselves non-trivial to compute \cite{Deconinck2004}.

In this paper, we develop and apply a numerical Riemann--Hilbert approach to the problem presented in \cite{KUIJLAARS2004337} and further extended in \cite{Ding2022} to compute the recurrence coefficients for Chebyshev-like polynomials with support on multiple, disjoint intervals on the real line. Such orthogonal polynomials have both analytical and computational applications \cite{Geronimo1988, Aptekarev1986, Saad, deBoor}. For example, such orthogonal polynomials produce a new class of iterative solvers \cite{Saad}, improve matrix spectrum approximation when bounds on the eigenvalues are known a priori \cite{Chen}, and describe the numerical evolution of a Toda lattice \cite{DEIFT1985358}. Analytical formulae are not known for most of the cases we consider. The method remains accurate as the degree of the polynomial under consideration grows, and it bypasses the need to compute theta functions. More specifically, we compute orthogonal polynomial recurrence coefficients for weight functions of the form 
\begin{equation}\label{weight}
w(x)=\sum_{j=1}^{g+1}\mathbbm{1}_{[a_j,b_j]}(x)h_j(x)\left(\sqrt{x-a_j}\right)^{\alpha_j}\left(\sqrt{b_j-x}\right)^{\beta_j},\quad \alpha_j,\beta_j\in\{-1,1\},
\end{equation}
for a set of disjoint intervals, $\bigcup_{j=1}^{g+1}[a_j,b_j]$ where each $h_j$ is positive on $[a_j,b_j]$ and analytic in a neighborhood $\Omega_j$ of $[a_j,b_j]$. We include a plot of a weight function of this form in Figure \ref{weightfunc}.
\begin{figure}
\includegraphics[scale=0.5]{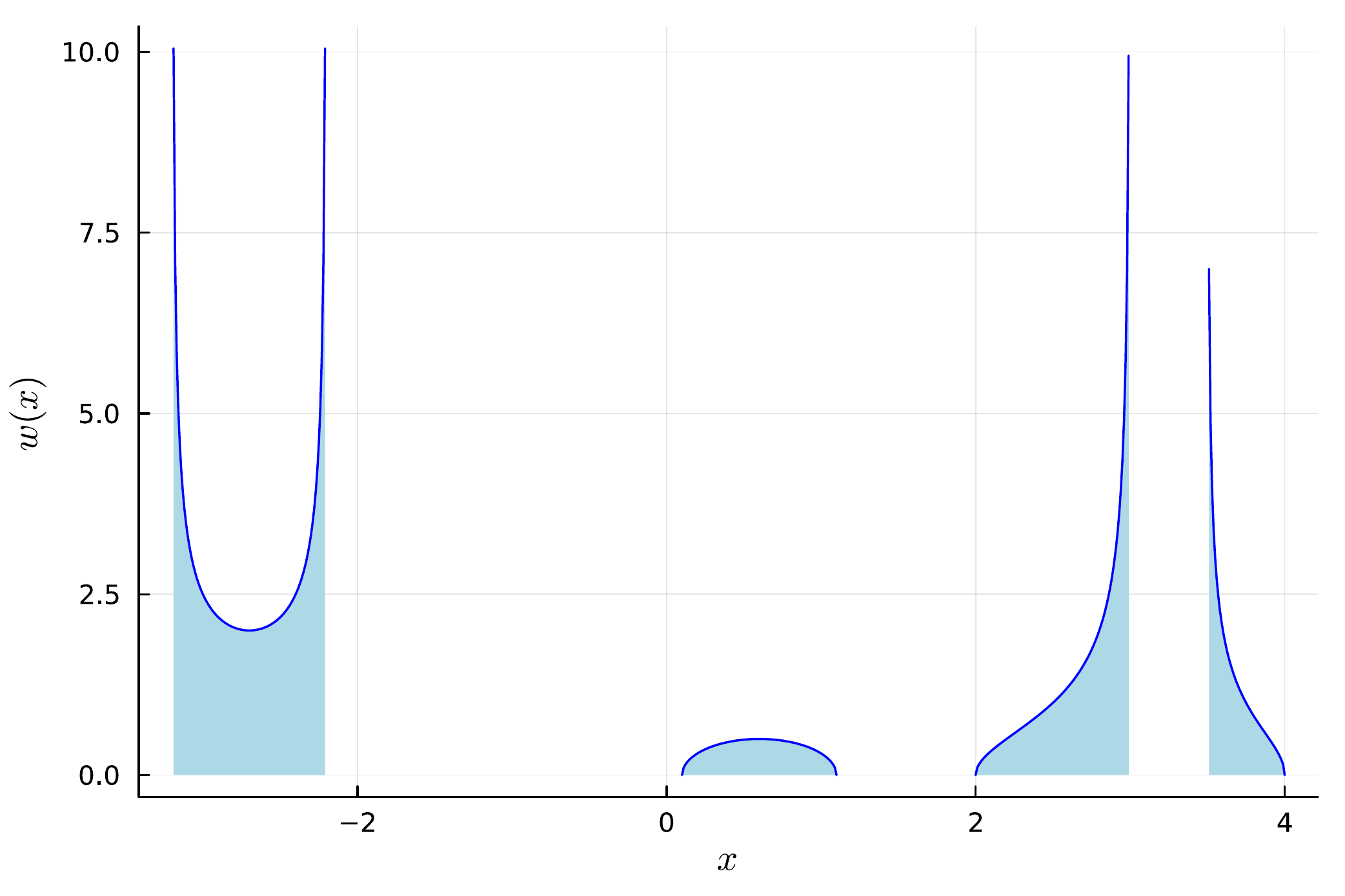}
\caption{A weight function of the form \eqref{weight} with support on $[-3.2, -2.2]\cup[0.1, 1.1]\cup[2, 3]\cup[3.5, 4]$ with $\alpha_1=\beta_1=-1$ (Chebyshev-$\TT$), $\alpha_2=\beta_2=1$ (Chebyshev-$U$), $\alpha_3=1,\beta_3=-1$ (Chebyshev-$V$), $\alpha_4=-1,\beta_4=1$ (Chebyshev-$W$).}
\label{weightfunc}
\end{figure}
In most cases, it will suffice to consider $h_j(x)=1$ for all $j$, and recent work \cite{Gutleb2023} has demonstrated that recurrence coefficients corresponding to a different choice of $h_j$ can be obtained from those corresponding to $h_j(x)=1$.  In principle, our weight function \eqref{weightfunc} could be generalized to a Jacobi-like weight by allowing for any $\alpha_j,\beta_j>-2$; however this would complicate our numerics significantly, as it would result in a Riemann--Hilbert problem defined on self-intersecting contours. Furthermore, the asymptotics are different if $\alpha_j,\beta_j$ are not integers, with error terms behaving polynomially rather than exponentially \cite{KUIJLAARS2004337}. Thus, while our method could, in principle, be extended to an $\OO(N)$ method for Jacobi-like weights, this method would likely have a much larger leading constant than the method that we consider.

In Section \ref{OPCI}, we introduce notation for orthogonal polynomials and Cauchy integrals. In Section \ref{theory}, we apply a portion of the method of nonlinear steepest descent to the FIK Riemann--Hilbert problem with weight \eqref{weight} and introduce an additional transformation so that a numerical solution can be computed. In Section \ref{numerics}, we develop numerical methods to solve the FIK Riemann--Hilbert problem and recover recurrence coefficients. These methods encode the singularity behavior of the weight \eqref{weight}, enabling much faster computations than the methods found in \cite{tomjacobi}. For example, our method computes the first 51 recurrence coefficients for a Chebyshev-$\TT$-like weight ($\alpha_j=\beta_j=-1$) on the domain $[-1.8,-1]\cup[2,3]$ to machine precision in 0.47 seconds\footnote{All computations in this paper are performed on a Lenovo laptop running Ubuntu version 20.04 with 8 cores and 16 GB of RAM with an Intel\textregistered{} Core\texttrademark{} i7-11800H processor running at 2.30 GHz.}. A further timing comparison with an optimized $\OO(N^2)$ method \cite{Gragg1984} can be found in Section \ref{sect:multi}. In Section \ref{examp}, we demonstrate the numerical accuracy of our method and apply our method to evolving a Toda lattice and approximating $1/x$ on a disconnected domain. Code used to generate the plots in this paper can be found at \cite{ballew_trogdon_2023}.

Our method could also potentially be extended to compute the associated $N$-point Gauss quadrature nodes and weights for \eqref{weight} in $\OO(N)$ arithmetic operations using techniques similar to \cite{Hale2012,Townsend2014,Bogaert2014,Huybrechs2016}.  Many extra complexities are introduced for \eqref{weight}, including the possibility of orthogonal polynomials having roots outside the support of the weight.  For this approach to have any advantage over using the individual rules for each subinterval $[a_j,b_j]$, it is likely that $g$ will need to be large.

\section{Orthogonal polynomials and Cauchy integrals}\label{OPCI}
\subsection{Orthogonal polynomials}
For our purposes, a weight function $w$ is a nonnegative function defined on $\Sigma\subset\real$ that is continuous and positive on the interior of $\Sigma$. Consider a sequence of univariate monic polynomials $\pi_0(x),\pi_1(x),\pi_2(x),\ldots$ such that $\pi_j$ has degree $j$ for all $j\in\mathbb{N}$. These polynomials are said to be orthogonal with respect to a weight function $w$ if
\[
\int_{\Sigma}\pi_j(x)\pi_k(x)w(x)\df x=h_j\delta_{jk},
\]
where $h_j>0$ and $\delta_{jk}$ is the Kronecker delta.

The orthonormal polynomials $p_0(x),p_1(x),p_2(x),\ldots$ are defined such that
\[
p_j(x)=\gamma_j\pi_j(x),
\]
for all $j\in\mathbb{N}$ where $\gamma_j=\frac{1}{\sqrt{h_j}}$. Clearly, the orthonormal polynomials satisfy
\[
\int_{\Sigma}p_j(x)p_k(x)w(x)\df x=\delta_{jk}.
\]
Additionally, the orthonormal polynomials satisfy a symmetric three-term recurrence
\begin{equation}\label{recurr}
\begin{aligned}
	&xp_0(x)=a_0p_0(x)+b_0p_1(x),\\
	&xp_k(x)=b_{k-1}p_{k-1}(x)+a_kp_k(x)+b_kp_{k+1}(x),\quad k\geq1,
\end{aligned}
\end{equation}
where $b_k>0$ for all $k$. The Jacobi operator associated with a sequence of orthonormal polynomials is defined as the infinite tridiagonal matrix
\[
\bJ(w)=\begin{pmatrix}
	a_0 &b_0\\
	b_0 &a_1 &b_1\\
	&b_1 &a_2 &b_2\\
	&&\ddots &\ddots &\ddots
\end{pmatrix}.
\]
If we let 
\[
\bp(x)=\begin{pmatrix}
	p_0(x) &p_1(x) &p_2(x) &\cdots
\end{pmatrix},
\]
the three-term recurrence can be expressed as
\[
\bp(x)^Tx=\bJ(w)\bp(x)^T.
\]
Chebyshev polynomials defined on $\mathbb{I}=[-1,1]$ will be particularly useful.  A general reference is  \cite{Mason2002}. Normalized Chebyshev polynomials of the first kind are denoted by $\TT_k(x)$ and have weight function 
\[
w_\TT(x)=\frac{1}{\pi\sqrt{1-x}\sqrt{x+1}},
\]
and Jacobi operator
\[
\bJ(w_\TT)=\begin{pmatrix}
	0 &1/\sqrt{2}\\
	1/\sqrt{2} &0 &1/2\\
	&1/2 &0 &1/2\\
	&&\ddots &\ddots &\ddots
\end{pmatrix}.
\]
We will always use the square root symbol to denote the principal branch defined on $\compl\setminus(-\infty,0)$. Thus, $w_\TT$ is analytic on $\compl\setminus\left(\real\setminus(-1,1)\right).$ Note that our convention relates to the classical Chebyshev-$T$ polynomials by $\TT_0(x)=T_0(x)$, and $\TT_j(x)=\sqrt{2}T_j(x)$ for $j>0$.

Chebyshev polynomials of the second kind on $\mathbb{I}$ are denoted by $U_k(x)$ and have weight function 
\[
w_U(x)=\frac{2}{\pi}\sqrt{1-x}\sqrt{x+1},
\]
and Jacobi operator
\[
\bJ(w_U)=\begin{pmatrix}
	0 &1/2\\
	1/2 &0 &1/2\\
	&1/2 &0 &1/2\\
	&&\ddots &\ddots &\ddots
\end{pmatrix}.
\]

Chebyshev polynomials of the third kind on $\mathbb{I}$ are denoted by $V_k(x)$ and have weight function 
\[
w_V(x)=\frac{1}{\pi}\frac{\sqrt{x+1}}{\sqrt{1-x}},
\]
and Jacobi operator
\[
\bJ(w_V)=\begin{pmatrix}
	1/2 &1/2\\
	1/2 &0 &1/2\\
	&1/2 &0 &1/2\\
	&&\ddots &\ddots &\ddots
\end{pmatrix}.
\]

Chebyshev polynomials of the fourth kind on $\mathbb{I}$ are denoted by $W_k(x)$ and have weight function
\[
w_W(x)=\frac{1}{\pi}\frac{\sqrt{1-x}}{\sqrt{x+1}},
\]
and Jacobi operator
\[
\bJ(w_W)=\begin{pmatrix}
	-1/2 &1/2\\
	1/2 &0 &1/2\\
	&1/2 &0 &1/2\\
	&&\ddots &\ddots &\ddots
\end{pmatrix}.
\]

\subsection{Cauchy integrals}
Given a curve $\Gamma\subset\mathbb{C}$ and a function $f:\Gamma\to\mathbb{C}$, the Cauchy transform $\mathcal{C}_\Gamma$ is as an operator that maps $f$ to its Cauchy integral, i.e.,
\[
\mathcal{C}_\Gamma f(z)=\frac{1}{2\pi \im}\int_\Gamma\frac{f(s)}{s-z}\df s.
\]
If $\Gamma$ is a smooth, non-self-intersecting, oriented contour, then to each interior point $z\in\Gamma$, we associate a ball $B=B(z)$ such that $B\setminus\Gamma$ has two connected components $B_\pm$ where $B_+$ ($B_-$) is to the left (right) with respect to the orientation of $\Gamma$. For $F:\compl\setminus\Gamma\to\compl$, define the (non-tangential) boundary values by 
\[
F^\pm(z)=\lim_{\substack{z'\to z,\\z\in B_\pm}}F(z'),
\]
where this limit exists. We define the notation for the Cauchy operators
\[
\mathcal{C}_\Gamma^\pm f(z)=\lim_{\substack{z'\to z,\\z\in B_\pm}}\mathcal{C}_\Gamma f(z').
\]
Cauchy integrals of orthogonal polynomials will also be of use in this work. For orthonormal polynomials $p_k(x)$ associated with weight $w$ on $\Sigma\subset\real$, define $$C_k(z)=\mathcal{C}_\Sigma[p_kw](z)=\frac{1}{2\pi \im}\int_\Sigma\frac{p_k(s)w(s)}{s-z}\df s,$$ for $z\in\mathbb{C}\setminus \Sigma$.
It is well-known that if the three-term recurrence for $p_k$ is given by \eqref{recurr}, these Cauchy integrals also satisfy the three-term recurrence \cite{olver_slevinsky_townsend_2020}:
\begin{equation}\label{intrecurr}
	\begin{aligned}
		&zC_0(z)=a_0C_0(z)+b_0C_1(z)-\frac{1}{2\pi \im}\int_\Sigma w(x)\df x,\\
		&zC_k(z)=b_{k-1}C_{k-1}(z)+a_kC_k(z)+b_kC_{k+1}(z),\quad k\geq1.
	\end{aligned}
\end{equation}

\section{The Riemann--Hilbert problem}\label{theory}
\subsection{Setting up the Riemann--Hilbert Problem}
The original formulation of orthogonal polynomials as a matrix Riemann--Hilbert problem is due to Fokas, Its, and Kitaev and is for orthogonal polynomials with $\Sigma=\real$ \cite{Fokas1992}. Let $\pi_0,\pi_1,\ldots$ be the sequence of monic orthogonal polynomials for weight function $w$. Consider the Riemann--Hilbert problem for the unknown $\bY_n:\compl\setminus\real\to\compl^{2\times2}$
\begin{equation}\label{FIK}
\begin{aligned}
	&\bY_n~\text{is analytic in}~\compl\setminus\real,\\
	&\bY_n^+(z)=\bY_n^-(z)\begin{pmatrix}
		1 &w(z)\\0 &1
	\end{pmatrix},\quad z\in\real,\\
	&\bY_n(z)=\left(\bI+\OO(z^{-1})\right)\begin{pmatrix}
		z^n &0\\0&z^{-n}
	\end{pmatrix},\quad z\to\infty.
\end{aligned} 
\end{equation}
This problem has a unique solution that is given by \cite{deift_2000}
\begin{equation}\label{exactsol}
	\bY_n(z)=\begin{pmatrix}
		\pi_n(z)&\mathcal{C}_\Sigma[\pi_nw](z)\\
		-2\pi \im\gamma^2_{n-1}\pi_{n-1}(z) &-2\pi \im\gamma^2_{n-1}\mathcal{C}_\Sigma[\pi_{n-1}w](z)
	\end{pmatrix}.
\end{equation}
Then, \eqref{FIK} can be modified for orthogonal polynomials defined on a bounded set $\Sigma\subset\real$ by including asymptotic conditions at the boundary of $\Sigma$ \cite{KUIJLAARS2004337}. We specifically consider orthogonal polynomials for a weight function \eqref{weight} as discussed above.
Recall that $\alpha_j,\beta_j$ determine which Chebyshev weight function variant $w$ behaves like on $[a_j,b_j]$.

The Riemann--Hilbert problem for this weight function is given by
\begin{equation}\label{RHPOG}
	\begin{aligned}
		&\bY_n~\text{is analytic in}~\compl\setminus \Sigma,\\
		&\bY_n^+(z)=\bY_n^-(z)\begin{pmatrix}
			1 &w(z)\\0 &1
		\end{pmatrix},\quad z\in \Sigma,\\
		&\bY_n(z)=\left(\bI+\OO(z^{-1})\right)\begin{pmatrix}
			z^n &0\\0&z^{-n}
		\end{pmatrix},\quad z\to\infty,\\
	&\bY_n(z)=\OO\begin{pmatrix}
		1 &1+|z-a_j|^{\alpha_j/2}\\
		1 &1+|z-a_j|^{\alpha_j/2}
	\end{pmatrix},\quad z\to a_j,\\
	&\bY_n(z)=\OO\begin{pmatrix}
		1 &1+|z-b_j|^{\beta_j/2}\\
		1 &1+|z-b_j|^{\beta_j/2}
	\end{pmatrix},\quad z\to b_j,
	\end{aligned} 
\end{equation}
where the two asymptotic statements are taken element-wise. The unique solution is still given by \eqref{exactsol}. For computational ease, we do not require the weight function $w$ to be normalized.
\subsection{Transforming the Riemann--Hilbert problem}
We transform this problem into one that can be solved numerically through the following procedure:
\begin{itemize}
	\item{\underline{Step 1}: Lens the Riemann--Hilbert problem, factoring and moving jump conditions to regions where we will induce exponential decay.}
	\item{\underline{Step 2}: Determine and apply a differential to remove singularities at infinity.}
	\item{\underline{Step 3}: Determine and apply an auxiliary function to remove the jump conditions on the gaps induced by the differential, and ``solve'' the resulting problem.}
\end{itemize}
This procedure is performed in full for orthogonal polynomials with support on a single interval in \cite{KUIJLAARS2004337,Kuijlaars2003} and for multiple intervals in \cite{Ding2022,YATTSELEV201573}; the novelty of our approach lies in how Step 3 is applied and the fact that ``solve'' means constructing a reliable numerical approximation, not an asymptotic approximation. 
\subsubsection{Step 1: Lensing}
Consider
\begin{equation*}
	w_j(x)=h_j(x)\left(\sqrt{x-a_j}\right)^{\alpha_j}\left(\sqrt{b_j-x}\right)^{\beta_j}.
\end{equation*}
This function is analytic on $\Omega_j\setminus(\real\setminus(a_j,b_j))$. Let $C_j$ be a counterclockwise-oriented curve lying in $\Omega_j$ that encircles, but does not intersect $[a_j,b_j]$, and denote its interior by $D_j$. In the end, our numerical method will take $C_j$ to be a circle, imposing additional restrictions on $\Omega_j$. We include examples of these regions in Figure \ref{regions}.
\begin{figure}
	\resizebox{\textwidth}{!}{%
		\begin{tikzpicture}
			\fill[gray!50] (-6.5,0) circle (2.5 and 2.1) ;
			\node[anchor=south] at (-6.5,0.5) {$D_1$};
			\node[anchor=east] at (-8.4,0) {$C_1$};
			\draw[->] (-8.8,0.25) arc (180:80:0.3);
			\node[anchor=south] at (-6.5,2) {$\Omega_1$};
			\fill[gray!50] (0,0) circle (3.5 and 2.6);
			\node[anchor=south] at (0,0.5) {$D_2$};
			\node[anchor=east] at (-2.5,0) {$C_2$};
			\draw[->] (-2.9,0.25) arc (180:80:0.3 );
			\node[anchor=south] at (0,2.5) {$\Omega_2$};
			\fill[gray!50] (7,0) circle (2 and 1.6);
			\node[anchor=south] at (7,0.25) {$D_3$};
			\node[anchor=east] at (5.6,0) {$C_3$};
			\draw[->] (5.2,0.25) arc (180:80:0.3 );
			\node[anchor=south] at (7,1.5) {$\Omega_3$};
			\draw (-8,0) -- (-5,0);
			\draw (-2,0) -- (2,0);
			\draw (6,0) -- (8,0);
			\filldraw[black] (-8,0) circle (1pt) node[anchor=north]{$a_1$};
			\filldraw[black] (-5,0) circle (1pt) node[anchor=north]{$b_1$};
			\filldraw[black] (-2,0) circle (1pt) node[anchor=north]{$a_2$};
			\filldraw[black] (2,0) circle (1pt) node[anchor=north]{$b_2$};
			\filldraw[black] (6,0) circle (1pt) node[anchor=north]{$a_3$};
			\filldraw[black] (8,0) circle (1pt) node[anchor=north]{$b_3$};
			\draw (-6.5,0) circle (2);
			\draw (0,0) circle (2.5);
			\draw (7,0) circle (1.5);
		\end{tikzpicture}
	}%
\caption{Sample regions $\Omega_j$ and $D_j$ and curves $C_j$ for a three interval case.}
\label{regions}
\end{figure}
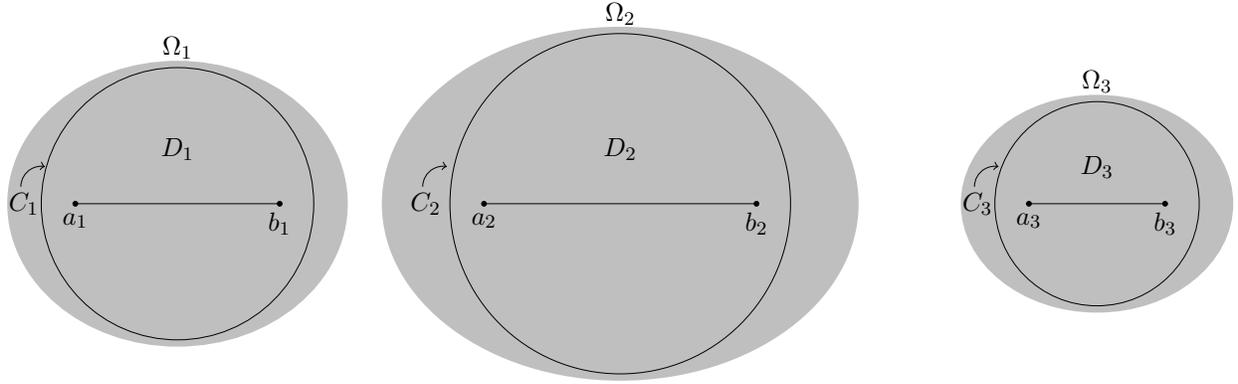
Define the transformation
\begin{equation*}
\bZ_n(z)=\begin{cases}
	\bY_n(z)\begin{pmatrix}
		1 &0\\-1/w_j(z) &1
	\end{pmatrix}  &z\in D_j\cap\compl^+,\\[15 pt]
\bY_n(z)\begin{pmatrix}
	1 &0\\1/w_j(z) &1
\end{pmatrix} &z\in D_j\cap\compl^-,\\
\bY_n(z) &\text{otherwise},
\end{cases}
\end{equation*}
for $j=1,\ldots,g+1$. It follows that for $z \in C_j\cap\compl^+$,
\begin{equation*}
\bZ_n^+(z)=\bZ_n^-(z)\begin{pmatrix}
	1 &0\\-1/w_j(z) &1
\end{pmatrix},
\end{equation*}
and for $z \in C_j\cap\compl^-$, 
\begin{equation*}
	\bZ_n^+(z)=\bZ_n^-(z)\begin{pmatrix}
		1 &0\\1/w_j(z) &1
	\end{pmatrix}.
\end{equation*}
For $z \in (a_j,b_j)$, 
\begin{equation*}
\begin{aligned}
\bZ_n^+(z)&=\bY_n^+(z)\begin{pmatrix}
	1 &0\\-1/w_j(z) &1
\end{pmatrix}=\bY_n^-(z)\begin{pmatrix}
1 &w_j(z)\\0 &1
\end{pmatrix}\begin{pmatrix}
1 &0\\-1/w_j(z) &1
\end{pmatrix}\\&=
\bZ_n^-(z)\begin{pmatrix}
1 &0\\-1/w_j(z) &1
\end{pmatrix}\begin{pmatrix}
1 &w_j(z)\\0 &1
\end{pmatrix}\begin{pmatrix}
1 &0\\-1/w_j(z) &1
\end{pmatrix}=\bZ_n^-(z)\begin{pmatrix}
0 &w_j(z)\\-1/w_j(z) &0
\end{pmatrix}.
\end{aligned}
\end{equation*}
Now, we compute the jump condition on $D_j\cap(\real\setminus(a_j,b_j))$. For $x>b_j$ or $x<a_j$, $w_j^+(x)=-w_j^-(x)$ where now the $+$ denotes a limit from the upper half plane, and $-$ denotes a limit from the lower. Then, 
\begin{equation*}
\begin{aligned}
	\bZ_n^+(z)&=\bY_n^+(z)\begin{pmatrix}
		1 &0\\-1/w^+_j(z) &1
	\end{pmatrix}=\bY_n^-(z)\begin{pmatrix}
		1 &0\\-1/w^+_j(z) &1
	\end{pmatrix}\\&=
	\bZ_n^-(z)\begin{pmatrix}
		1 &0\\-1/w^+_j(z) &1
	\end{pmatrix}\begin{pmatrix}
		1 &0\\-1/w^-_j(z) &1
	\end{pmatrix}=\bZ_n^-(z)\begin{pmatrix}
		1 &0\\-\left(\frac{1}{w^+_j(z)}+\frac{1}{w^-_j(z)}\right) &1
	\end{pmatrix}=\bZ_n^-(z).
\end{aligned}
\end{equation*} 
We find that there is no additional jump condition induced between $[a_j,b_j]$ and $C_j$. This is a special situation in which no local parametrices are required to compute the asymptotics of the polynomials. The asymptotic behavior at the endpoints is multiplied by $\pm 1/w_j(z)$ in this transformation, and the lensed Riemann--Hilbert problem is given by
\begin{equation}\label{RHPlensed}
	\begin{aligned}
		&\bZ_n~\text{is analytic in}~\compl\setminus \left(\bigcup_{j=1}^{g+1}C_j\cup[a_j,b_j]\right),\\
		&\bZ_n^+(z)=\bZ_n^-(z)\begin{pmatrix}
			1 &0\\-1/w_j(z) &1
		\end{pmatrix},\quad z\in C_j\cap\compl^+,\\\
		&\bZ_n^+(z)=\bZ_n^-(z)\begin{pmatrix}
			1 &0\\1/w_j(z) &1
		\end{pmatrix},\quad z\in 	C_j\cap\compl^-,\\
		&\bZ_n^+(z)=\bZ_n^-(z)\begin{pmatrix}
			0 &w_j(z)\\-1/w_j(z) &0
		\end{pmatrix},\quad z\in(a_j,b_j),\\
		&\bZ_n(z)=\left(\bI+\OO(z^{-1})\right)\begin{pmatrix}
			z^n &0\\0&z^{-n}
		\end{pmatrix},\quad z\to\infty,\\
		&\bZ_n(z)=\OO\begin{pmatrix}
		1+|z-a_j|^{-\alpha_j/2} &1+|z-a_j|^{\alpha_j/2}\\
		1+|z-a_j|^{-\alpha_j/2} &1+|z-a_j|^{\alpha_j/2}
		\end{pmatrix},\quad z\to a_j,\\
		&\bZ_n(z)=\OO\begin{pmatrix}
			1+|z-b_j|^{-\beta_j/2} &1+|z-b_j|^{\beta_j/2}\\
			1+|z-b_j|^{-\beta_j/2} &1+|z-b_j|^{\beta_j/2}
		\end{pmatrix},\quad z\to b_j.
	\end{aligned} 
\end{equation}
\subsubsection{Step 2: The differential $\mathfrak g'$} To remove the growth/decay at infinity,
we seek a function $\mathfrak g$ that satisfies the following:
\begin{itemize}
	\item{$\mathfrak g'(z)=z^{-1}+\OO(z^{-2})$ as $z\to\infty$}.
	\item{$\mathfrak g'^+(z),\mathfrak g'^-(z)$ are purely imaginary on $[a_j,b_j]$ for $j=1,\ldots,g+1$.}
	\item{$\int_{b_j}^{a_{j+1}}\mathfrak g'(z)\df z=0$ for $j=1,\ldots,g$.}
\end{itemize}
This is referred to as the exterior Green's function with pole at infinity. We achieve this by defining 
\begin{equation*}
	\mathfrak{g}'(z)=\frac{Q_g(z)}{R(z)},\quad\text{where}~R(z)^2=\prod_{j=1}^{g+1}(z-a_j)(z-b_j).
\end{equation*}
Here, $Q_g$ is a monic polynomial of degree $g$
\begin{equation*}
	Q_g(z)=z^g+\sum_{k=0}^{g-1}h_kz^k,
\end{equation*}
whose coefficients $h_k$ are determined by the linear system 
\begin{equation}\label{gsys}
	\int_{b_j}^{a_{j+1}}\sum_{k=0}^{g-1}h_k\frac{z^k}{R(z)}\df z=-\int_{b_j}^{a_{j+1}}\frac{z^g}{R(z)}\df z,
\end{equation}
for $j=1,\ldots,g$. The unique solvability of this system follows from the fact that $\frac{z^k}{R(z)}$ for $k=0,1,\ldots,g-1$ form a basis for holomorphic differentials on the hyperelliptic Riemann surface defined by $w^2=R(z)^2$. Indeed, the coefficient matrix here is $-\frac{1}{2}\Tilde\bA$ where $\Tilde \bA$ is invertible as in Lemma \ref{invert} below. Then, we define  $\mathfrak{g}$ by normalizing at the leftmost endpoint. If some function satisfies $\psi'=\mathfrak{g}'$, then we let 
\begin{equation*}
	\mathfrak{g}(z)=\psi(z)-\psi(a_1).
\end{equation*}
We also define
\begin{equation}\label{delta}
	\Delta_j=\mathfrak{g}^+(z)-\mathfrak{g}^-(z)=2\sum_{k=1}^{j}\int_{a_k}^{b_k}(\mathfrak{g}')^+(z)\df z\in\im\real,
\end{equation}
for $z\in(b_j,a_{j+1})$, $j=1,\ldots,g$, where this value is independent of $z$. We have that
\begin{equation*}
	\mathfrak{g}^+(z)+\mathfrak{g}^-(z)=0,
\end{equation*}
for $z\in(a_j,b_j)$. Now, we consider the transformation
\begin{equation*}
	\bS_n(z)=\mathfrak{c}^{n\sigma_3}\bZ_n(z)\ex^{-n\mathfrak{g}(z)\sigma_3},\quad \sigma_3=\begin{pmatrix}
		1 &0\\0 &-1
	\end{pmatrix},\quad \mathfrak{c}=\lim_{z\to\infty}\frac{\ex^{\mathfrak g(z)}}{z},
\end{equation*}
where $\sigma_3$ is the third Pauli matrix,
i.e., $\mathfrak g$ has asymptotic expansion
\begin{equation}\label{gasymp}
\mathfrak{g}(z)=\log\mathfrak{c} z+\frac{\mathfrak{g}_1}{z}+\ldots,
\end{equation}
as $z\to\infty$. This asymptotic form follows from the asymptotic expansion of $\mathfrak g'$ at infinity. By construction, $\bS_n(z)=\bI+\OO(z^{-1})$ as $z\to\infty$. Now, the jump condition for $z\in C_j\cap\compl^+$ is given by
\begin{equation*}
\bS_n^+(z)=\bS_n^-(z)\ex^{n\mathfrak{g}(z)\sigma_3}\begin{pmatrix}
	1 &0\\-1/w_j(z) &1
\end{pmatrix}\ex^{-n\mathfrak{g}(z)\sigma_3}=\bS_n^-(z)\begin{pmatrix}
1 &0\\-\ex^{-2n\mathfrak{g}(z)}/w_j(z) &1
\end{pmatrix},
\end{equation*}
and similarly, the jump condition for $z\in C_j\cap\compl^-$ is given by
\begin{equation*}
	\bS_n^+(z)=\bS_n^-(z)\begin{pmatrix}
		1 &0\\\ex^{-2n\mathfrak{g}(z)}/w_j(z) &1
	\end{pmatrix}.
\end{equation*}
For $z\in(a_j,b_j)$, the jump condition is unchanged:
\begin{equation*}
	\begin{aligned}
	\bS_n^+(z)&=\bS_n^-(z)\ex^{n\mathfrak{g}^-(z)\sigma_3}\begin{pmatrix}
		0 &w_j(z)\\-1/w_j(z) &0
	\end{pmatrix}\ex^{-n\mathfrak{g}^+(z)\sigma_3}\\&=
\bS_n^-(z)\begin{pmatrix}
		0 &\ex^{n(\mathfrak{g}^+(z)+\mathfrak{g}^-(z))}w_j(z)\\-\ex^{n(\mathfrak{g}^+(z)+\mathfrak{g}^-(z))}/w_j(z) &1
	\end{pmatrix}=
\bS_n^-(z)\begin{pmatrix}
0 &w_j(z)\\-1/w_j(z) &1
\end{pmatrix}.
\end{aligned}
\end{equation*}
For $z\in(b_j,a_{j+1})$, the jump condition is given by
\begin{equation*}
	\begin{aligned}
		\bS_n^+(z)&=\bS_n^-(z)\ex^{n\mathfrak{g}^-(z)\sigma_3}\ex^{-n\mathfrak{g}^+(z)\sigma_3}\\&=
		\bS_n^-(z)\begin{pmatrix}
			\ex^{-n(\mathfrak{g}^+(z)-\mathfrak{g}^-(z))} &0\\
			0 &\ex^{n(\mathfrak{g}^+(z)-\mathfrak{g}^-(z))}
		\end{pmatrix}=
		\bS_n^-(z)\begin{pmatrix}
			\ex^{-n\Delta_j} &0\\
			0 &\ex^{n\Delta_j}
		\end{pmatrix}.
	\end{aligned}
\end{equation*}
The asymptotic conditions at the endpoints remain unchanged since $|\mathfrak g|$ is bounded in the neighborhood of the contours. We have arrived at the normalized Riemann--Hilbert problem:
\begin{equation}\label{RHPgaps}
\begin{aligned}
	&\bS_n~\text{is analytic in}~\compl\setminus\left([a_1,b_{g+1}]\cup\bigcup_{j=1}^{g+1}C_j\right),\\
	&\bS_n^+(z)=\bS_n^-(z)\begin{pmatrix}
		1 &0\\-\frac{\ex^{-2n\mathfrak{g}(z)}}{w_j(z)} &1
	\end{pmatrix},\quad z\in C_j\cap\compl^+,\\
	&\bS_n^+(z)=\bS_n^-(z)\begin{pmatrix}
		1 &0\\\frac{\ex^{-2n\mathfrak{g}(z)}}{w_j(z)} &1
	\end{pmatrix},\quad z\in C_j\cap\compl^-,\\
	&\bS_n^+(z)=\bS_n^-(z)\begin{pmatrix}
		0 &w_j(z)\\-\frac{1}{w_j(z)} &0
	\end{pmatrix},\quad z\in(a_j,b_j),\\
	&\bS_n^+(z)=\bS_n^-(z)\begin{pmatrix}
		\ex^{-n\Delta_j} &0\\0 &\ex^{n\Delta_j}
	\end{pmatrix},\quad z\in(b_j,a_{j+1}),\\
	&\bS_n(z)=\bI+\OO(z^{-1}),\quad z\to\infty,\\
		&\bS_n(z)=\OO\begin{pmatrix}
	1+|z-a_j|^{-\alpha_j/2} &1+|z-a_j|^{\alpha_j/2}\\
	1+|z-a_j|^{-\alpha_j/2} &1+|z-a_j|^{\alpha_j/2}
\end{pmatrix},\quad z\to a_j,\\
&\bS_n(z)=\OO\begin{pmatrix}
	1+|z-b_j|^{-\beta_j/2} &1+|z-b_j|^{\beta_j/2}\\
	1+|z-b_j|^{-\beta_j/2} &1+|z-b_j|^{\beta_j/2}
\end{pmatrix},\quad z\to b_j.
\end{aligned} 
\end{equation}
Note that $\re \mathfrak g(z)>0$ for $z\in C_j$ (see \cite[Appendix B]{Ding2022}, for example) so the jump condition on $C_j$ tends exponentially to the identity as $n\to\infty$.
\subsubsection{Step 3: Removing the jump conditions on the gaps} Since the normalized Riemann--Hilbert problem \eqref{RHPgaps} contains jump conditions on the gaps between intervals, the jump matrix is discontinuous. In order to appropriately solve this problem numerically, we must further transform it. We do this by devising an auxiliary function to remove the jump conditions on the gaps between intervals.

Consider a function $\mathfrak h_n$ defined as 
\begin{equation}\label{hdef}
\begin{aligned}
	\mathfrak h_n(z)&=R(z)\left(\sum_{j=1}^{g+1}
	\frac{1}{2\pi \im }\int_{a_j}^{b_j}\frac{A_j(n)}{R_+(s)(s-z)}\df s+\sum_{\ell=1}^{g}\int_{b_\ell}^{a_{\ell+1}}\frac{\log(\ex^{n\Delta_\ell})}{R(s)(s-z)}\df s\right)\\&=
R(z)\left(\sum_{j=1}^{g+1}A_j(n)
\mathcal{C}_{[a_j,b_j]}\left[\frac{1}{R_+}\right](z)+\sum_{\ell=1}^{g}\log(\ex^{n\Delta_\ell})\mathcal{C}_{[b_\ell,a_{\ell+1}]}\left[\frac{1}{R}\right](z)\right),
\end{aligned}
\end{equation}
where $A_1(n),\ldots,A_{g+1}(n)$ are undetermined constants. Recall that on the gaps $(b_j,a_{j+1})$, $R$ is analytic. Since $R(z)\sim z^{g+1}$ as $z\to\infty$, we can use the asymptotics of Cauchy integrals to obtain conditions on $A_1(n),\ldots,A_{g+1}(n)$ such that $\mathfrak h_n(z)\sim z^{-1}$ as $z\to\infty$. We need to solve
\begin{equation}\label{hsys}
	\sum_{j=1}^{g+1}
	A_j(n)\int_{a_j}^{b_j}\frac{s^{k-1}}{R_+(s)}\df s=-\sum_{\ell=1}^{g}\log(\ex^{n\Delta_\ell})\int_{b_\ell}^{a_{\ell+1}}\frac{s^{k-1}}{R(s)}\df s,
\end{equation}
for $k=1,\ldots,g+1$ for constants $A_1(n),\ldots,A_{g+1}(n)$. This forms an $(g+1)\times(g+1)$ linear system. Note that $A_1(n),\ldots,A_{g+1}(n)\in\real$ and are bounded above and below due to the fact that $\Delta_j$ is purely imaginary, and the coefficient matrix does not depend on $n$.  Furthermore, this coefficient matrix is always invertible.
\begin{lemma}\label{invert}
The $(g+1)\times(g+1)$ matrix $\bH$ defined by
\begin{equation*}
\bH =\begin{pmatrix}
\int_{a_j}^{b_j}\frac{s^{k-1}}{R_+(s)}\df s
\end{pmatrix}_{1 \leq j,k \leq g+1},
\end{equation*}
is invertible.
\end{lemma}
\begin{proof}
As noted before, $\frac{z^{k-1}}{R(z)}$ for $k=1,\ldots,g$ form a basis for holomorphic differentials on the hyperelliptic Riemann surface defined by $w^2=R(z)^2$. Furthermore, for a differential $\omega$, if we define the cycle $\tilde {\mathfrak b}_j$ by
\begin{align*}
    \oint_{\tilde {\mathfrak b}_j} \omega = 2 \int_{a_j}^{b_j} \omega,
\end{align*}
the classical $b$-periods, denoted by $\mathfrak b_j$, $j = 1,2,\ldots, g$, can be expressed via
\begin{align*}
    \tilde {\mathfrak b}_1 &= {\mathfrak b}_1,\\
    \tilde {\mathfrak b}_j &= {\mathfrak b}_j - {\mathfrak b}_{j-1} , \quad j > 1.
\end{align*}
Define the $a$- and $b$-period matrices $\tilde {\mathbf A}$ and $\tilde {\mathbf B}$, respectively, by
\begin{align*}
    \tilde {\mathbf B} = \begin{pmatrix}
\oint_{\mathfrak b_j} \frac{s^{k-1}}{R(s)}\df s
\end{pmatrix}_{1 \leq j,k \leq g}, \quad \tilde {\mathbf A} = \begin{pmatrix}
\oint_{\mathfrak a_j} \frac{s^{k-1}}{R(s)}\df s
\end{pmatrix}_{1 \leq j,k \leq g}, \quad  \oint_{\mathfrak a_j} \omega = -2 \int_{b_j}^{a_{j+1}} \omega.
\end{align*}
Standard theory implies that $\Tilde \bA$ is invertible.  This implies that if $\bH_g$ is the upper-left $g \times g$ subblock of $\bH$, then
\begin{align*}
    \tilde{\mathbf B} \mathbf L = 2 \bH_g, \quad \mathbf L = \mathbf I - \mathbf N^T,
\end{align*}
where $\mathbf N$ is the nilpotent matrix with ones on its superdiagonal and zeros everywhere else.  Then $4 \pi \im \tilde {\mathbf A}^{-1}  \bH_g \mathbf L^{-1}$ is a Riemann matrix (e.g., it has a negative definite real part).  Since this matrix is real, we conclude that $\bH_g$ is invertible.

Furthermore,
\[
\sum_{j=1}^{g+1}\bH_{kj}=\begin{cases}
    0 &k<g+1,
    \\
    \pi\im &k=g+1.
\end{cases}
\]
This follows from the fact that this summation can be written as 
\[
\frac 1 2 \oint_\gamma\frac{s^{k-1}}{R(s)}\df s,
\]
where $\gamma$ is a clockwise oriented curve encircling $[a_1,b_{g+1}]$. This integral can be evaluated by taking the residue at infinity using $R(z)\sim z^{g+1}$ as $z\to\infty$. Thus, we have that 
\[
\bH\begin{pmatrix}
\bI_g &\bone_g\\
\bzero_g^T &1
\end{pmatrix}=\begin{pmatrix}
\bH_g &\bzero_g\\
\bH_{g+1,1:g}&\pi\im
\end{pmatrix},
\]
where $\bzero_g$ is the zero vector of length $g$, $\bone_g$ is a vector of length $g$ of all ones, and $\bH_{g+1,1:g}$ is the first $g$ columns of the $(g+1)$th row of $\bH$. This new matrix has invertible diagonal blocks, so it must be invertible, and thus $\bH$ is also invertible. 
\end{proof}
By construction, we have that 
\begin{equation*}
	\begin{cases}
		\mathfrak{h}_n^+(z)+\mathfrak{h}_n^-(z)=A_j(n) & z\in(a_j,b_j),\\
		\mathfrak{h}_n^+(z)-\mathfrak{h}_n^-(z)=\log(\ex^{n\Delta_j}) & z\in(b_{j+1},a_j).
	\end{cases}
\end{equation*}
Now, consider the transformation
\begin{equation*}
	\Tilde{\bS}_n(z)=\bS_n(z)\ex^{\sigma_3\mathfrak{h}_n(z)}.
\end{equation*} 
Noting that
if $\bS_n^+(z)=\bS_n^-(z)\bK(z)$ on a contour, then 
\begin{equation*}
	\Tilde{\bS}^+_n(z)=\Tilde{\bS}^-_n(z)\ex^{-\sigma_3\mathfrak{h}_n^-(z)}\bK(z)\ex^{\sigma_3\mathfrak{h}_n^+(z)},
\end{equation*}
we find that for $z\in(a_{j+1},b_j)$, 
\begin{equation*}
	\Tilde{\bS}^+_n(z)=\Tilde{\bS}^-_n(z)\begin{pmatrix}
		\ex^{-\mathfrak h_n^-(z)} &0\\0 &\ex^{\mathfrak h_n^-(z)}
	\end{pmatrix}\begin{pmatrix}
		\ex^{-n\Delta_j} &0\\0 &\ex^{n\Delta_j}
	\end{pmatrix}\begin{pmatrix}
		\ex^{\mathfrak h_n^+(z)} &0\\0 &\ex^{-\mathfrak h_n^+(z)}
	\end{pmatrix}=\Tilde{\bS}^-_n(z),
\end{equation*}
meaning that the jump condition on the gaps has been eliminated. Applying our transformation to the remaining jump conditions, we use the fact that the asymptotic conditions at the endpoints are unaffected since $|\mathfrak h_n|$ is bounded by \cite[Lemma 7.2.2]{ablowitz_fokas_2003}. We arrive at the new Riemann--Hilbert problem
\begin{equation}\label{RHPfinal}
\begin{aligned}
	&\Tilde \bS_n~\text{is analytic in}~\compl\setminus\left(\bigcup_{j=1}^{g+1}[a_j,b_j]\cup\bigcup_{j=1}^{g+1}C_j\right),\\
	&\Tilde \bS_n^+(z)=\Tilde \bS_n^-(z)\begin{pmatrix}
		1 &0\\-\frac{1}{w_j(z)}\ex^{2\mathfrak h_n(z)-2n\mathfrak{g}(z)} &1
	\end{pmatrix},\quad z\in C_j\cap\compl^+,\\
	&\Tilde \bS_n^+(z)=\Tilde \bS_n^-(z)\begin{pmatrix}
		1 &0\\\frac{1}{w_j(z)}\ex^{2\mathfrak h_n(z)-2n\mathfrak{g}(z)} &1
	\end{pmatrix},\quad z\in C_j\cap\compl^-,\\
	&\Tilde \bS_n^+(z)=\Tilde \bS_n^-(z)\begin{pmatrix}
		0 &w_j(z)\ex^{-A_j(n)}\\-\frac{1}{w_j(z)}\ex^{A_j(n)} &0
	\end{pmatrix},\quad z\in(a_j,b_j),\\
	&\Tilde \bS_n(z)=\bI+\OO(z^{-1}),\quad z\to\infty,\\
&\Tilde\bS_n(z)=\OO\begin{pmatrix}
	1+|z-a_j|^{-\alpha_j/2} &1+|z-a_j|^{\alpha_j/2}\\
	1+|z-a_j|^{-\alpha_j/2} &1+|z-a_j|^{\alpha_j/2}
\end{pmatrix},\quad z\to a_j,\\
&\Tilde\bS_n(z)=\OO\begin{pmatrix}
	1+|z-b_j|^{-\beta_j/2} &1+|z-b_j|^{\beta_j/2}\\
	1+|z-b_j|^{-\beta_j/2} &1+|z-b_j|^{\beta_j/2}
\end{pmatrix},\quad z\to b_j,
\end{aligned}
\end{equation}
that we will solve numerically. As shorthand, we define
\begin{equation}\label{shorthand}
\begin{split}
&F^1_j(z)=\begin{pmatrix}
	1 &0\\-\frac{1}{w_j(z)}\ex^{2\mathfrak h_n(z)-2n\mathfrak{g}(z)} &1
\end{pmatrix},\quad F^2_j(z)=\begin{pmatrix}
0 &w_j(z)\ex^{-A_j(n)}\\-\frac{1}{w_j(z)}\ex^{A_j(n)} &0
\end{pmatrix}, \\ &F^3_j(z)=\begin{pmatrix}
1 &0\\\frac{1}{w_j(z)}\ex^{2\mathfrak h_n(z)-2n\mathfrak{g}(z)} &1
\end{pmatrix}.
\end{split}
\end{equation}
We include example plots of the curves and jump conditions in this problem in Figure \ref{jumpfig}.\\
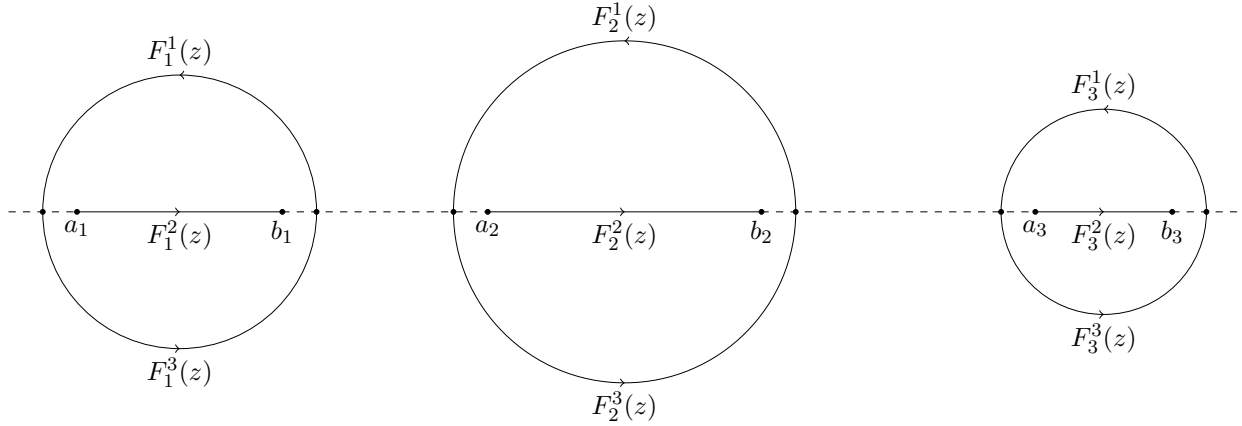
\begin{figure}
	\resizebox{\textwidth}{!}{%
		\begin{tikzpicture}
			\draw (-8,0) -- (-5,0);
			\draw (-2,0) -- (2,0);
			\draw (6,0) -- (8,0);
			\draw [dashed] (-9,0) -- (-8,0); 
			\draw [dashed] (-5,0) -- (-2,0);
			\draw [dashed] (2,0) -- (6,0);
			\draw [dashed] (8,0) -- (9,0);
			\filldraw[black] (-8,0) circle (1pt) node[anchor=north]{$a_1$};
			\filldraw[black] (-5,0) circle (1pt) node[anchor=north]{$b_1$};
			\filldraw[black] (-2,0) circle (1pt) node[anchor=north]{$a_2$};
			\filldraw[black] (2,0) circle (1pt) node[anchor=north]{$b_2$};
			\filldraw[black] (6,0) circle (1pt) node[anchor=north]{$a_3$};
			\filldraw[black] (8,0) circle (1pt) node[anchor=north]{$b_3$};
			\filldraw[black] (-8.5,0) circle (1pt);
			\filldraw[black] (-4.5,0) circle (1pt);
			\filldraw[black] (-2.5,0) circle (1pt);
			\filldraw[black] (2.5,0) circle (1pt);
			\filldraw[black] (5.5,0) circle (1pt);
			\filldraw[black] (8.5,0) circle (1pt);
			
			\draw (-6.5,0) circle (2);
			\path[tips,->] (-8,0) -- (-6.5,0);
			\node[anchor=north] at (-6.5,0) {$F_1^2(z)$};
			\path[tips,->] (-4.5,0) arc (0:90:2 );
			\node[anchor=south] at (-6.5,2) {$F_1^1(z)$};
			\path[tips,->] (-8.5,0) arc (180:270:2);
			\node[anchor=north] at (-6.5,-2) {$F_1^3(z)$};
			
			\draw (0,0) circle (2.5);
			\path[tips,->] (-2,0) -- (0,0);
			\node[anchor=north] at (0,0) {$F_2^2(z)$};
			\path[tips,->] (2.5,0) arc (0:90:2.5 );
			\node[anchor=south] at (0,2.5) {$F_2^1(z)$};
			\path[tips,->] (-2.5,0) arc (180:270:2.5);
			\node[anchor=north] at (0,-2.5) {$F_2^3(z)$};
			
			\draw (7,0) circle (1.5);
			\path[tips,->] (6,0) -- (7,0);
			\node[anchor=north] at (7,0) {$F_3^2(z)$};
			\path[tips,->] (8.5,0) arc (0:90:1.5);
			\node[anchor=south] at (7,1.5) {$F_3^1(z)$};
			\path[tips,->] (5.5,0) arc (180:270:1.5);
			\node[anchor=north] at (7,-1.5) {$F_3^3(z)$};
			
		\end{tikzpicture}
	}%
\caption{Jump conditions in a three interval case. $F^k_j(z)$ are as defined in \eqref{shorthand}.}
\label{jumpfig}
\end{figure}

\subsection{Recurrence coefficients}\label{sect:recurr}
Recurrence coefficients for the orthogonal polynomials corresponding to the weight function \eqref{weight} are directly related to the first-order asymptotic behavior (i.e., the residue at infinity) of the unique solution \eqref{exactsol} to the original problem \eqref{RHPOG}. If \begin{equation*}
	\bY_n(z)=\left(\bI+z^{-1}\bY_n^{(1)}+\OO(z^{-2})\right)\begin{pmatrix}
		z^n &0\\0&z^{-n}
	\end{pmatrix},
\end{equation*} 
as $z\to\infty$, then the recurrence coefficients \eqref{recurr} are given by
\begin{equation}\label{recurrinit}
	\begin{aligned}
		&a_n=\left(\bY_n^{(1)}\right)_{11}-\left(\bY_{n+1}^{(1)}\right)_{11},\\
		&b_n=\sqrt{\left(\bY_{n+1}^{(1)}\right)_{12}\left(\bY_{n+1}^{(1)}\right)_{21}},
	\end{aligned}
\end{equation}
for $n=0,1,2,\ldots$ \cite{deift_2000}. We now use this to derive equivalent formulae based on the first-order asymptotic behavior of \eqref{RHPfinal} by working backwards through the transformations. Define $\mathfrak h_n^{(1)}$ by
\begin{equation}\label{hasymp}
	\mathfrak h_n(z)=\frac{\mathfrak h^{(1)}_{n}}{z}+\OO(z^{-2}),\quad z\to\infty.
\end{equation}
Then, 
\begin{equation*}
	\ex^{\mathfrak{h}_n(z)}=1+\frac{\mathfrak h_n^{(1)}}{z}+\OO(z^{-2}),\quad z\to\infty.
\end{equation*}
From this, if 
\begin{equation*}
	\begin{aligned}
	&\bS_n(z)=\bI+z^{-1} \bS_{n}^{(1)}+\OO(z^{-2}),\\
	&\Tilde \bS_n(z)=\bI+z^{-1}\Tilde \bS_{n}^{(1)}+\OO(z^{-2}),
	\end{aligned}
\end{equation*}
as $z\to\infty$, then 
\begin{align*}
	\bS_n(z)&=\Tilde \bS_n(z)\ex^{-\mathfrak{h}_n(z)\sigma_3}=\left(\bI+z^{-1}\Tilde \bS_{n}^{(1)}+\OO(z^{-2})\right)\left(\bI-z^{-1}\mathfrak h_n^{(1)}\sigma_3+\OO(z^{-2})\right)\\&=
	\left(\bI+z^{-1}\left(\Tilde \bS_{n}^{(1)}-\mathfrak h_n^{(1)}\sigma_3\right)+\OO(z^{-2})\right),
\end{align*}
implying $\bS_{n}^{(1)}=\Tilde\bS_{n}^{(1)}-\mathfrak h_n^{(1)}\sigma_3$.

Now, we obtain the first-order asymptotics of $\bZ_n$ \eqref{RHPlensed} from $\bS_{n}$ \eqref{RHPgaps}. Recall that 
\begin{equation*}
	\bZ_n(z)=\mathfrak{c}^{-n\sigma_3}\bS_n(z)\ex^{n\mathfrak{g}(z)\sigma_3}.
\end{equation*}
Then, the asymptotic expansion of $\mathfrak g$ \eqref{gasymp} gives
\begin{equation*}
	 \bZ_n(z)=\mathfrak{c}^{-n\sigma_3}\left(\bI+\bS_n^{(1)}z^{-1}+\OO(z^{-2})\right)\mathfrak{c}^{n\sigma_3}z^n\sigma_3\ex^{n\sigma_3\left(\mathfrak g_1z^{-1}+\OO(z^{-2})\right)}.
\end{equation*}
Expanding
\begin{equation*}
	\ex^{n\sigma_3\left(\mathfrak g_1z^{-1}+\OO(z^{-2})\right)}=\bI+n\mathfrak{g}_1\sigma_3z^{-1}+\OO(z^{-2}),
\end{equation*}
gives
\begin{equation*}
	\bZ_n(z)=\left(\bI+\left(\mathfrak{c}^{-n\sigma_3}\bS_n^{(1)}\mathfrak{c}^{n\sigma_3}+n\mathfrak{g}_1\sigma_3\right)z^{-1}+\OO(z^{-2})\right)\begin{pmatrix}
		z^n&0\\0&z^{-n}
	\end{pmatrix},
\end{equation*}
and therefore
\begin{equation*}
	 \bZ_n^{(1)}=\mathfrak{c}^{-n\sigma_3}\bS_n^{(1)}\mathfrak{c}^{n\sigma_3}+n\mathfrak{g}_1\sigma_3=\begin{pmatrix}
		\left(\bS_n^{(1)}\right)_{11}+n\mathfrak{g}_1 &\mathfrak{c}^{2n}\left(\bS_n^{(1)}\right)_{12}\\
		\mathfrak{c}^{-2n}\left(\bS_n^{(1)}\right)_{21} &\left(\bS_n^{(1)}\right)_{22}-n\mathfrak{g}_1
	\end{pmatrix}.
\end{equation*}
Finally, the lensing of the problem in Step 1 does not affect the asymptotic behavior at infinity, so
\begin{equation*}
\bY_n^{(1)}=\begin{pmatrix}
	\left(\Tilde \bS_n^{(1)}\right)_{11}-\mathfrak h_n^{(1)}+n\mathfrak{g}_1 &\mathfrak{c}^{2n}\left(\Tilde \bS_n^{(1)}\right)_{12}\\
	\mathfrak{c}^{-2n}\left(\Tilde \bS_n^{(1)}\right)_{21} &\left(\Tilde \bS_n^{(1)}\right)_{22}+\mathfrak h_n^{(1)}-n\mathfrak{g}_1
\end{pmatrix},
\end{equation*}
yielding new recurrence formulae
\begin{equation}\label{realrecurr}
	\begin{aligned}
		a_n&=\left(\Tilde \bS_{n}^{(1)}\right)_{11}-\left(\Tilde \bS_{n+1}^{(1)}\right)_{11}-\mathfrak h_n^{(1)}+\mathfrak h_{n+1}^{(1)}-\mathfrak g_1,\\
		b_n&=\sqrt{\left(\Tilde \bS_{n+1}^{(1)}\right)_{12}\left(\Tilde \bS_{n+1}^{(1)}\right)_{21}},
	\end{aligned}
\end{equation}
that we will use to recover the recurrence coefficients when numerically solving the problem \eqref{RHPfinal}.

As $n\to\infty$, the jump conditions on $C_j$ tend exponentially to the identity matrix, so the FIK Riemann--Hilbert problem can be considered with just the jump conditions on the intervals, asymptotically. While formulae for the recurrence coefficients in this asymptotic case are known \cite[Appendix B]{Ding2022}, their evaluation is made difficult due to the presence of theta functions composed with an Abel map. Thus, a numerical method for solving this form of the FIK Riemann--Hilbert problem not only allows for efficient evaluation of these rather complicated formulae, but can potentially be utilized to compute theta functions. Related to this, one could choose the functions $h_j$ so that $w_j(z) = w_k(z)$ for all $j,k$.  In this case, individual circles around each interval $[a_j,b_j]$ could be replaced with a large circle around $[a_1,b_{g+1}]$.  We will explore this further in future work.

\section{Numerical methods}\label{numerics}
\subsection{Numerically computing Cauchy integrals}
We compute orthogonal polynomial coefficients via \eqref{realrecurr} by solving the Riemann--Hilbert problem \eqref{RHPfinal} numerically. Doing so will heavily rely on the computation of Cauchy integrals and, specifically, on computing Cauchy integrals of Chebyshev polynomials on an interval. Recalling that the weight function of Chebyshev-$\TT$ on $\mathbb{I}=[-1,1]$ is given by $w_\TT(x)=\frac{1}{\pi\sqrt{x+1}\sqrt{1-x}}$ and that $\TT_0=1$, 
\[
\mathcal{C}_{\mathbb{I}}[\TT_0w_\TT](z)=\frac{\im}{2\pi\sqrt{z-1}\sqrt{z+1}}.
\]
Similarly, 
\begin{align*}
	&\mathcal{C}_{\mathbb{I}}[U_0w_U](z)=\frac{\im}{\pi}\left(z-\sqrt{z-1}\sqrt{z+1}\right),\\ &\mathcal{C}_{\mathbb{I}}[V_0w_V](z)=\frac{\im}{2\pi}\left(-1+\frac{\sqrt{z+1}}{\sqrt{z-1}}\right),\quad
	\mathcal{C}_{\mathbb{I}}[W_0w_W](z)=\frac{\im}{2\pi}\left(1-\frac{\sqrt{z-1}}{\sqrt{z+1}}\right).
\end{align*}
\begin{theorem}\label{chebints}
	Let $C^\square_k(z)=\mathcal{C}_{\mathbb{I}}[\square_kw_\square](z)$ where $\square$ denotes the kind of Chebyshev polynomial $\TT,U,V,W$. Then, for $k\geq1$,
	\begin{align*}
		&C^\TT_k(z)=\sqrt{2}C_0^\TT(z)J_+^{-1}(z)^k,
		&C^U_k(z)=C_0^U(z) J_+^{-1}(z)^k,\\
		&C^V_k(z)=C_0^V(z) J_+^{-1}(z)^k,
		&C^W_k(z)=C_0^W(z) J_+^{-1}(z)^k,
	\end{align*}
where $J^{-1}_+$ is a right inverse of the Joukowsky map, mapping $\compl\setminus[-1,1]$ to the interior of the unit circle, given by
\[
J^{-1}_+(z)=z-\sqrt{z-1}\sqrt{z+1}.
\] 
\end{theorem}
\begin{proof}
	We proceed by induction by first considering the base case $k=1$ separately for each kind of Chebyshev polynomial. Using the recurrence relation \eqref{intrecurr} that $C^\square_k$ must satisfy, we can directly compute the following.
	\begin{align*}
		C^\TT_1(z)&=\sqrt{2}\left(zC^\TT_0(z)+\frac{1}{2\pi \im}\right)=
		\sqrt{2}\left(z\frac{\im}{2\pi\sqrt{z-1}\sqrt{z+1}}+\frac{1}{2\pi \im}\right)\\&=
		\sqrt{2}\frac{\im}{2\pi}\frac{1}{\sqrt{z-1}\sqrt{z+1}}\left(z-\sqrt{z-1}\sqrt{z+1}\right)=\sqrt{2}C_0^\TT(z)J_+^{-1}(z).
	\end{align*}
	\begin{align*}
		C^U_1(z)&=2\left(zC^U_0(z)+\frac{1}{2\pi \im}\right)=
		\frac{\im}{\pi}\left(z^2-2z\sqrt{z-1}\sqrt{z+1}+(z^2-1)\right)\\&=\frac{\im}{\pi}J^{-1}_+(z)^2=\frac{\im}{\pi}\left(z-\sqrt{z-1}\sqrt{z+1}\right)J^{-1}_+(z)=C_0^U(z)J_+^{-1}(z).
	\end{align*}
	\begin{align*}
		C^V_1(z)&=(2z-1)C^V_0(z)+2\frac{1}{2\pi i}=
		\frac{\im}{2\pi}\left(-(z+1)+z\frac{\sqrt{z+1}}{\sqrt{z-1}}+\sqrt{z+1}\sqrt{z-1}-z\right)\\&=
		\frac{\im}{2\pi}\left(-1+\frac{\sqrt{z+1}}{\sqrt{z-1}}\right)\left(z-\sqrt{z-1}\sqrt{z+1}\right)=C_0^V(z)J_+^{-1}(z).
	\end{align*}
	\begin{align*}
		C^W_1(z)&=(2z+1)C^W_0(z)+2\frac{1}{2\pi i}=
		\frac{\im}{2\pi}\left((z-1)-z\frac{\sqrt{z-1}}{\sqrt{z+1}}-\sqrt{z+1}\sqrt{z-1}+z\right)\\&=
		\frac{\im}{2\pi}\left(1-\frac{\sqrt{z-1}}{\sqrt{z+1}}\right)\left(z-\sqrt{z-1}\sqrt{z+1}\right)=C_0^W(z)J_+^{-1}(z).
	\end{align*}
	Now,
	\begin{align*}
		2zJ^{-1}_+(z)-1&=2z\left(z-\sqrt{z-1}\sqrt{z+1}\right)-1=
		z^2-2z\sqrt{z-1}\sqrt{z+1}+(z^2-1)=J^{-1}_+(z)^2.
	\end{align*}
	Using this for the second base case $k=2$, 
	\begin{align*}
		C^\TT_2(z)&=2zC^\TT_1(z)-\sqrt{2}C_0^\TT(z)=2\sqrt{2}zC^\TT_0(z)J^{-1}_+(z)-\sqrt{2}C_0^\TT(z)\\&=
		\sqrt{2}C_0^\TT(z)\left(2zJ^{-1}_+(z)-1\right)=\sqrt{2}C_0^\TT(z)J^{-1}_+(z)^2.
	\end{align*}
	As an inductive step, we let $k\geq3$ and find that our recurrence gives
	\begin{align*}
		C^\TT_k(z)&=2zC^\TT_{k-1}(z)-C^\TT_{k-2}(z)=2z\sqrt{2}C_0^\TT(z)J^{-1}_+(z)^{k-1}+\sqrt{2}C_0^\TT(z)J^{-1}_+(z)^{k-2}\\&=
		\sqrt{2}C_0^\TT(z)J^{-1}_+(z)^{k-2}\left(2zJ^{-1}_+(z)-1\right)=\sqrt{2}C_0^\TT(z)J^{-1}_+(z)^{k},
	\end{align*}
	completing the proof for Chebyshev-$\TT$. For $U,V$, and $W$, see that the theorem is trivially true for $k=0$, so we can use $k=0,1$ as our base cases. Letting $\square$ denote $U,V$, or $W$, we let $k\geq2$ and find from the recurrence that
	\begin{align*}
		C^\square_k(z)&=2zC^\square_{k-1}(z)-C^\square_{k-2}(z)=2zC^\square_0(z)J^{-1}_+(z)^{k-1}+C^\square_0(z)J^{-1}_+(z)^{k-2}\\&=
		C^\square_0(z)J^{-1}_+(z)^{k-2}\left(2zJ^{-1}_+(z)-1\right)=C^\square_0(z)J^{-1}_+(z)^{k},
	\end{align*}
	completing the proof for all cases.
\end{proof}
These integrals can be scaled and shifted appropriately to compute $\mathcal{C}_{[a,b]}[fw]$ for an arbitrary interval $[a,b]$. If we let $\Tilde{\TT}_k(x;[a,b])=\TT_k(M^{-1}(x;[a,b]))$ with weight function $\Tilde{w}_\TT(x;[a,b])=\frac{1}{\pi\sqrt{x-a}\sqrt{b-x}}$, $\Tilde{U}_k(x;[a,b])=U_k(M^{-1}(x;[a,b]))$ with weight function $\Tilde{w}_U(x;[a,b])=\frac{2}{\pi}\left(\frac{2}{b-a}\right)^2\sqrt{x-a}\sqrt{b-x}$, $\Tilde{V}_k(x;[a,b])=V_k((M^{-1}(x;[a,b])))$ with weight function $\Tilde{w}_V(x;[a,b])=\frac{2}{\pi(b-a)}\frac{\sqrt{x-a}}{\sqrt{b-x}}$, and $\Tilde{W}_k(x;[a,b])=W_k(M^{-1}(x;[a,b]))$ with weight function $\Tilde{w}_W(x;[a,b])=\frac{2}{\pi(b-a)}\frac{\sqrt{b-x}}{\sqrt{x-a}}$ where
\begin{equation*}
	M^{-1}(x;[a,b])=\frac{2}{b-a}\left(x-\frac{b+a}{2}\right),
\end{equation*} then
\begin{align*}
	\mathcal{C}_{[a,b]}[\Tilde \TT_0\Tilde w_T](z)&=\frac{\im}{2\pi\sqrt{z-a}\sqrt{z-b}},\\
	\mathcal{C}_{[a,b]}[\Tilde \TT_k\Tilde w_T](z)&=\frac{\im}{\sqrt{2}\pi\sqrt{z-a}\sqrt{z-b}}J^{-1}_+\left(M^{-1}(z;[a,b])\right)^k,&k\geq1,\\
	\mathcal{C}_{[a,b]}[\Tilde U_k\Tilde w_U](z)&=\frac{\im}{\pi}\frac{2}{b-a}{J^{-1}_+\left(M^{-1}(z;[a,b])\right)}^{k+1},&k\geq0,\\
	\mathcal{C}_{[a,b]}[\Tilde V_k\Tilde w_V](z)&=\frac{\im}{2\pi}\frac{2}{b-a}\left(-1+\frac{\sqrt{z-a}}{\sqrt{z-b}}\right)J^{-1}_+\left(M^{-1}(z;[a,b])\right)^k,&k\geq0,\\
	\mathcal{C}_{[a,b]}[\Tilde W_k\Tilde w_W](z)&=\frac{\im}{2\pi}\frac{2}{b-a}\left(1-\frac{\sqrt{z-b}}{\sqrt{z-a}}\right)J^{-1}_+\left(M^{-1}(z;[a,b])\right)^k,&k\geq0.
\end{align*}
It is important to determine the leading-order asymptotic behavior of these integrals for our later calculations. We have
\begin{equation}\label{chebasymp}
\begin{aligned}
	&\mathcal{C}_{[a,b]}[\Tilde \square_0\Tilde w_\square](z)=\frac{\im}{2\pi}z^{-1}+\OO(z^{-2}),\\
	&\mathcal{C}_{[a,b]}[\Tilde \square_j\Tilde w_\square](z)=\OO(z^{-j-1}),\quad j>0,
\end{aligned}
\end{equation}
where $\square$ denotes any of the four kinds of Chebyshev polynomial as $z\to\infty$. This follows from the above normalizations.

\subsection{Numerically computing $\mathfrak{g}$}
To determine $\mathfrak{g}$, we first obtain the coefficients of $Q_g(z)$ via solving the linear system \eqref{gsys} by numerically computing the relevant integrals via a Chebyshev-$\TT$-based quadrature rule\footnote{We can easily do this by getting the first Chebyshev coefficient of $z^j\sqrt{z-a_k}\sqrt{b_k-x}/R(z)$ in a mapped Chebyshev-$\TT$ series via the discrete cosine transform.}. 
\begin{lemma}[Plemelj Lemma, \cite{Muskhelishvili1977}, Section 17]\label{Plemelj} Let $\Gamma\subset\compl$ and $f$ be H\"older continuous on $\Gamma$, then for $z\in\Gamma$ that is not an endpoint of $\Gamma$,
\[
\mathcal{C}_\Gamma^+f(z)-\mathcal{C}_\Gamma^-f(z)=f(z).
\]
\end{lemma}
Motivated by \cite[Section 5.5]{trogsworldllc}, we use Lemma \ref{Plemelj} to rewrite $\mathfrak{g}'$ as a Cauchy integral. We have that $R^+(z)=-R^-(z)$ for $z\in(a_j,b_j)$, and $Q_g$ is of lower degree than $R$, so $\mathfrak g'$ solves the scalar Riemann--Hilbert problem
\begin{equation*}
\begin{aligned}
&(\mathfrak{g}')^+(z)=(\mathfrak{g}')^-(z)+2\frac{Q_g(z)}{R^+(z)},\quad z\in\bigcup_{j=1}^{g+1}(a_j,b_j),\\
&\mathfrak{g}'(z)=0,\quad z\to\infty.
\end{aligned}
\end{equation*}
Thus, by Lemma \ref{Plemelj},
\begin{equation*}
\mathfrak{g}'(z)=\sum_{j=1}^{g+1}2\mathcal{C}_{[a_j,b_j]}\left[\frac{Q_g}{R^+}\right](z).
\end{equation*}
Of course, this implies that 
\begin{equation*}
	\mathfrak{g}(z)=\int_{a_1}^z\mathfrak{g}'(z')\df z'=\sum_{j=1}^{g+1}2\int_{a_1}^z\mathcal{C}_{[a_j,b_j]}\left[\frac{Q_g}{R^+}\right](z')\df z',
\end{equation*}
where the integration path is taken to be a straight line between $a_1$ and $z$. We approximate $\frac{Q_g(z)}{R^+(z)}\sqrt{z-a_j}\sqrt{z-b_j}$ to high-precision as a Chebyshev-$\TT$ series on $[a_j,b_j]$ using the discrete cosine transform. Letting $w_{\TT,j}(x)=\Tilde w_{\TT}(x;[a_j,b_j])$, $\TT_{j,k}(x)=\Tilde\TT_k(x;[a_j,b_j])$ and dropping the lower bound of integration that will be accounted for later, it thus suffices to consider computing the anti-derivatives
\begin{equation*}
	p_{j,k}(z)=-2\pi \im\alpha_{j,k}\int^z\frac{1}{2\pi \im}\int_{a_j}^{b_j}\frac{\TT_{j,k}(s)w_{T,j}(s)}{s-z'}\df s\df z',
\end{equation*}
where $\alpha_{j,k}$ is the $k$th Chebyshev-$\Tilde\TT$ coefficient of $\frac{Q_g(z)}{R^+(z)}\sqrt{z-a_j}\sqrt{z-b_j}$ on $[a_j,b_j]$. For $k>0$, we swap the order of integration
\begin{equation*}
	p_{j,k}(z)=\alpha_{j,k}\int_{a_j}^{b_j}\TT_{j,k}(s)w_{\TT,j}(s)\log(s-z)\df s.
\end{equation*} 
Let $f(s)=\TT_{j,k}(s)w_{\TT,j}(s)$. Using integration by parts, 
\begin{align*}
	\frac{1}{2\pi \im}\int_{a_j}^{b_j}f(s)\log(s-z)\df s=
	\frac{\log(b_j-z)}{2\pi i}\int_{a_j}^{b_j}f(s)\df s-\mathcal{C}_{[a_j,b_j]}\left[F\right](z),
\end{align*}
where 
\begin{equation*}
	F(s)=\int_{a_j}^{s}f(s')\df s'.
\end{equation*} 
By orthogonality,
\begin{equation*}
	\int_{a_j}^{b_j}f(s)\df s=\int_{a_j}^{b_j}\TT_{j,k}(s)w_{\TT,j}(s)\df s=0,
\end{equation*}
for $k>0$, and\footnote{We use $\diamond$ notation to specify a function without naming it. For example, we use $1/\diamond$ to denote the function $f(x)=1/x.$}
\begin{equation*}
	p_{j,k}(z)=-2\pi \im\alpha_{j,k}\mathcal{C}_{[a_j,b_j]}\left[\int_{a_j}^{\diamond}\TT_{j,k}(s)w_{\TT,j}(s)\df s\right](z).
\end{equation*}
Now, \cite[Eq. 18.9.22]{NIST:DLMF} gives that 
\begin{equation*}
	\frac{\df}{\df x}\left(\sqrt{1-x}\sqrt{x+1}U_{n}(x)\right)=-(n+1)\frac{1}{\sqrt{1-x}\sqrt{x+1}} T_{n+1}(x),
\end{equation*} 
where  $\TT_j(x)=\sqrt{2}T_{j}(x)$ in our convention. Substituting this in and integrating both sides gives
\begin{equation*}
	\int_{-1}^{x}\frac{\TT_k(s)}{\sqrt{1-s}\sqrt{s+1}}\df s=-\frac{\sqrt{2}}{k}U_{k-1}(x)\sqrt{1-x}\sqrt{x+1},
\end{equation*}
and
\begin{equation*}
	\int_{-1}^{x}\TT_k(s)w_T(s)\df s=-\frac{1}{\sqrt{2}k}U_{k-1}(x)w_U(x).
\end{equation*}
Using the change of variables $t=M_j^{-1}(s)$ where $M_j^{-1}(s)=M^{-1}(s;[a_j,b_j])$, 
\begin{align*}
	\int_{a_j}^{x}\TT_{j,k}(s)w_{\TT,j}(s)\df s=\int_{-1}^{M_j^{-1}(x)}\TT_{j,k}(M_j(t))w_{\TT,j}(M_j(t))\frac{b_j-a_j}{2}\df t.
\end{align*}
Recall that by definition 
\begin{equation*}
	\TT_{j,k}(M_j(t))=\TT_{k}(t),
\end{equation*}
and
\begin{align*}
	w_{\TT,j}(M_j(t))=\frac{2}{b_j-a_j}w_\TT(t),
\end{align*}
so that
\begin{equation*}
	\int_{a_j}^{x}\TT_{j,k}(s)w_{\TT,j}(s)\df s=\int_{-1}^{M_j^{-1}(x)} \TT_k(t)w_\TT(t)\df t=-\frac{1}{\sqrt{2}k}\frac{b_j-a_j}{2}{U}_{j,k-1}(x)w_{U,j}(x),
\end{equation*}
where we use the same notation $w_{U,j}(x)=w_{U}(x;[a_j,b_j])$, $U_{j,k}(x)=\Tilde U_k(x;[a_j,b_j])$.
Thus,
\begin{equation*}
	p_{j,k}(z)=\alpha_{j,k}\frac{\pi i}{k}\frac{b_j-a_j}{\sqrt{2}}\mathcal{C}_{[a_j,b_j]}\left[U_{j,k-1}w_{U,j}\right](z),
\end{equation*} 
for $k>0$. For $k=0$,
\begin{align*}
	p_{j,0}(z)&=-2\im\alpha_{j,0}\int^z\frac{1}{2\pi \im}\int_{a_j}^{b_j}\frac{w_{\TT,j}(s)}{s-z'}\df s\df z'=-2i\alpha_{j,0}\int^z\mathcal{C}_{[a_j,b_j]}\left[ \TT_{j,0}w_{\TT,j}\right]\df z'\\&=
	\alpha_{j,0}\int^z\frac{1}{\sqrt{z'-a}\sqrt{z'-b}}\df z',
\end{align*}
which we have computed using the exact formulae for the Cauchy integrals of Chebyshev-$\TT$ polynomials with their respective weight functions. With a change of variables $t=M(z')$ as before, we find \cite{olver_slevinsky_townsend_2020}
\begin{equation}\label{0order}
	p_{j,0}(z)=-\alpha_{j,0}\log J^{-1}_+\left(M_j^{-1}(z)\right).
\end{equation}
Set
\begin{equation}\label{expansion}
	\Phi(z)=\sum_{j=1}^{g+1}\sum_{k=0}^{m_j}p_{j,k}(z),
\end{equation}
where $m_j$ is the number of terms we use in each Chebyshev series. Then, $\mathfrak{g}$ is approximated by
\begin{equation*}
	\mathfrak{g}(z)\approx\Phi(z)-\Phi(a_1).
\end{equation*}
A numerical instability can arise when computing $\Phi(a_1)$ with this approach due to the non-Lipschitz nature of $J^{-1}_+(z)$ at $z=-1$. This can be circumvented by replacing $J^{-1}_+(-1)=-1$ at each instance in the evaluation of $\Phi(a_1)$ that it occurs.

\subsubsection{Asymptotic behavior}
 We compute the logarithmic\footnote{This does not appear in our recurrence formulae but is necessary in \eqref{cauchyints} below. Additionally, $|\mathfrak c|$ is commonly referred to as the capacity of $\cup_j[a_j,b_j]$ \cite{PEHERSTORFER2011814}.} and first order terms $\mathfrak c$ and $\mathfrak{g}_1$ in the asymptotic expansion of $\mathfrak{g}$ \eqref{gasymp}. Recall that $\mathfrak g$ can be numerically computed as the Cauchy integrals of Chebyshev-$U$ polynomials and the logarithm of the inverse Joukowsky map, so it suffices to compute the asymptotics of these functions.
 
The only term in our expansion \eqref{expansion} that contains a logarithmic term is \eqref{0order}.
We have
\begin{equation*}
	\log J^{-1}_+\left(M_j^{-1}(z)\right)=\log\left(\frac{b_j-a_j}{4z}\right)+\frac{a_j+b_j}{2z}+\OO(z^{-2}),
\end{equation*}
as $z\to\infty$, so that
\begin{equation*}
	p_{j,0}(z)=\alpha_{j,0}\log\left(\frac{4}{b_j-a_j}\right)+\alpha_{j,0}\log z-\alpha_{j,0}\frac{a_j+b_j}{2z}+\OO(z^{-2}).
\end{equation*}
The only other contribution to the $\OO(z^{-1})$ term comes from 
\begin{align*}
	p_{j,1}(z)&=\alpha_{j,1}\pi i\frac{b_j-a_j}{\sqrt{2}}\mathcal{C}_{[a_j,b_j]}\left[U_{j,0}w_{U,j}\right](z)\\&=
	\alpha_{j,1}\pi i\frac{b_j-a_j}{\sqrt{2}}\frac{\im}{2\pi}z^{-1}+\OO(z^{-2})=-\alpha_{j,1}\frac{b_j-a_j}{2\sqrt{2}}z^{-1}+\OO(z^{-2}),
\end{align*}
by \eqref{chebasymp}. This gives 
\begin{equation*}
	\mathfrak c\approx\exp\left(\sum_{j=1}^{g+1}\alpha_{j,0}\log\left(\frac{4}{b_j-a_j}\right)-\Phi(a_1)\right).
\end{equation*}
The first order term can be approximated as 
\begin{equation*}
	\mathfrak{g}_1\approx-\sum_{j=1}^{g+1}\left(\alpha_{j,0}\frac{a_j+b_j}{2}+\alpha_{j,1}\frac{b_j-a_j}{2\sqrt{2}}\right).
\end{equation*}

\subsection{Numerically computing $\mathfrak{h}_n$}
To numerically compute $\mathfrak{h}_n$ we first compute $\Delta_1,\ldots,\Delta_n$ according to \eqref{delta} by numerically evaluating $\mathfrak{g}^+(z)+\mathfrak{g}^-(z)$ on the gaps $(b_j,a_{j+1})$. We then obtain $A_1(n),\ldots,A_{g+1}(n)$ in the same manner as we obtained the coefficients for $\mathfrak g$. We do this by numerically computing the integrals via a Chebyshev-$\TT$-based quadrature rule and solving the system \eqref{hsys}.

Once the coefficients are known, we can evaluate $\mathfrak{h}_n$ in the complex plane in accordance with \eqref{hdef} by developing a method to compute the relevant Cauchy integrals. We do so by utilizing the formulae of Theorem \ref{chebints}. Since we know $R(z)$ and our coefficients, we need only numerically evaluate 
\begin{equation*}
	\mathcal{C}_{[a_j,b_j]}\left[\frac{1}{R_+}\right](z),\quad \mathcal{C}_{[b_\ell,a_{\ell+1}]}\left[\frac{1}{R}\right](z),
\end{equation*}
for $j=1,\ldots,g+1$, $\ell=1,\ldots,g$. To do this, let $a = b_j$, $b = a_{j+1}$ or $a = a_j$, $b = b_j$ for some $j$. If we consider a truncated Chebyshev series
\begin{equation*}
	\frac{\sqrt{s-a}\sqrt{b-s}}{R_+(s)}\approx \sum_{j=0}^{p}\beta_j\Tilde \TT_j(s;[a,b]).
\end{equation*} 
Then,
\begin{equation*}
	\mathcal{C}_{[a,b]}\left[\frac{1}{R_+}\right]\approx\mathcal{C}_{[a,b]}\left[\sum_{j=0}^{p}\beta_j(-\pi i)\Tilde \TT_j(\diamond;[a,b]) \Tilde w_\TT(\diamond;[a,b])\right]=-\pi i\sum_{j=0}^{p}\beta_j\mathcal{C}_{[a,b]}\left[\Tilde \TT_j(\diamond;[a,b]) \Tilde w_\TT(\diamond;[a,b])\right].
\end{equation*}
Of course, these Cauchy integrals can be computed via our exact formulae in Theorem \ref{chebints}, so we can approximate $\mathfrak h_n$ by computing a truncated series and taking linear combinations of Cauchy integrals as appropriate.

\subsubsection{First-order asymptotic behavior}
Now, we compute the first order term $\mathfrak h_n^{(1)}$ in \eqref{hasymp}. This is the $\OO(z^{-g-2})$ term of the Cauchy integrals in \eqref{hdef} since $R(z)\sim z^{g+1}$ as $z\to\infty$. To compute this, we  use the asymptotics of Cauchy integrals to find
\begin{equation*}
	\mathfrak h^{(1)}_n=-\sum_{j=1}^{g+1}
	A_j(n)\int_{a_j}^{b_j}\frac{s^{g+1}}{R_+(s)}\df s-\sum_{\ell=1}^{g}\log(\ex^{n\Delta_\ell})\int_{b_\ell}^{a_{\ell+1}}\frac{s^{g+1}}{R(s)}\df s.
\end{equation*}
To compute this term numerically, we  use a Chebyshev-$\TT$-based quadrature rule on each interval to evaluate the integrals.

\subsection{Numerically solving the Riemann--Hilbert problem and recovering coefficients}
Now that we can compute $\mathfrak{g}$ and $\mathfrak{h}_n$ reliably in the complex plane, we develop a method to solve the Riemann--Hilbert problem \eqref{RHPfinal} numerically. We do so by developing methods to solve smaller and simpler problems that we then combine into a block linear system.

\subsubsection{Numerically solving scalar Riemann--Hilbert problems on a circle}
As a first step, we numerically solve a Riemann--Hilbert problem with a jump condition occurring only on the unit circle. Letting $\mathbb{U}$ denote the unit circle in the complex plane oriented counterclockwise, consider the problem of finding a function $\Phi:\compl\setminus\mathbb{U}\to\compl$ that satisfies
\begin{equation}\label{jumpcirc}
\begin{aligned}
	\Phi~\text{is analytic in}~\compl\setminus\mathbb{U},\\
	\Phi^+(z)=\Phi^-(z)f(z),\quad z\in\mathbb{U},\\
	\lim_{z\to\infty}\Phi(z)=1.
\end{aligned} 
\end{equation}
To find a numerical solution, we make the ansatz 
\[
\Phi(z)=1+\mathcal{C}_\mathbb{U}u(z),
\]
for some unknown function $u$.  This ansatz can be rigorously justified, see \cite[Section 2.5]{trogsworldllc}. Plugging this into the jump condition in \eqref{jumpcirc},
\begin{equation}\label{newjumpcirc}
	\mathcal{C}^+_\mathbb{U}u(z)-f(z)\mathcal{C}^-_\mathbb{U}u(z)=f(z)-1.
\end{equation}
Now, we let $u$ be given by a Laurent series
\begin{equation*}
	u(z)=\sum_{j=-\infty}^{\infty}c_jz^j.
\end{equation*}
A simple residue calculation yields that 
\begin{equation}\label{residue}
	\mathcal{C}_\mathbb{U}u(z)=\begin{cases}
		\sum_{j=0}^{\infty}c_jz^j& |z|<1,\\
		-\sum_{j=-\infty}^{-1}c_jz^j& |z|>1.
	\end{cases}
\end{equation}
To form a numerical solution, we truncate our Laurent series so that its terms range from $-N$ to $N$ and consider $2N+1$ collocation points $z_0,\ldots,z_{2N}$ evenly distributed on $\mathbb{U}$ \footnote{For instance, given by $z_k=\ex^{2\pi ik/(2N+1)}$.}. We then determine coefficients $c_{-N},\ldots,c_N$ such that \eqref{newjumpcirc} is satisfied at all collocation points:
\begin{equation}\label{linearsys}
	\sum_{j=0}^{N}c_jz_k^j+f(z_k)\sum_{j=-N}^{-1}c_jz_k^j=f(z_k)-1,
\end{equation}
for $k=0,\ldots,2N$. To solve for the coefficients, we write \eqref{linearsys} as a linear system and solve for $c_{-N},\ldots,c_N$. We then use \eqref{residue} to evaluate our numerical solution $\Phi(z)$ at any $z\in\compl\setminus\mathbb{U}$. One can of course generalize this to numerically solve a scalar Riemann--Hilbert problem on any circle by scaling and shifting appropriately \cite{Olver2012}. 

\subsubsection{Numerically solving scalar Riemann--Hilbert problems on an interval}
Now, we numerically solve the Riemann--Hilbert problem of finding a function $\Phi:\compl\setminus\mathbb{I}\to\compl$ that satisfies
\begin{equation*}
\begin{aligned}
	&\Phi~\text{is analytic in}~\compl\setminus\mathbb{I},\\
	&\Phi^+(z)=\Phi^-(z)f(z),\quad z\in\mathbb{I},\\
	&\lim_{z\to\infty}\Phi(z)=1, 
\end{aligned} 
\end{equation*}
where $\mathbb{I}=[-1,1]$. We again make the ansatz 
\[
\Phi(z)=1+\mathcal{C}_\mathbb{I}u(z),
\]
but need a different approach to determine $u$. As with most numerical analysis problems on the unit interval, we cannot safely use a monomial basis and equally-spaced points as in the problem on the unit circle. Instead, we need to use an orthogonal polynomial basis, meaning that we will need to use the Cauchy integrals of orthogonal polynomials
when enforcing our jump condition of the form \eqref{newjumpcirc} but on $\mathbb{I}$. We choose $2N+1$ collocation points to be more concentrated near the endpoints as the roots of Chebyshev-$\TT$, i.e., we set
\[
z_k=\cos\left(\frac{(2k+1)\pi}{2(2N+1)}\right),\quad k=0,\ldots,2N.
\]
If we use an orthogonal polynomial basis $p_0,\ldots,p_{2N}$ with associated weight function $w$, then we find coefficients $d_0,\ldots,d_{2N}$ such that the jump condition holds at all collocation points:
\begin{equation}\label{linearsysint}
	\sum_{j=0}^{2N}d_j\mathcal{C}_\mathbb{I}^+[p_jw](z_k)-f(z_k)\sum_{j=0}^{2N}d_j\mathcal{C}_\mathbb{I}^-[p_jw](z_k)=f(z_k)-1,
\end{equation}
for $k=0,\ldots,2N$. This is a linear system and can (potentially) be solved for $d_0,\ldots,d_{2N}$ by computing $\mathcal{C}_\mathbb{I}^+[p_jw](z_k),\mathcal{C}_\mathbb{I}^-[p_jw](z_k)$ as described in accordance with Theorem \ref{chebints} if we choose to use a Chebyshev basis for $p_j$, $j=0,\ldots,2N$. Of course, this problem can be shifted and scaled to an arbitrary interval $[a,b]$ \cite{Olver2012}.

While solving this linear system appears straightforward, a potential difficulty arises in choosing the orthogonal polynomial basis used to build the numerical solution. If the correct basis is not chosen, the method will not converge. The Cauchy integrals of different orthogonal polynomial bases have different asymptotic behavior near the endpoints $z\to\pm1$, and therefore different choices of bases will produce functions that behave very differently. If one wishes to pose a Riemann--Hilbert problem with a unique solution, additional asymptotic conditions must be prescribed at the endpoints. Then, the correct orthogonal polynomial basis associated with that asymptotic condition must be selected.

\subsubsection{Combining methods and solving the orthogonal polynomial problem}
Consider solving a Riemann--Hilbert problem on the contour $\Gamma=\Gamma_1\cup\Gamma_2$ where $\Gamma_1,\Gamma_2$ are disjoint contours: Find a function $\Phi:\compl\setminus\Gamma\to\compl$ that satisfies
\begin{equation*}
\begin{aligned}
	&\Phi~\text{is analytic in}~\compl\setminus\Gamma,\\
	&\Phi^+(z)=\Phi^-(z)f_1(z),\quad z\in\Gamma_1,\\
	&\Phi^+(z)=\Phi^-(z)f_2(z),\quad z\in\Gamma_2,\\
	&\lim_{z\to\infty}\Phi(z)=1.
\end{aligned} 
\end{equation*}
Define the projection operator $P_j$ by
\begin{equation*}
P_ju(z)=\begin{cases}
	u(z) & z\in\Gamma_j,\\
	0 & \text{otherwise}.
\end{cases}
\end{equation*}
The ansatz becomes
\[
\Phi(z)=1+\mathcal{C}_\Gamma u_1(z)+\mathcal{C}_\Gamma u_2(z),
\]
where $u_j=P_ju$ for some unknown $u$. Noting that $\mathcal{C}_{\Gamma}=\mathcal{C}_{\Gamma_1}+\mathcal{C}_{\Gamma_2}$, the jump condition becomes
\begin{equation*}
P_jC_{\Gamma_j}^+u_j+P_jC_{\Gamma_k}^+u_k-f_j(P_jC_{\Gamma_j}^-u_j+P_jC_{\Gamma_k}^-u_k)=f_j-1,
\end{equation*}
for $j,k\in\{1,2\}$, $j\neq k$. This gives the block system
\begin{equation}\label{blocksys}
	\begin{pmatrix}P_1C_{\Gamma_1}^+-f_1P_1C_{\Gamma_1}^- &P_1C_{\Gamma_2}-f_1P_1C_{\Gamma_2}\\
		P_2C_{\Gamma_1}-f_2P_2C_{\Gamma_1} &P_2C_{\Gamma_2}^+-f_2P_2C_{\Gamma_2}^-
	\end{pmatrix}\begin{pmatrix}
	u_1\\u_2
\end{pmatrix}=\begin{pmatrix}
f_1-1\\f_2-1
\end{pmatrix},
\end{equation}
where $P_jC_{\Gamma_k}^+=P_jC_{\Gamma_k}^-=P_jC_{\Gamma_k}$ when $j\neq k$. We expand $u_1,u_2$ in the usual way\footnote{It is expanded using either a Laurent series or an appropriate orthogonal polynomial basis depending on if $\Gamma_1,\Gamma_2$ are intervals or circles.} and choose collocation points for each on $\Gamma_1$ and $\Gamma_2$, respectively. Then, we enforce the jump conditions at the collocation points as a block linear system via the matrix structure \eqref{blocksys} and solve for the basis coefficients for both $u_1$ and $u_2$. A block system can be formulated analogously when $\Gamma$ is composed of an arbitrary number of disjoint contours. In this case, the $(j,k)$th entry in the block matrix is given by $P_jC_{\Gamma_k}^+-f_jP_jC_{\Gamma_k}^-$.

We can extend these methods to a matrix Riemann--Hilbert problem for $\mathbf{\Phi}:\compl\setminus\Gamma\to\compl^{2\times2}$
\begin{align*}
	&\mathbf{\Phi}~\text{is analytic in}~\compl\setminus\Gamma,\\
	&\mathbf{\Phi}^+(z)=\mathbf{\Phi}^-(z)\bF(z),\quad z\in\Gamma,\\
	&\lim_{z\to\infty}\mathbf{\Phi}(z)=\bI.
\end{align*} 
If one considers the ansatz 
\begin{equation}\label{matrixansatz}
\mathbf{\Phi}(z)=\bI+\mathcal{C}_\Gamma \bU(z),
\end{equation}
where $\bU$ is now matrix-valued and the Cauchy operator is taken element-wise. We find a system of the form
\begin{equation*}
	\begin{pmatrix}\mathcal C_\Gamma^+-\bF_{11}\mathcal C_\Gamma^- &-\bF_{21}\mathcal C_\Gamma^-\\-\bF_{12}\mathcal C_\Gamma^- &\mathcal C_\Gamma^+-\bF_{22}\mathcal C_\Gamma^-\end{pmatrix}\begin{pmatrix}\bU_{11} &\bU_{21}\\\bU_{12} &\bU_{22}\end{pmatrix}=\bF^T-\bI.
\end{equation*}
We solve this using our methods for scalar problems by forming a block system. Solving a matrix problem on multiple domains will require forming a block system of block systems \cite{Olver2012}. The equations for $\bU_{11},\bU_{12}$ and $\bU_{21},\bU_{22}$ decouple, so $\bU$ can be solved for row-wise.

The pieces are now in place to numerically solve the Riemann--Hilbert problem \eqref{RHPfinal}.
As a simplifying assumption, we take each $C_j$ in Figure \ref{regions} to be a circle centered at $(a_j+b_j)/2$ with a diameter greater than $b_j-a_j$\footnote{Heuristically, we find that the diameter should be at least $5(b_j-a_j)/4$ to avoid stability issues.} but small enough so that $C_j$ and $C_k$ do not intersect when $j\neq k$\footnote{To be fully robust, one should take $C_j$ to be a Bernstein ellipse to which $1/h_j(z)$ is analytically continuable, but this requires a method to numerically solve Riemann--Hilbert problems on such an ellipse.}. Then, \eqref{RHPfinal} can be solved via a block linear system according to our methods for circles, intervals, multiple domains, and matrix problems by accounting for shifting and scaling things to the unit interval and circle as appropriate. The only remaining difficulty lies in choosing orthogonal polynomial bases so that the asymptotic conditions at the endpoints are satisfied.

On each interval $[a_j,b_j]$, we choose orthogonal polynomial bases based on the endpoint asymptotic conditions from \eqref{RHPfinal}:
\begin{equation*}
\begin{aligned}
	&\Tilde\bS_n(z)=\OO\begin{pmatrix}
		1+|z-a_j|^{-\alpha_j/2} &1+|z-a_j|^{\alpha_j/2}\\
		1+|z-a_j|^{-\alpha_j/2} &1+|z-a_j|^{\alpha_j/2}
	\end{pmatrix},\quad z\to a_j,\\
	&\Tilde\bS_n(z)=\OO\begin{pmatrix}
		1+|z-b_j|^{-\beta_j/2} &1+|z-b_j|^{\beta_j/2}\\
		1+|z-b_j|^{-\beta_j/2} &1+|z-b_j|^{\beta_j/2}
	\end{pmatrix},\quad z\to b_j.
\end{aligned}
\end{equation*}
Because the weight functions have asymptotic behavior
\begin{align*}
w_{\TT,j}(z)&=\OO(|z-a_j|^{-1/2}),\quad &z\to a_j,\quad&w_{\TT,j}(z)=\OO(|z-b_j|^{-1/2}),&\quad &z\to b_j,\\
w_{U,j}(z)&=\OO(|z-a_j|^{1/2}),\quad &z\to a_j,\quad&w_{U,j}(z)=\OO(|z-b_j|^{1/2}),&\quad &z\to b_j,\\
w_{V,j}(z)&=\OO(|z-a_j|^{1/2}),\quad &z\to a_j,\quad&w_{V,j}(z)=\OO(|z-b_j|^{-1/2}),&\quad &z\to b_j,\\
w_{W,j}(z)&=\OO(|z-a_j|^{-1/2}),\quad &z\to a_j,\quad&w_{W,j}(z)=\OO(|z-b_j|^{1/2}),&\quad &z\to b_j,
\end{align*}
we choose bases as follows:
\begin{itemize}
\item{If $\alpha_j,\beta_j=-1$, then we use $\mathcal{C}_{[a_j,b_j]}[U_{j,k}w_{U,j}]$ in the first column of the matrix ansatz $\bU$ \eqref{matrixansatz} and $\mathcal{C}_{[a_j,b_j]}[\TT_{j,k}w_{\TT,j}]$ in the second column of the ansatz.}
\item{If $\alpha_j,\beta_j=1$, then we use $\mathcal{C}_{[a_j,b_j]}[\TT_{j,k}w_{\TT,j}]$ in the first column and $\mathcal{C}_{[a_j,b_j]}[U_{j,k}w_{U,j}]$ in the second column.}
\item{If $\alpha_j=1,\beta_j=-1$, then we use $\mathcal{C}_{[a_j,b_j]}[W_{j,k}w_{W,j}]$ in the first column and $\mathcal{C}_{[a_j,b_j]}[V_{j,k}w_{V,j}]$ in the second column.}
\item{If $\alpha_j=-1,\beta_j=1$, then we use $\mathcal{C}_{[a_j,b_j]}[V_{j,k}w_{V,j}]$ in the first column and $\mathcal{C}_{[a_j,b_j]}[W_{j,k}w_{W,j}]$ in the second column.}
\end{itemize}

\subsubsection{Recovering coefficients}
Once a numerical solution to \eqref{RHPfinal} is obtained, we compute $\Tilde \bS_n^{(1)}$ by summing the first order terms of each approximate function in the block system. 
More explicitly, if the $(\ell,m)$ entry of $\Tilde\bS_n$ has a Laurent series representation
\begin{equation*}
\sum_{k=-N_j}^{N_j}c_{j,k}\left(\frac{z-q_j}{r_j}\right)^k,
\end{equation*}
on each $C_j$ where $q_j=(a_j+b_j)/2$ and $r_j$ denotes the radius of $C_j$ and an orthogonal polynomial series representation 
\begin{equation*}
\sum_{k=0}^{2L_j}d_{j,k}\mathcal{C}_{[a_j,b_j]}\left[p_kw\right],
\end{equation*}
on each $[a_j,b_j]$, then
\begin{equation*}
\left(\Tilde\bS_{n}^{(1)}\right)_{\ell m}=\sum_{j=1}^{g+1}\left(r_jc_{j,-1}+\frac{\im}{2\pi}d_{j,1}\right),
\end{equation*}
where the scaling follows from \eqref{chebasymp} if $p_k$ and $w$ are a shifted and scaled Chebyshev polynomial and its weight, respectively. This combined with our computed first order contributions of $\mathfrak{g}$ and $\mathfrak{h}_n$ from earlier allows us to recover our recurrence coefficients from \eqref{realrecurr} as
\begin{equation*}
	\begin{aligned}
		a_n&=\left(\Tilde \bS_{n}^{(1)}\right)_{11}-\left(\Tilde \bS_{n+1}^{(1)}\right)_{11}-\mathfrak h_n^{(1)}+\mathfrak h_{n+1}^{(1)}-\mathfrak g_1,\\
		b_n&=\sqrt{\left(\Tilde \bS_{n+1}^{(1)}\right)_{12}\left(\Tilde \bS_{n+1}^{(1)}\right)_{21}}.
	\end{aligned}
\end{equation*}

\section{Examples and applications}\label{examp}
\subsection{The single interval case}\label{singint}
As a first example, we consider our methods applied to a single interval for which recurrence coefficients are known asymptotically. For simplicity, we consider this interval to be $[-1,1]$. It turns out that the functions $\mathfrak{g}$ and $\mathfrak{h}_n$ can be computed exactly in this case, so solving the Riemann--Hilbert problem \eqref{RHPfinal} numerically becomes simpler. First, we observe that there are no gaps between intervals on which jump conditions can be induced, so the function $\mathfrak h_n$ is identically zero for all $n$. Furthermore, determining $\mathfrak g'$ becomes trivial, and the exponential of $\mathfrak g$ can be computed explicitly as \cite{Kuijlaars2003}
\begin{equation*}
\ex^{\mathfrak g(z)}=\varphi(z)=z+\sqrt{z+1}\sqrt{z-1},
\end{equation*}
and the weight function is given by 
\[
w(x)=h_1(x)\left(\sqrt{x+1}\right)^{\alpha_1}\left(\sqrt{1-x}\right)^{\beta_1},
\]
and $h_1$ satisfies the usual assumptions of being positive on $[-1,1]$ and analytic in a neighborhood of it \cite{KUIJLAARS2004337}.
We additionally find that the first-order correction $\mathfrak g_1=0$ in this case, meaning that all correction terms in the recurrence coefficient formulae \eqref{realrecurr} drop out.

As with the multiple interval problem, we choose $C_1=\{z:|z|=5/4\}$, restricting our choice of $h$ to functions that are positive on $[-1,1]$ and whose reciprocals are analytic in an open set containing $\bar D_1$. For our example, we choose
\begin{equation}\label{hexample}
h_1(x)=\frac{\ex^x+1}{4+x^2},
\end{equation}
and $\alpha_1=\beta_1=1$ to obtain a Chebyshev-$U$-like weight. We first compute the first 51 coefficients of these orthogonal polynomials using 160 collocation points on $C_1$ and 16 collocation points on $[-1,1]$. We should expect these coefficients $a_j,b_j$ to tend to the Chebyshev-$U$ coefficients ($a_j=0,b_j=1/2$) exponentially as $j\to\infty$ \cite{geronimo}. In Figure \ref{expconv}, we plot these differences $|a_j|$ and $|b_j-1/2|$ for the first 51 computed coefficients that exhibit the expected behavior.\\
\begin{figure}
	\centering
	\begin{subfigure}{0.495\linewidth}
		\centering
		\includegraphics[width=\linewidth]{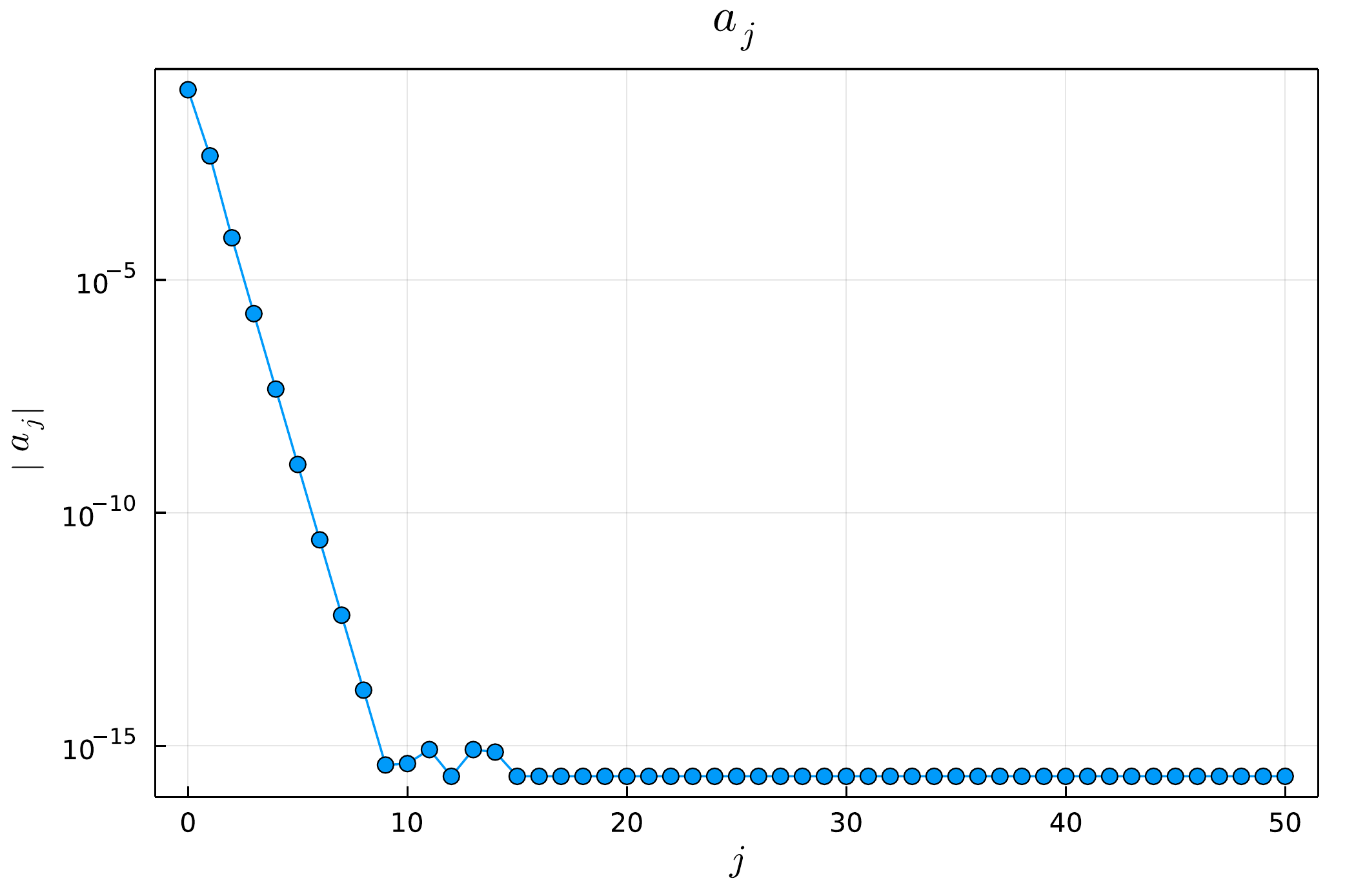}
	\end{subfigure}
	\begin{subfigure}{0.495\linewidth}
		\centering
		\includegraphics[width=\linewidth]{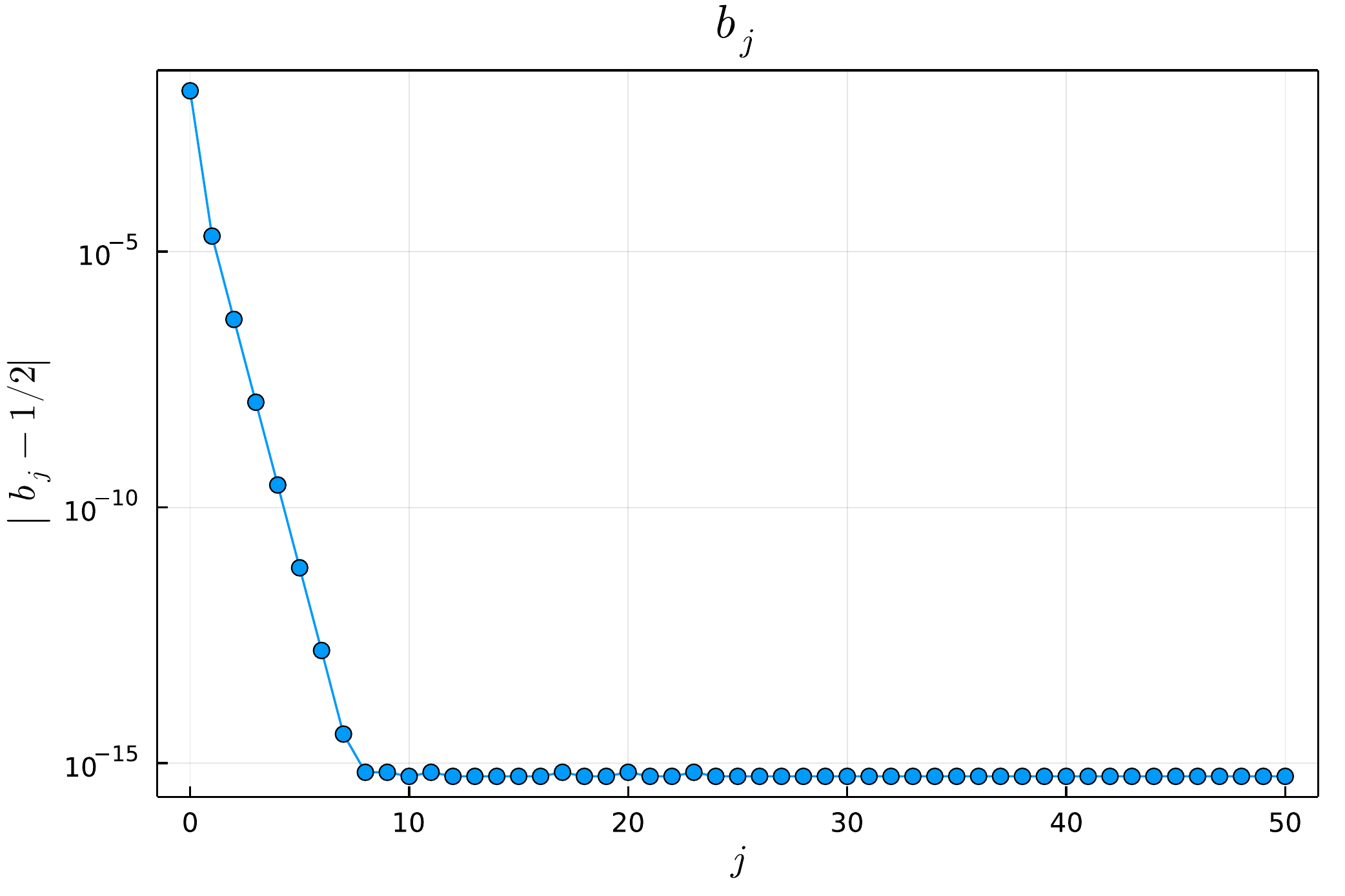}
	\end{subfigure}
\caption{Plots of the difference between the first 51 orthogonal polynomial coefficients of a Chebyshev-$U$ weight modified with \eqref{hexample} and an unmodified Chebyshev-$U$ weight.}
\label{expconv}
\end{figure}

Now, we present an alternative Householder method for computing recurrence coefficients, to which the accuracy of our method is compared. Consider a discrete $N$-point approximation to the weight $w$ with nodes $x_1,\ldots,x_N$ and weights $w_1,\ldots,w_N$, i.e.,
\begin{equation*}
\int_\Sigma f(x)w(x)\approx\sum_{j=1}^Nf(x_j)w_j.
\end{equation*} 
Define 
\begin{equation*}
\bd=\begin{pmatrix}
	\sqrt{\omega_1}\\
	\vdots\\
	\sqrt{\omega_N}
\end{pmatrix},\quad
\bLambda=\begin{pmatrix}
	x_1&0&\cdots&0\\
	0&x_2&\cdots&\vdots\\
	\vdots&\vdots&\ddots&\vdots\\
	0&0&\cdots&x_n
\end{pmatrix}.
\end{equation*} Then, from \cite[Eq. 2.2.24]{gautschi}, if the nodes and weights are selected according to the Gauss quadrature rule for $w$, we have a factorization of the form
\begin{equation*}
\begin{pmatrix}
1 &\bzero^T\\
\bzero &\bQ
\end{pmatrix}\begin{pmatrix}
1 &\bd^T\\
\bd^T &\bLambda
\end{pmatrix}\begin{pmatrix}
1 &\bzero^T\\
\bzero &\bQ^T
\end{pmatrix}=\begin{pmatrix}
1&b_{-1}\be_1^T\\
b_{-1}\be_1 &\bJ_N(w)
\end{pmatrix},
\end{equation*}
where $\be_1$ is the first unit vector and $\bJ_N(w)$ is the $N\times N$ principal subblock of $\bJ(w)$. If one can compute $\bQ$, this allows for the recovery of $\bJ_N(w)$; the Householder algorithm \cite{householder} is a stable way to do this. Due to the complexity of the Householder algorithm, this approach requires $\OO(N^3)$ arithmetic operations to compute the first $N$ recurrence coefficients. An algorithm based on Givens rotations with $\OO(N^2)$ arithmetic operations is presented in \cite{Gragg1984}.  An important distinction between these two methods and the Stieltjes procedure is that these two methods require a priori knowledge of the desired recurrence length $N$, and altering this a posteriori requires restarting the entire algorithm. 

In Figure \ref{1interr}, we plot the error at each of the first 51 coefficients for various numbers of collocation points on the $[-1,1]$, always choosing the number of collocation points on $C_1$ to be 10 times the number on $[-1,1]$.
\begin{figure}
	\centering
	\begin{subfigure}{0.95\linewidth}
		\centering
		\includegraphics[width=\linewidth]{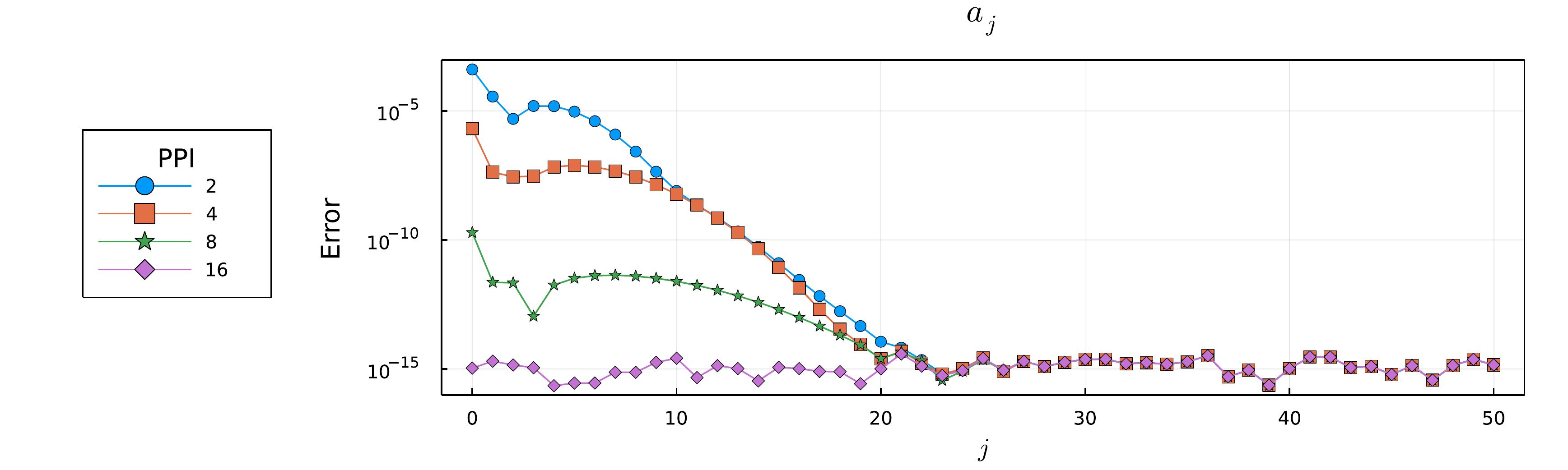}
	\end{subfigure}
	\begin{subfigure}{0.95\linewidth}
		\centering
		\includegraphics[width=\linewidth]{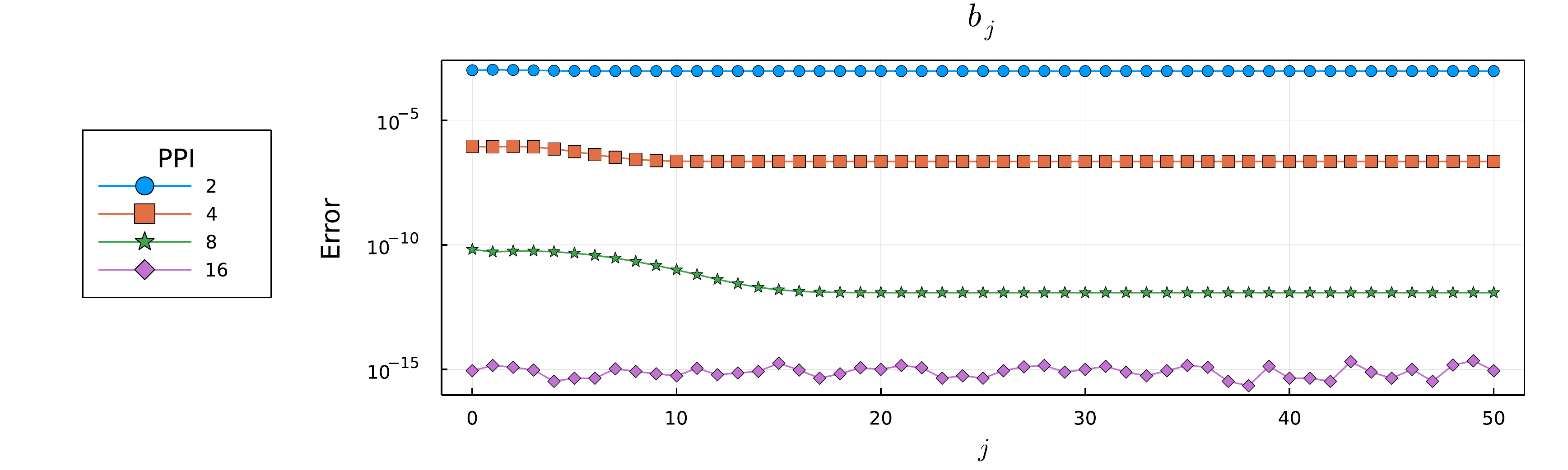}
	\end{subfigure}
	\caption{Absolute error in the computed first 51 orthogonal polynomial coefficients for a Chebyshev-$U$ weight modified with \eqref{hexample} for 2, 4, 8, and 16 collocation points per interval (PPI) on $[-1,1]$.}
	\label{1interr}
\end{figure}
The method is able to compute the recurrence coefficients $a_j,b_j$ starting at any choice of $j$ without first computing $a_k,b_k$ for $k < j$. To demonstrate this, we include a similar plot showing the errors in computing $a_j$ and $b_j$ for $j=1000,\ldots,1020$ for various numbers of collocation points in Figure \ref{1interrasymp}.\\
\begin{figure}
	\centering
	\begin{subfigure}{0.95\linewidth}
		\centering
		\includegraphics[width=\linewidth]{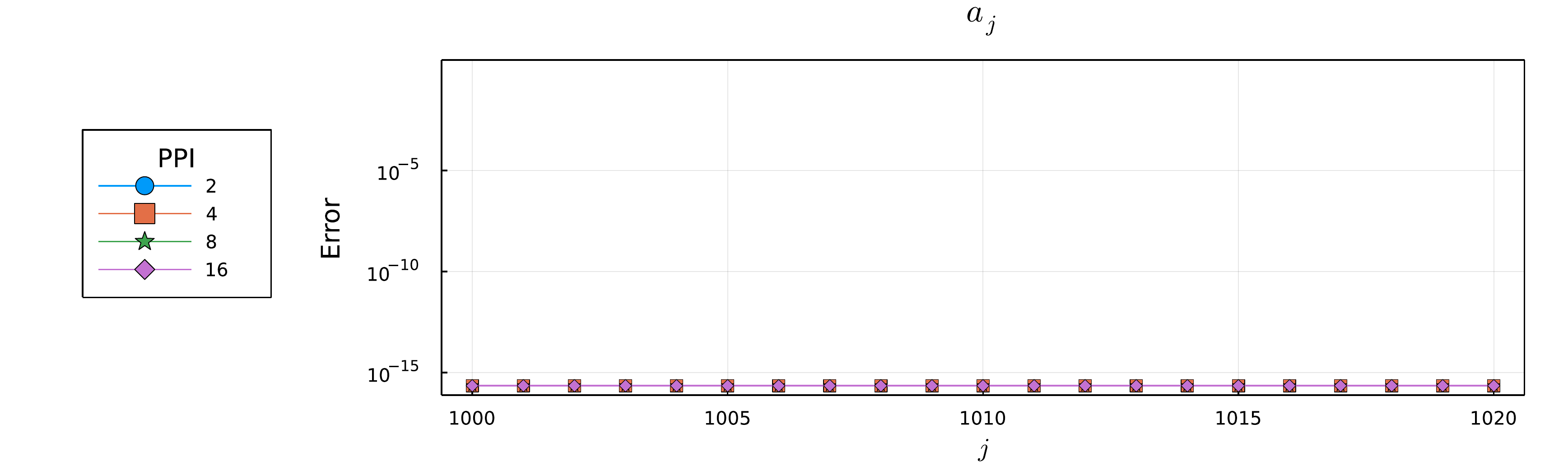}
	\end{subfigure}
	\begin{subfigure}{0.95\linewidth}
		\centering
		\includegraphics[width=\linewidth]{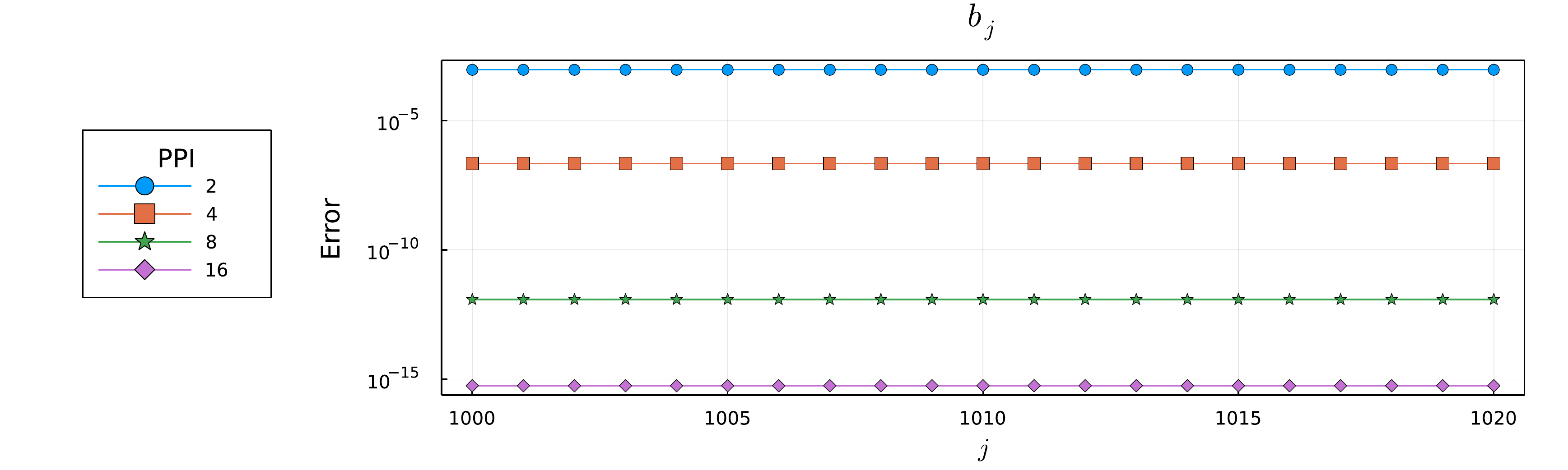}
	\end{subfigure}
	\caption{Error in the computed $a_j,b_j$ for $j=1000,\ldots,1020$ for a Chebyshev-$U$ weight modified with \eqref{hexample} for 2, 4, 8, and 16 collocation points per interval (PPI) on $[-1,1]$ . In this case, $a_j=0$, and the method approximates it well.}
	\label{1interrasymp}
\end{figure}

\subsection{Multiple intervals}\label{sect:multi}
Now, the accuracy of our method is demonstrated in computing with orthogonal polynomials on multiple intervals. In all examples, we will take $h_j(x)=1$ for all $j$, for simplicity, and compare our computed coefficients against those computed by the Householder method described in Section \ref{singint}. 

In Figure \ref{2interr}, we consider orthogonal polynomials defined on $[-1.8, -1]\cup[2, 3]$ with a Chebyshev-$\TT$-like weight on each ($\alpha_j=\beta_j=-1$ for $j=1,2$). 
\begin{figure}
	\centering
	\begin{subfigure}{0.95\linewidth}
		\centering
		\includegraphics[width=\linewidth]{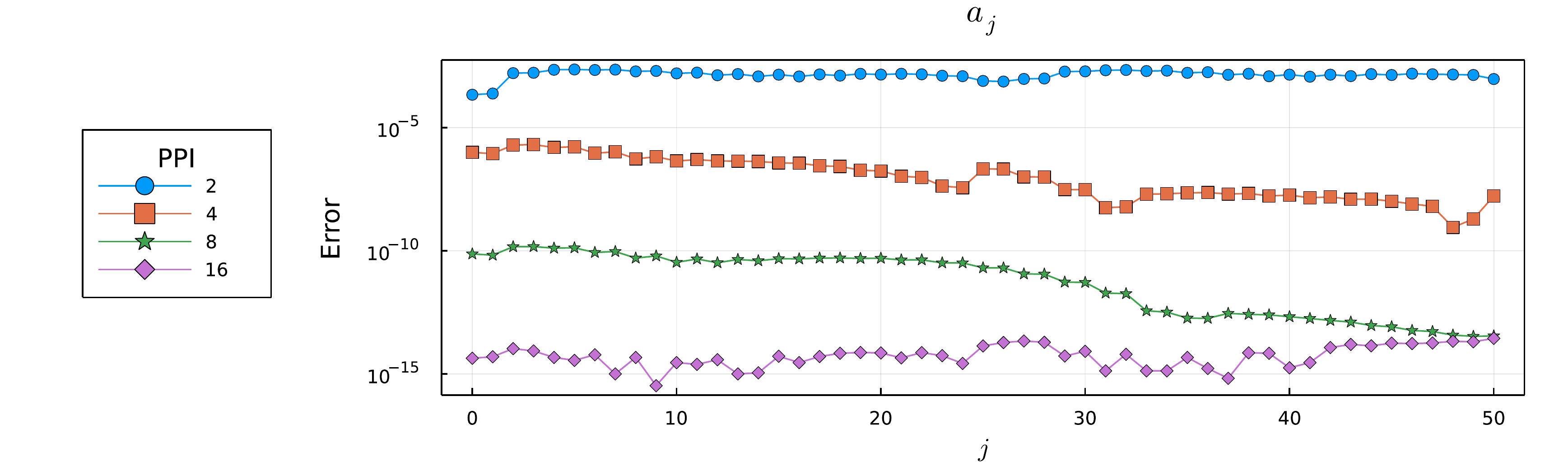}
	\end{subfigure}
	\begin{subfigure}{0.95\linewidth}
		\centering
		\includegraphics[width=\linewidth]{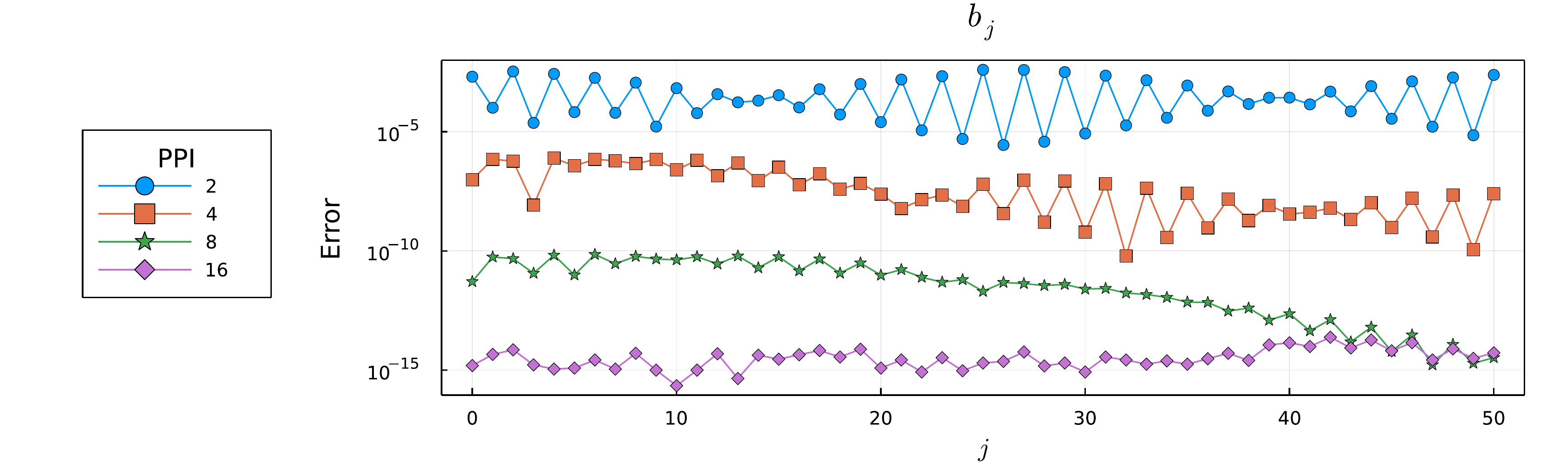}
	\end{subfigure}
	\caption{Error in the computed first 51 orthogonal polynomial coefficients for a Chebyshev-$\TT$-like weight on $[-1.8, -1]\cup[2, 3]$ for 2, 4, 8, and 16 collocation points per interval (PPI). The convergence for large $j$ in the 8 collocation point case suggests that ratio of collocation points on circles to points on intervals is not optimized. Computing these coefficients requires 31.44 seconds via the Householder method described in Section \ref{singint}, and 0.13, 0.12, 0.30, and 0.97 seconds, respectively, via our method.}
	\label{2interr}
\end{figure}
Furthermore, in Figure \ref{timings}, we include a timing comparison of our method to an optimized $\OO(N^2)$ algorithm given as RKPW in \cite{Gragg1984} and \texttt{lanczos.m} in \cite{gautschi}.
\begin{figure}
	\centering
        \includegraphics[scale=0.5]{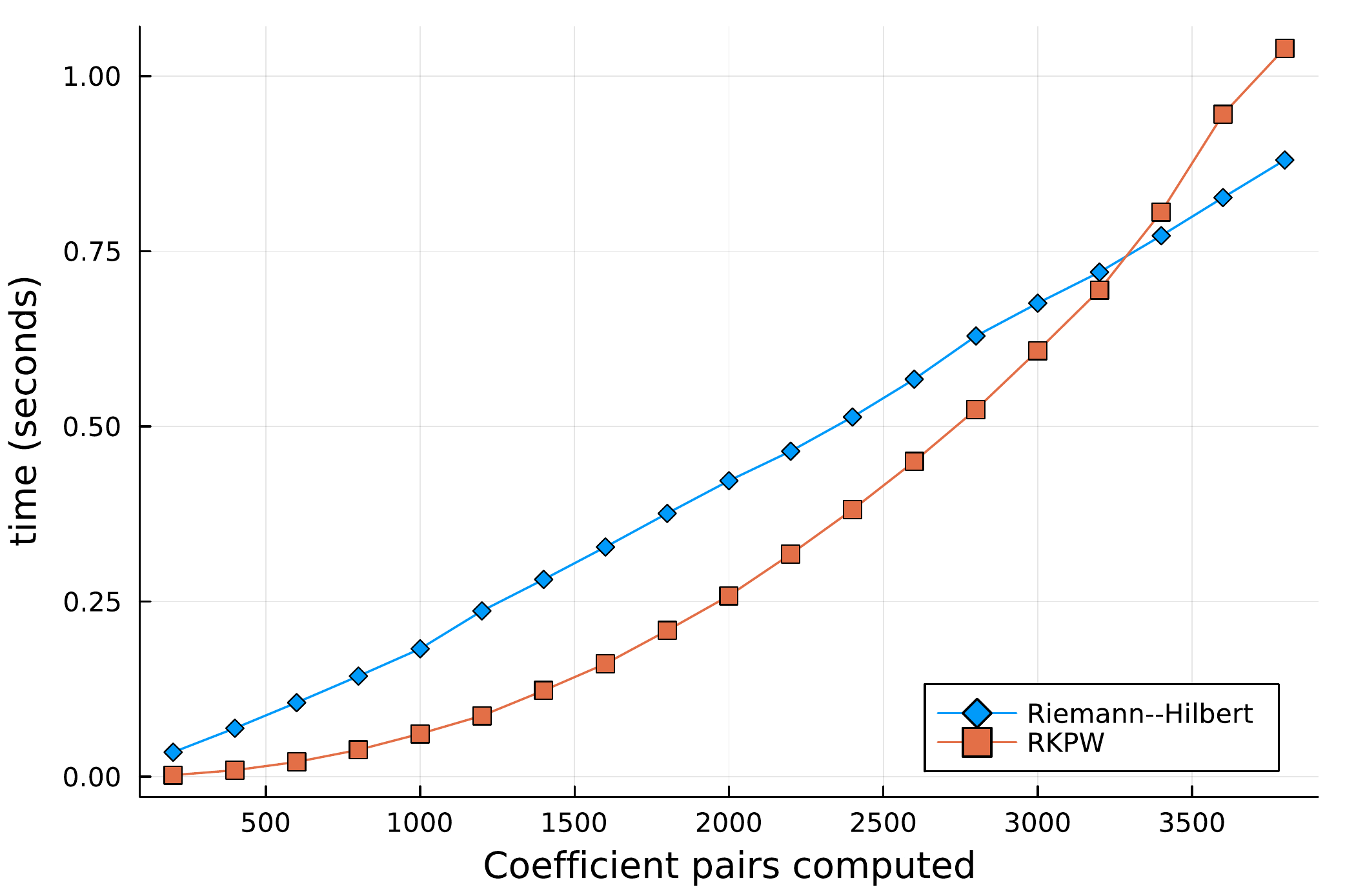}
	\caption{Timing comparison of the RKPW method \cite{Gragg1984} and our approach with only the jump conditions on the intervals considered. We observe that our method begins to outperform the RKPW method once around 3200 pairs of recurrence coefficients are computed.}
	\label{timings}
\end{figure}
Since we are concerned with asymptotic timings for this comparison, our method is applied with only the jump conditions on the intervals as mentioned in Section \ref{sect:recurr}. As is also mentioned in this section, there exist classes of weight functions for which the circles $C_j$ are not needed for any polynomial degree, and we will explore this further in future work. We observe that our method begins to outperform the RKPW method once around 3200 pairs of recurrence coefficients need to be computed; however, since the RKPW method is heavily optimized while our method is coded suboptimally, we expect the threshold to be much lower with a better implementation. The RKPW method also has the aforementioned weakness in that it requires a priori knowledge of the desired recurrence length $N$, and altering this a posteriori requires potentially restarting the entire algorithm. Thus, in practice, one should consider using the RKPW method to compute the first few coefficients and our method asymptotically.

In Figure \ref{3interr}, we consider orthogonal polynomials defined on $[0.1, 1.1]\cup[2, 3]\cup[3.5, 4]$ with a Chebyshev-$V$-like weight on each ($\alpha_j=1,\beta_j=-1$ for $j=1,2,3$). Finally, in Figure \ref{4interr}, we consider orthogonal polynomials defined on $[-3.2, -2.2]\cup[0.1, 1.1]\cup[2, 3]\cup[3.5, 4]$ with a Chebyshev-$\TT$-like weight on $[-3.2, -2.2]$, a Chebyshev-$U$-like weight on $[0.1, 1.1]$, a Chebyshev-$V$-like weight on $[2,3]$, and a Chebyshev-$W$-like weight on $[3.5,4]$ ($\alpha_1=\beta_1=-1$, $\alpha_2=\beta_2=1$, $\alpha_3=1,\beta_3=-1$, $\alpha_4=-1,\beta_4=1$). In each example, we use the same number (2, 4, 8, and 16) of collocation points on each interval and 10 times that number on each circle $C_j$ and compute the first 51 recurrence coefficients. In practice, one should consider using a different number of collocation points on each interval and varying the ratio of collocation points on circles to points on intervals to optimize runtime.

Heuristically, our method as described encounters stability issues once $g$ reaches 4 or 5; however, 6 digits of accuracy is maintained until about $g=11$. This is because the linear systems \eqref{gsys} and \eqref{hsys} become ill-conditioned as $g$ increases. If one wishes to consider larger $g$, these instabilities can be circumvented by instead building linear systems with the basis in \cite[Eq. 28]{kdvgenus}.

In all figures, we observe rapid convergence to the ``ground truth'' coefficients $a_j,b_j$ for any value of $j$. Errors on the order of machine precision are always achieved when we use 16 collocation points on each interval and 160 points on each circle. Oftentimes, less collocation points are needed to achieve this accuracy. 
\begin{figure}
	\centering
	\begin{subfigure}{0.95\linewidth}
		\centering
		\includegraphics[width=\linewidth]{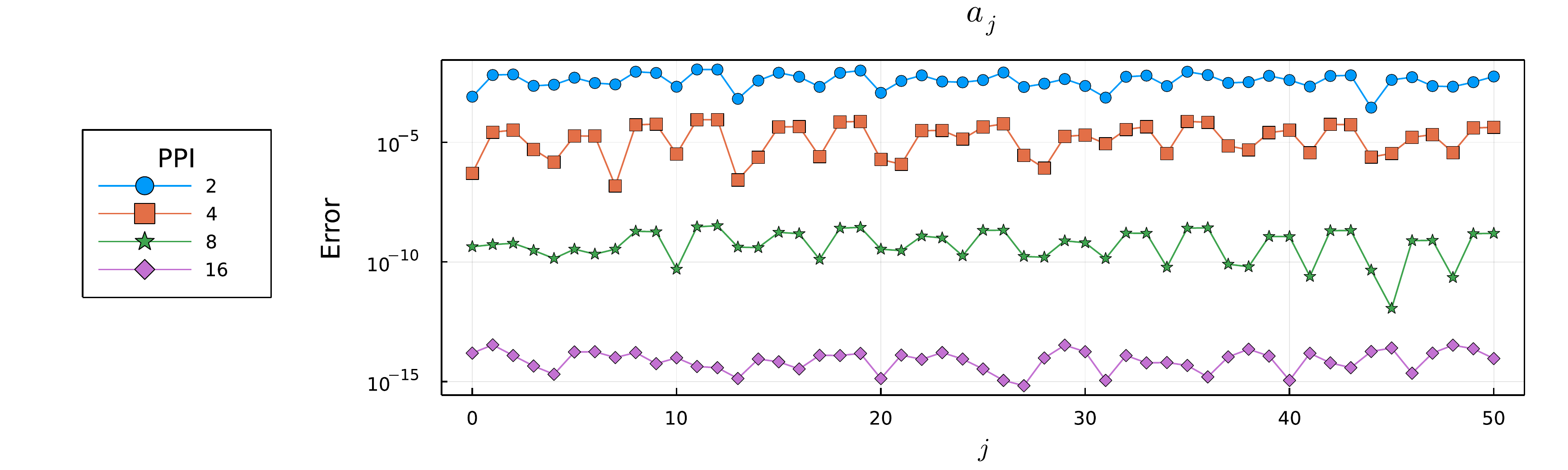}
	\end{subfigure}
	\begin{subfigure}{0.95\linewidth}
		\centering
		\includegraphics[width=\linewidth]{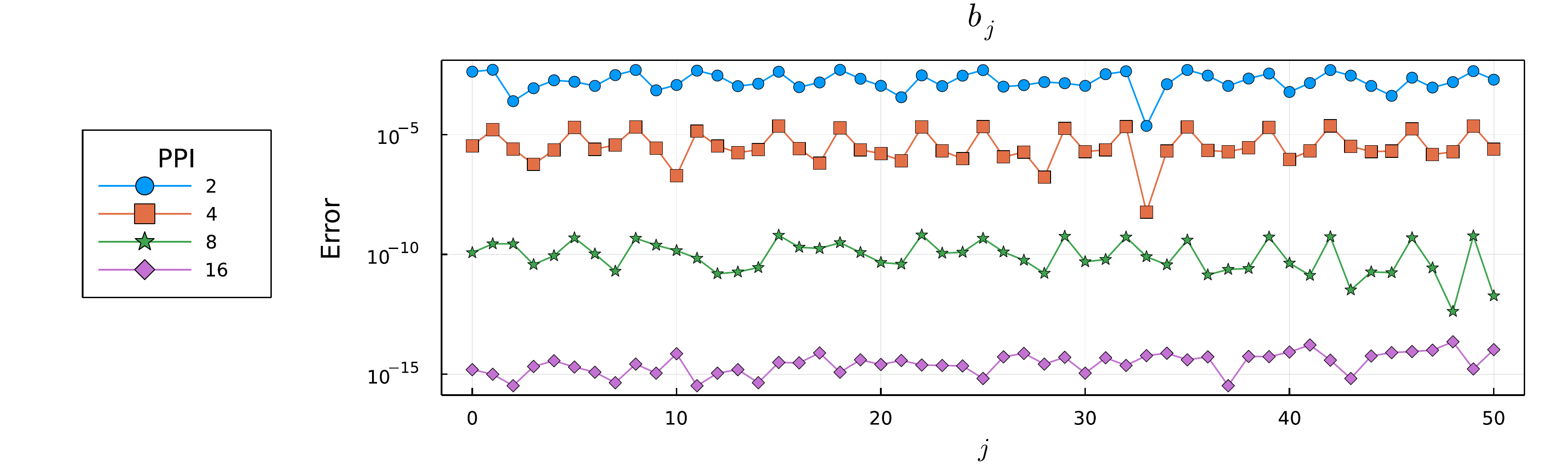}
	\end{subfigure}
	\caption{Error in the computed first 51 orthogonal polynomial coefficients for a Chebyshev-$V$-like weight on $[0.1, 1.1]\cup[2, 3]\cup[3.5, 4]$ for 2, 4, 8, and 16 collocation points per interval (PPI). Computing these coefficients requires 108.53 seconds via the Householder method described in Section \ref{singint}, and 0.13, 0.25, 0.57, and 2.78 seconds, respectively, via our method.}
	\label{3interr}
\end{figure}
\begin{figure}
	\centering
	\begin{subfigure}{0.95\linewidth}
		\centering
		\includegraphics[width=\linewidth]{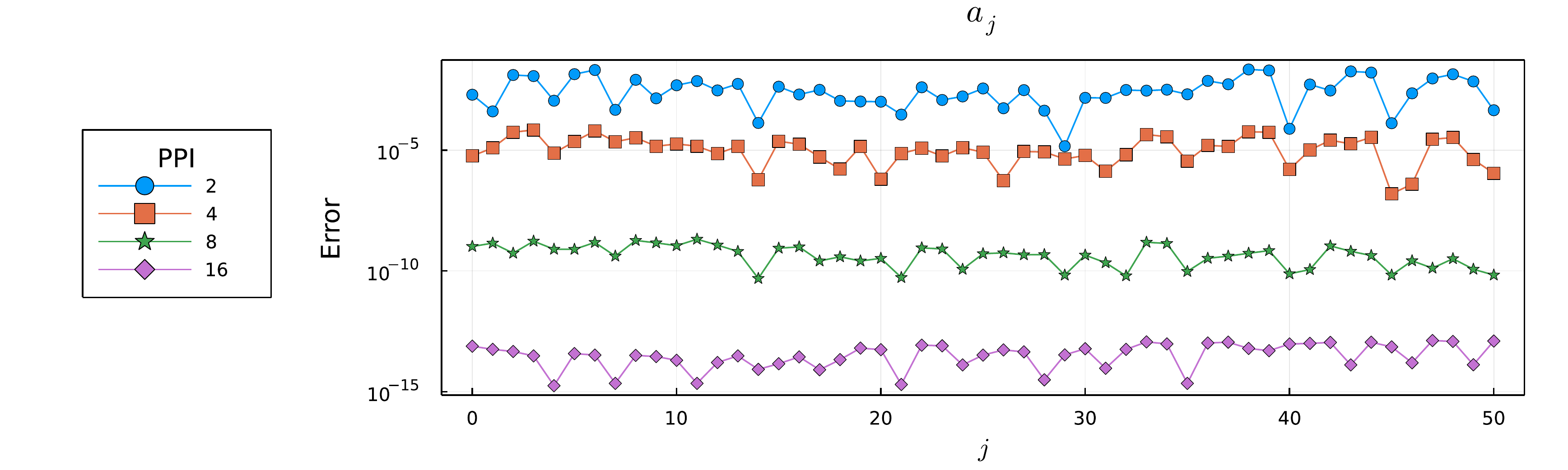}
	\end{subfigure}
	\begin{subfigure}{0.95\linewidth}
		\centering
		\includegraphics[width=\linewidth]{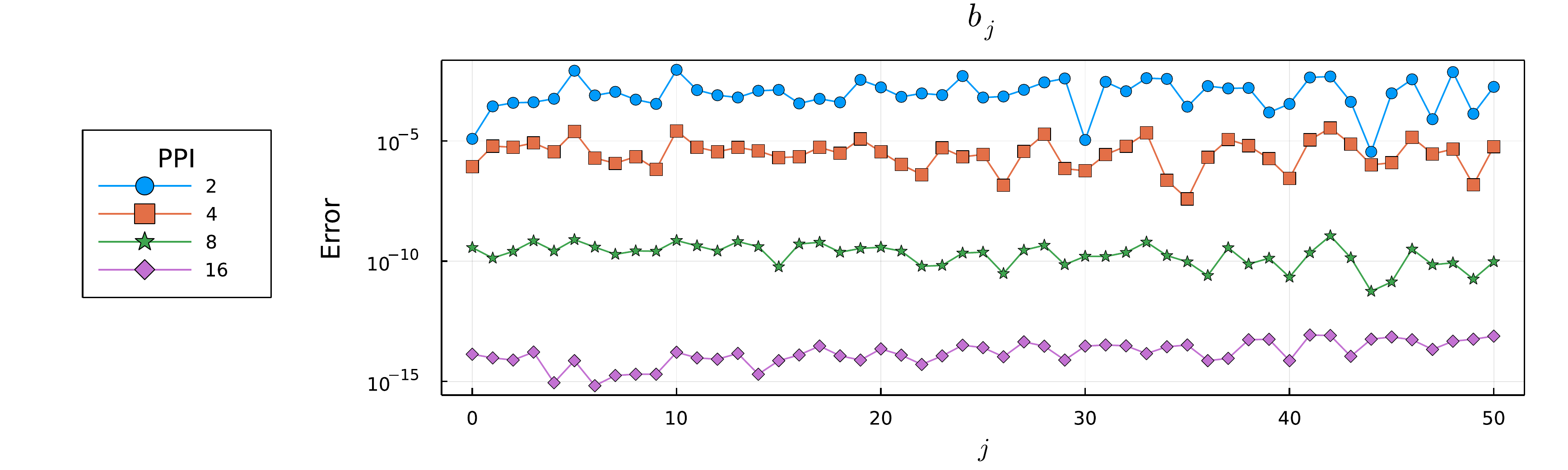}
	\end{subfigure}
	\caption{Error in the computed first 51 orthogonal polynomial coefficients for a Chebyshev-$\TT$-like weight on $[-3.2, -2.2]$, a Chebyshev-$U$-like weight on $[0.1, 1.1]$, a Chebyshev-$V$-like weight on $[2,3]$, and a Chebyshev-$W$-like weight on $[3.5,4]$ for 2, 4, 8, and 16 collocation points per interval (PPI). This weight function is plotted in Figure \ref{weightfunc}. Computing these coefficients requires 224.85 seconds the Householder method described in Section \ref{singint}, and 0.15, 0.19, 0.99, and 5.96 seconds, respectively, via our method.}
	\label{4interr}
\end{figure}

\subsection{The Toda lattice}
The semi-infinite Toda lattice as considered by Deift, Li, and Tomei \cite{DEIFT1985358} (see also \cite{MOSER1975197}) with boundary condition $b_{-1}(t)=0$ is defined by semi-infinite operators 
\begin{equation*}
\bX(t)=\begin{pmatrix}
	a_0(t) &b_0(t)\\
	b_0(t) &a_1(t) &b_1(t)\\
	&b_1(t) &a_2(t) &b_2(t)\\
	&&\ddots &\ddots &\ddots
\end{pmatrix},\quad \bB(t)=\begin{pmatrix}
0 &b_0(t)\\
-b_0(t) &0 &b_1(t)\\
&-b_1(t) &0 &b_2(t)\\
&&\ddots &\ddots &\ddots
\end{pmatrix},
\end{equation*}
using the relation 
\begin{equation*}
\dot\bX(t)=[\bB(t),\bX(t)],
\end{equation*}
where $[\cdot,\cdot]$ denotes the standard matrix commutator. Given an initial condition $\bX(0)=\bJ(w)$
corresponding to the Jacobi operator to some orthogonal polynomial weight $w$, the time-dependent $a_j(t),b_j(t)$ are given by the respective recurrence coefficients $a_j,b_j$ of the orthogonal polynomials with weight $\ex^{tx}w(x)$ \cite{DEIFT1985358} (see also \cite{Sogo1993}), i.e., if $\bX(0)=\bJ(w)$, then $\bX(t)=\bJ(w\ex^{t\diamond})$.

Since our method allows for weights to be scaled by positive analytic functions, we are able to numerically evolve the Toda lattice for initial conditions corresponding to the various orthogonal polynomial weights in our arsenal; however, the scaling of $\ex^{tx}$ grows exponentially in the complex plane, so we are only able to evolve the Toda lattice to a finite time horizon before encountering precision issues.

As a first example, we consider an initial condition of 
\begin{equation*}
\bX(0)=\begin{pmatrix}
	0 &1/2\\
	1/2 &0 &1/2\\
	&1/2 &0 &1/2\\
	&&\ddots &\ddots &\ddots
\end{pmatrix},
\end{equation*}
corresponding to a Chebyshev-$U$ weight on $[-1,1]$. Choosing to use 120 collocation points on $C_1$ and 20 on $[-1,1]$, we compute the first 11 recurrence coefficients with $h_1(x)=\ex^{tx}$ for times ranging from $t=0$ to $t=11.5$ and plot the results in Figure \ref{toda1int}. Numerical blowup is observed shortly after time $t=11.5$.
\begin{figure}
	\centering
	\begin{subfigure}{0.95\linewidth}
		\centering
		\includegraphics[width=\linewidth]{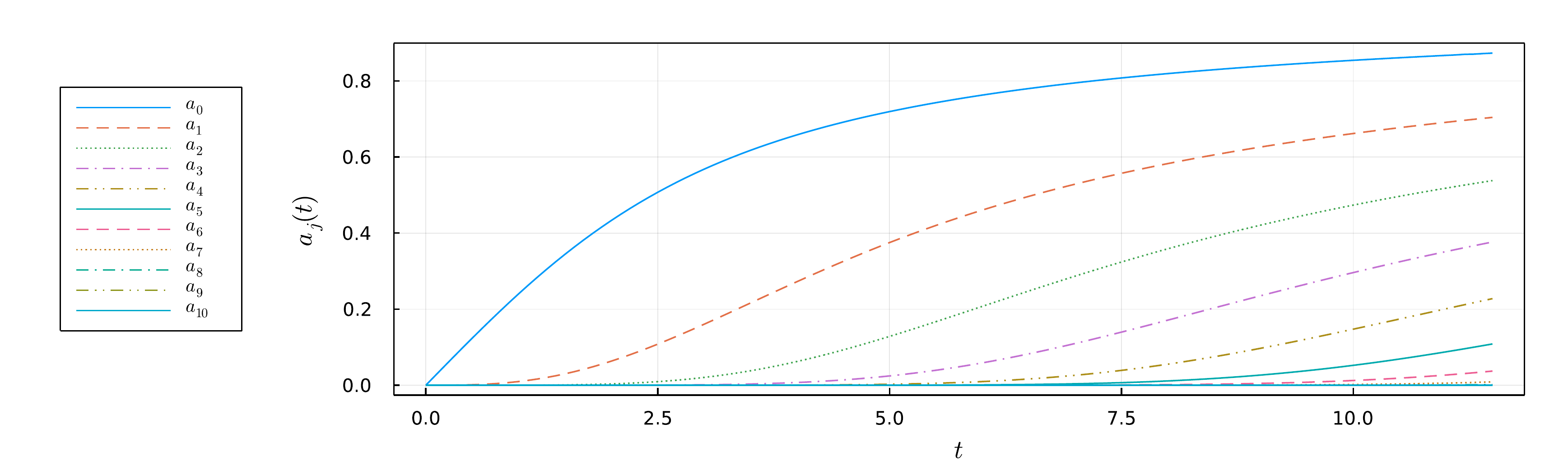}
	\end{subfigure}
	\begin{subfigure}{0.95\linewidth}
		\centering
		\includegraphics[width=\linewidth]{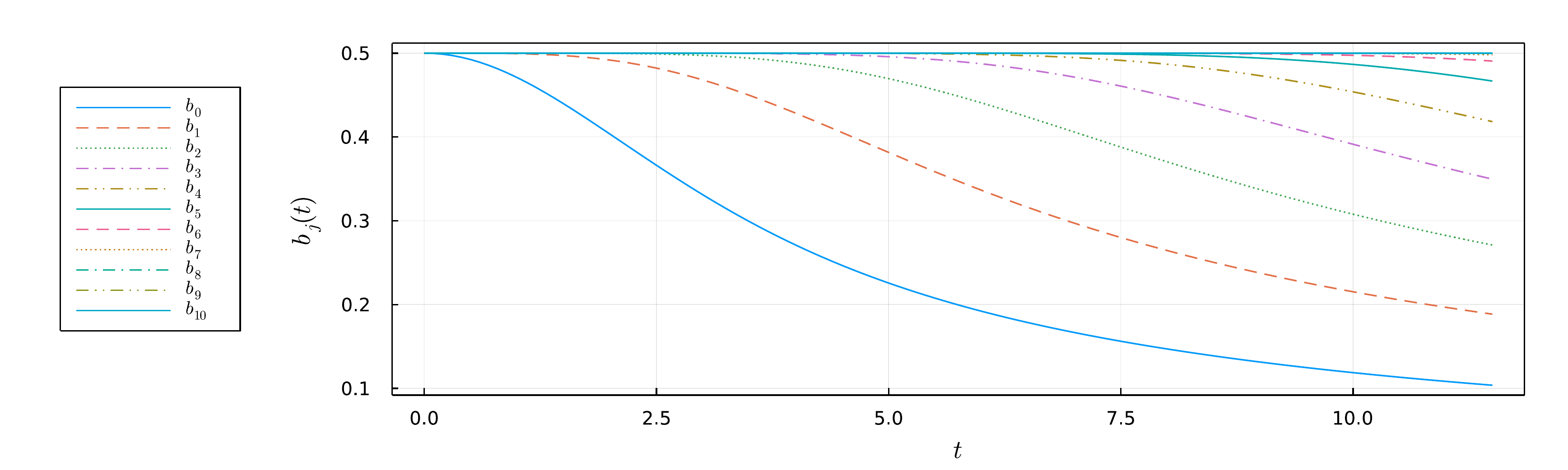}
	\end{subfigure}
	\caption{The evolution of the first 11 tridiagonal elements of the solution of the Toda lattice $\bX(t)$ with initial condition corresponding to a Chebyshev-$U$ weight on $[-1,1]$ from time $t=0$ to time $t=11.5$.}
	\label{toda1int}
\end{figure}

Now, we compute the time-evolution of the Toda lattice with initial condition $\bX(0)$ corresponding to the recurrence coefficients of orthogonal polynomials with a Chebyshev-$\TT$-like weight on $[-3,-2]\cup[2,3]$. Numerically, we observe that $a_j(0)=0$ for all $j$ and as $j\to\infty$, $b_j(0)$ alternates between $0.5$ and $2.5$ as displayed in Figure \ref{2int0}.
\begin{figure}
	\centering
	\begin{subfigure}{0.495\linewidth}
		\centering
		\includegraphics[width=\linewidth]{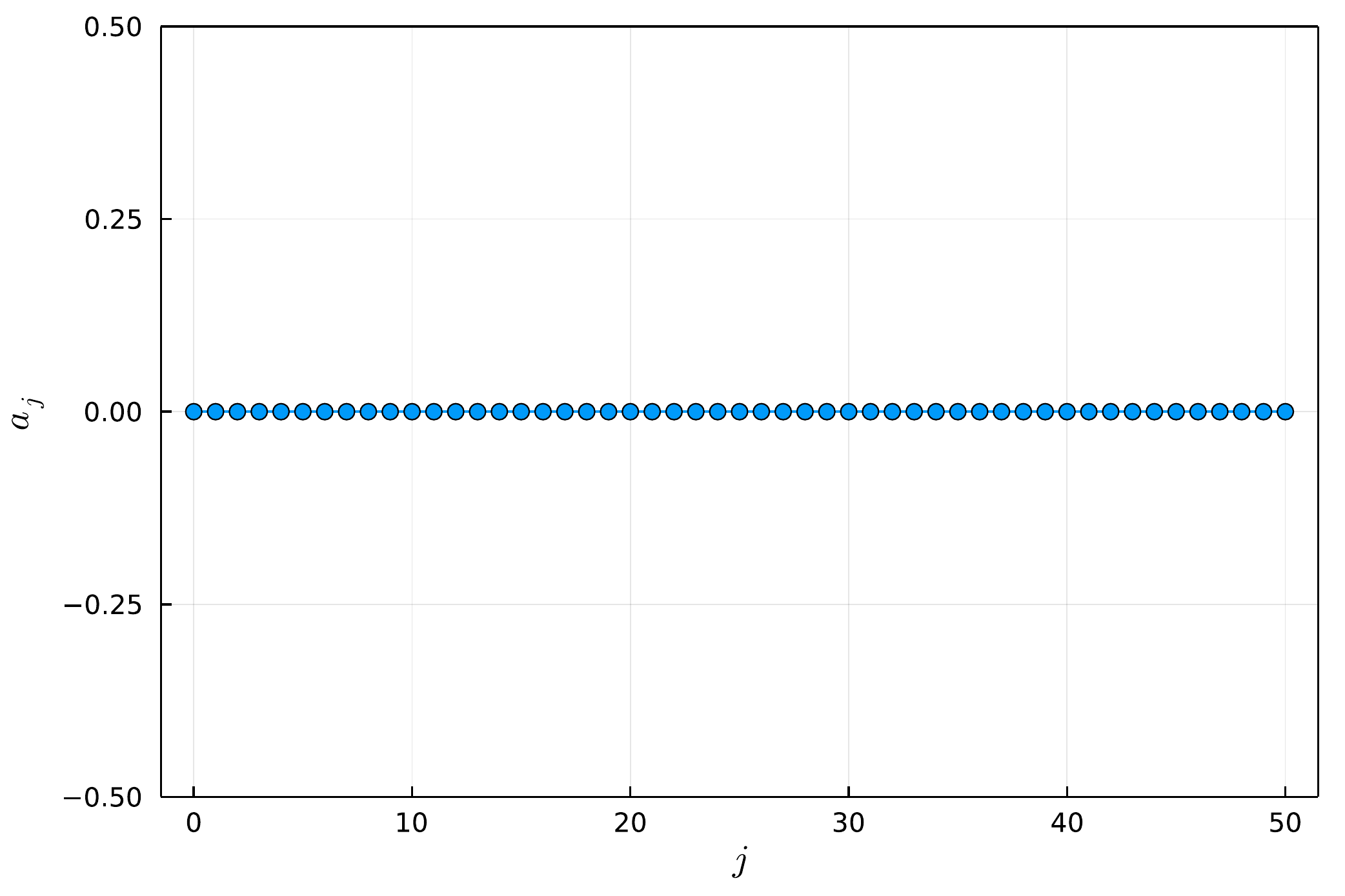}
	\end{subfigure}
	\begin{subfigure}{0.495\linewidth}
		\centering
		\includegraphics[width=\linewidth]{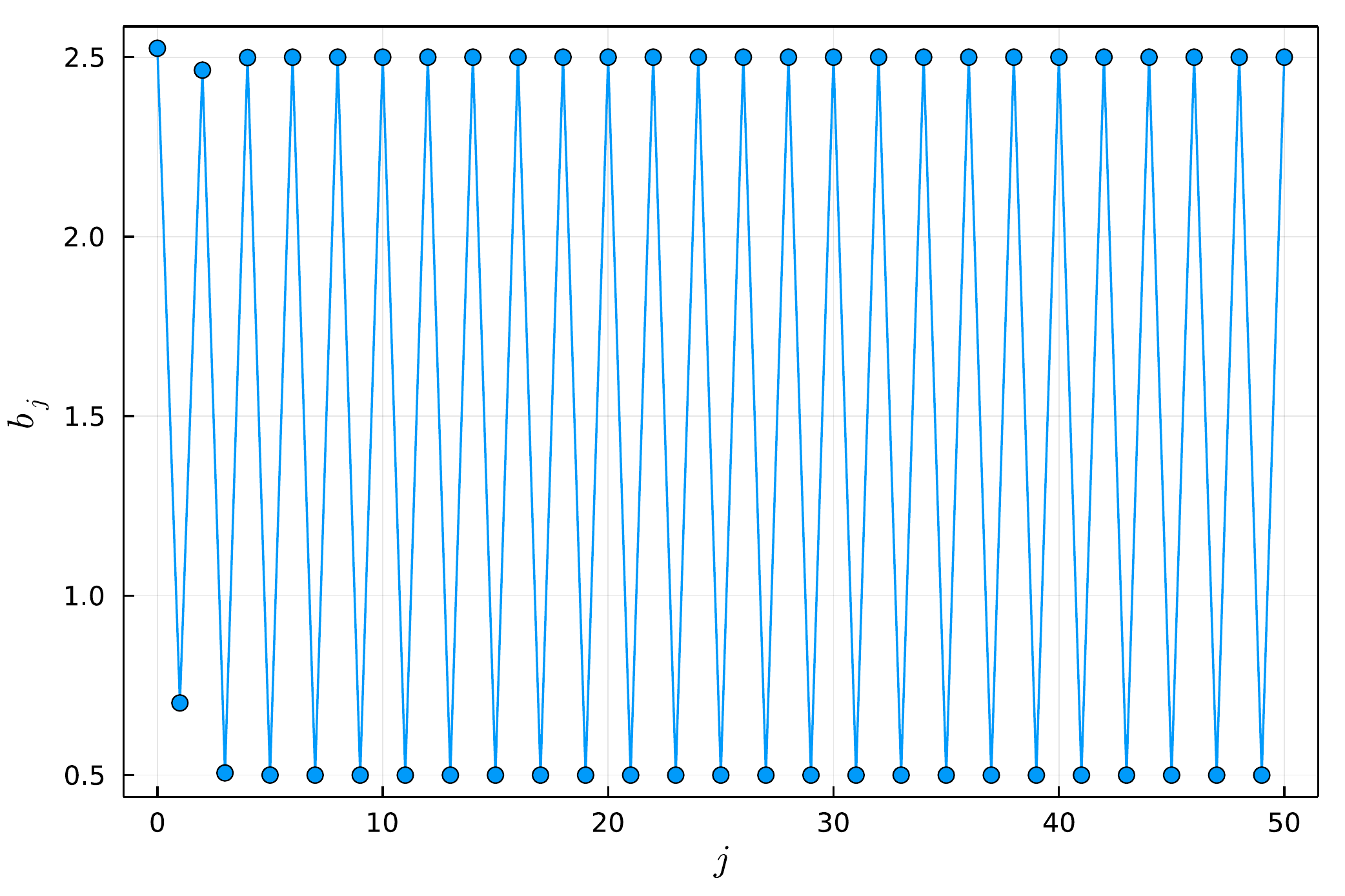}
	\end{subfigure}
	\caption{The first 51 recurrence coefficients for a Chebyshev-$\TT$-like weight on $[-3,-2]\cup[2,3]$ computed using 120 collocation points on each $C_j$ and 20 on each interval.}
	\label{2int0}
\end{figure}

We evolve the Toda lattice numerically for times ranging from $t=0$ to $t=4.5$ in the same manner as before\footnote{By taking $h_j(x)=\ex^{tx}$ for each $j$ and computing the first 11 recurrence coefficients at each time with 120 collocation points on each $C_j$ and 20 on each interval, we approximate the evolution of the Toda lattice.} and plot the results in Figure \ref{toda2int}.
\begin{figure}
	\centering
	\begin{subfigure}{0.95\linewidth}
		\centering
		\includegraphics[width=\linewidth]{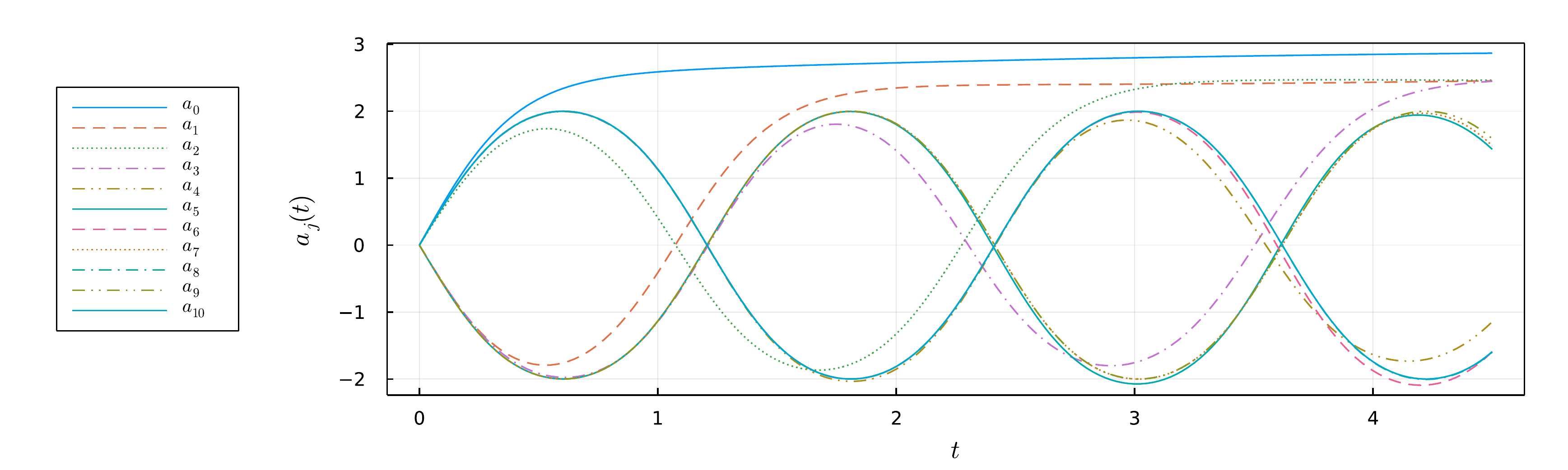}
	\end{subfigure}
	\begin{subfigure}{0.95\linewidth}
		\centering
		\includegraphics[width=\linewidth]{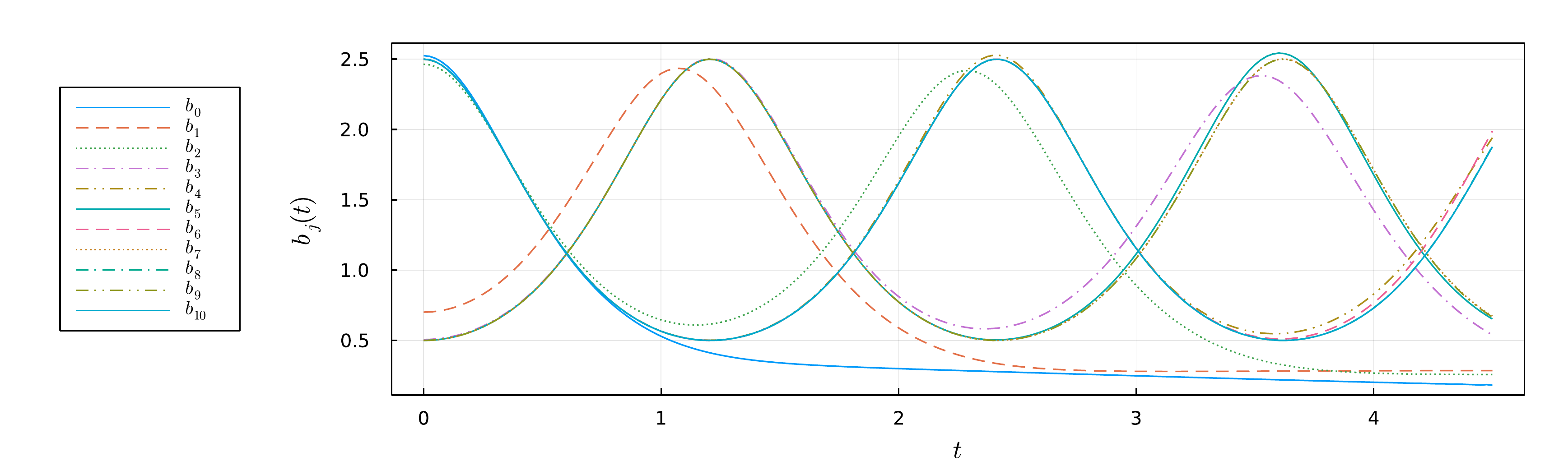}
	\end{subfigure}
	\caption{The evolution of the first 11 tridiagonal elements of the solution of the Toda lattice $\bX(t)$ with initial condition corresponding to a Chebyshev-$\TT$-like weight on $[-3,-2]\cup[2,3]$ from time $t=0$ to time $t=4.5$.}
	\label{toda2int}
\end{figure}

Additionally, we emphasize that our method allows us to compute any recurrence coefficient without first computing the others, meaning that the evolution of $a_j(t),b_j(t)$ can be computed for any $j$. We demonstrate this by plotting $a_{116}(t),b_{116}(t)$ in Figure \ref{toda116}.
\begin{figure}
	\centering
	\begin{subfigure}{0.495\linewidth}
		\centering
		\includegraphics[width=\linewidth]{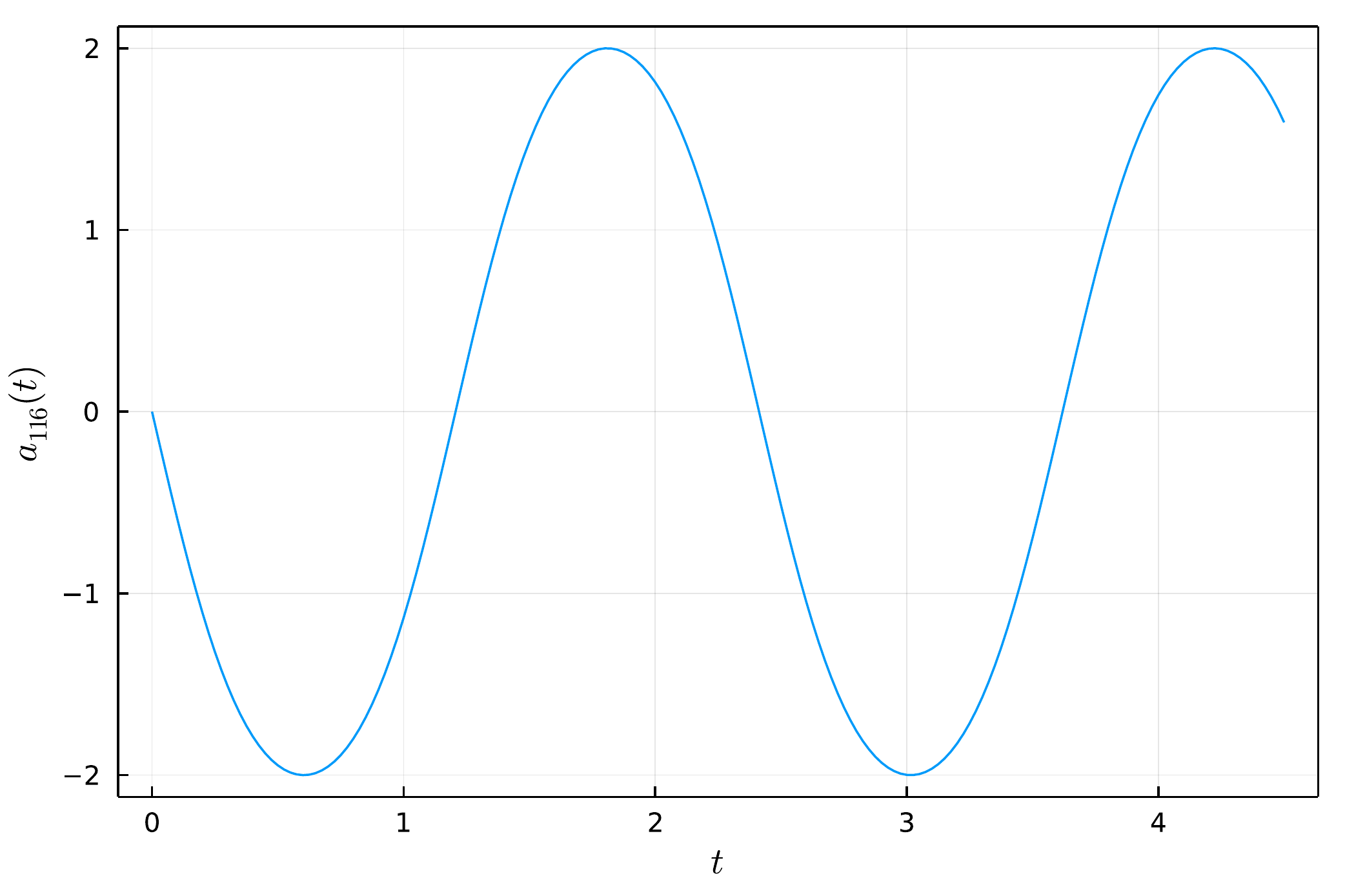}
	\end{subfigure}
	\begin{subfigure}{0.495\linewidth}
		\centering
		\includegraphics[width=\linewidth]{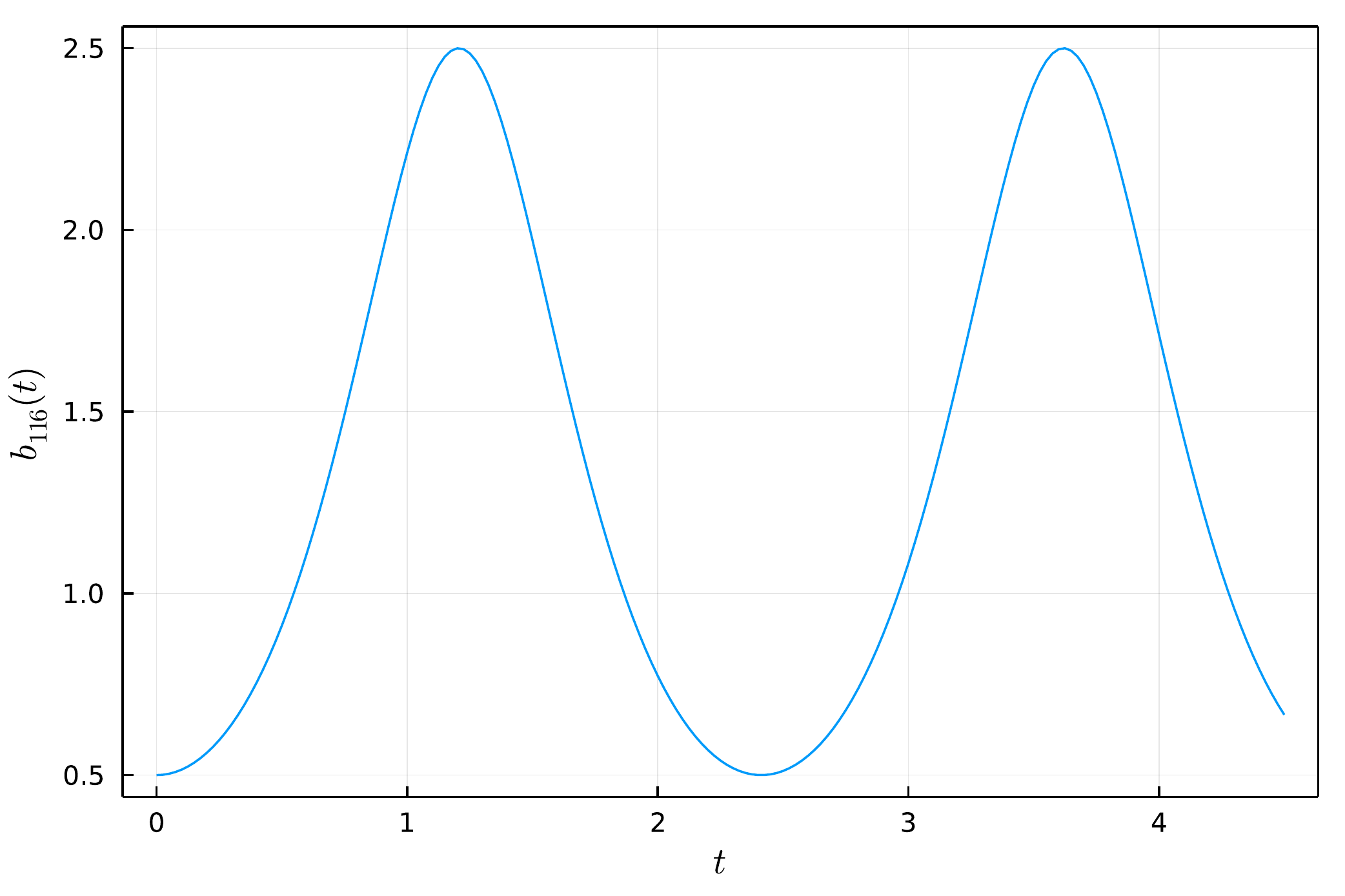}
	\end{subfigure}
	\caption{The evolution of $a_{116}(t),b_{116}(t)$ in the solution of the Toda lattice $\bX(t)$ with initial condition corresponding to a Chebyshev-$\TT$-like weight on $[-3,-2]\cup[2,3]$ from time $t=0$ to time $t=4.5$.}
	\label{toda116}
\end{figure}

Finally, we consider the time-evolution of the Toda lattice with initial condition $\bX(0)$ corresponding to the recurrence coefficients of orthogonal polynomials with a Chebyshev-$W$-like weight on $[0.1, 1.1]\cup[2, 3]\cup[3.5, 4]$. Now, the initial condition as displayed as in Figure \ref{3int0} appears almost noisy. 
\begin{figure}
	\centering
	\begin{subfigure}{0.495\linewidth}
		\centering
		\includegraphics[width=\linewidth]{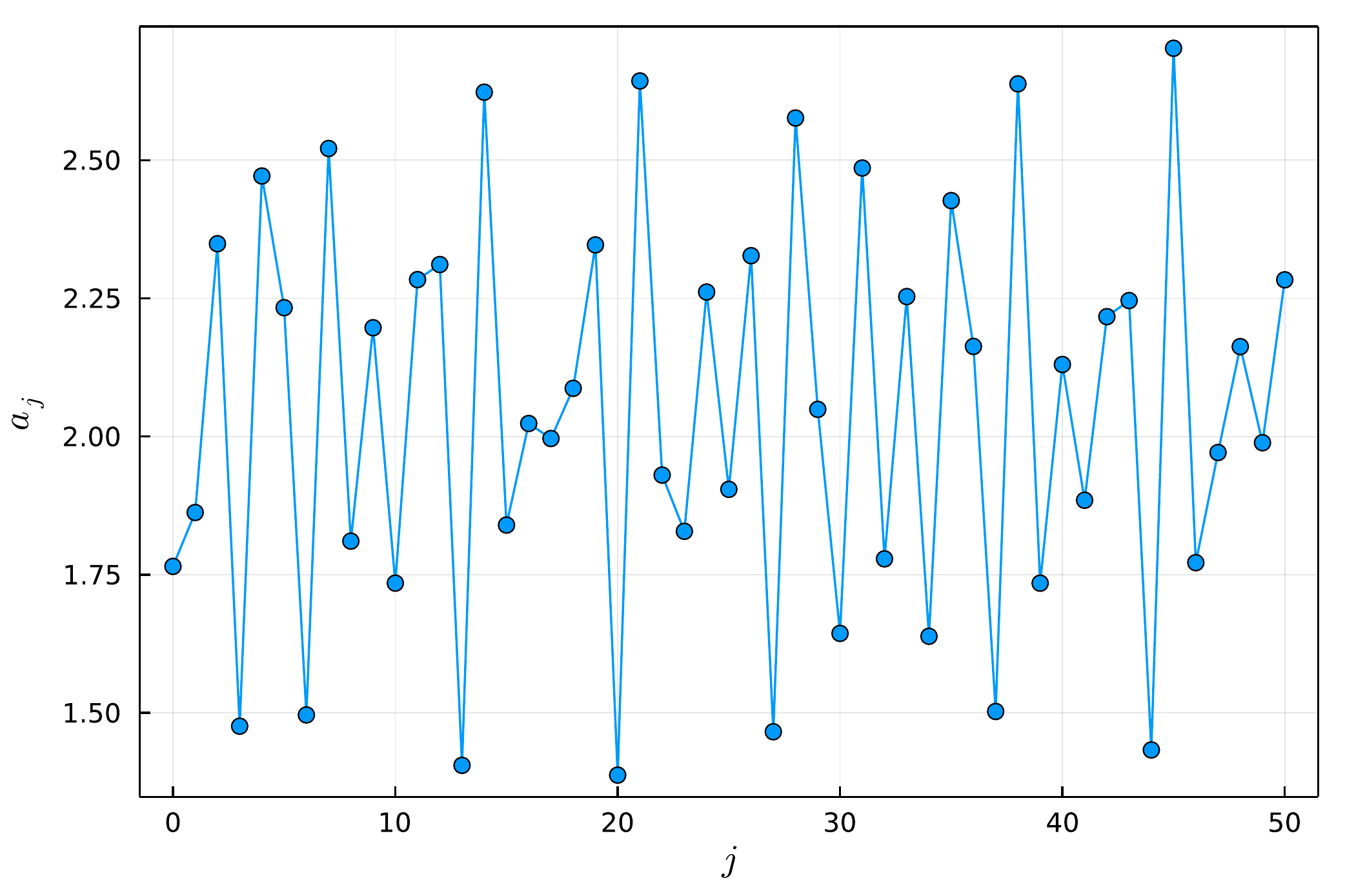}
	\end{subfigure}
	\begin{subfigure}{0.495\linewidth}
		\centering
		\includegraphics[width=\linewidth]{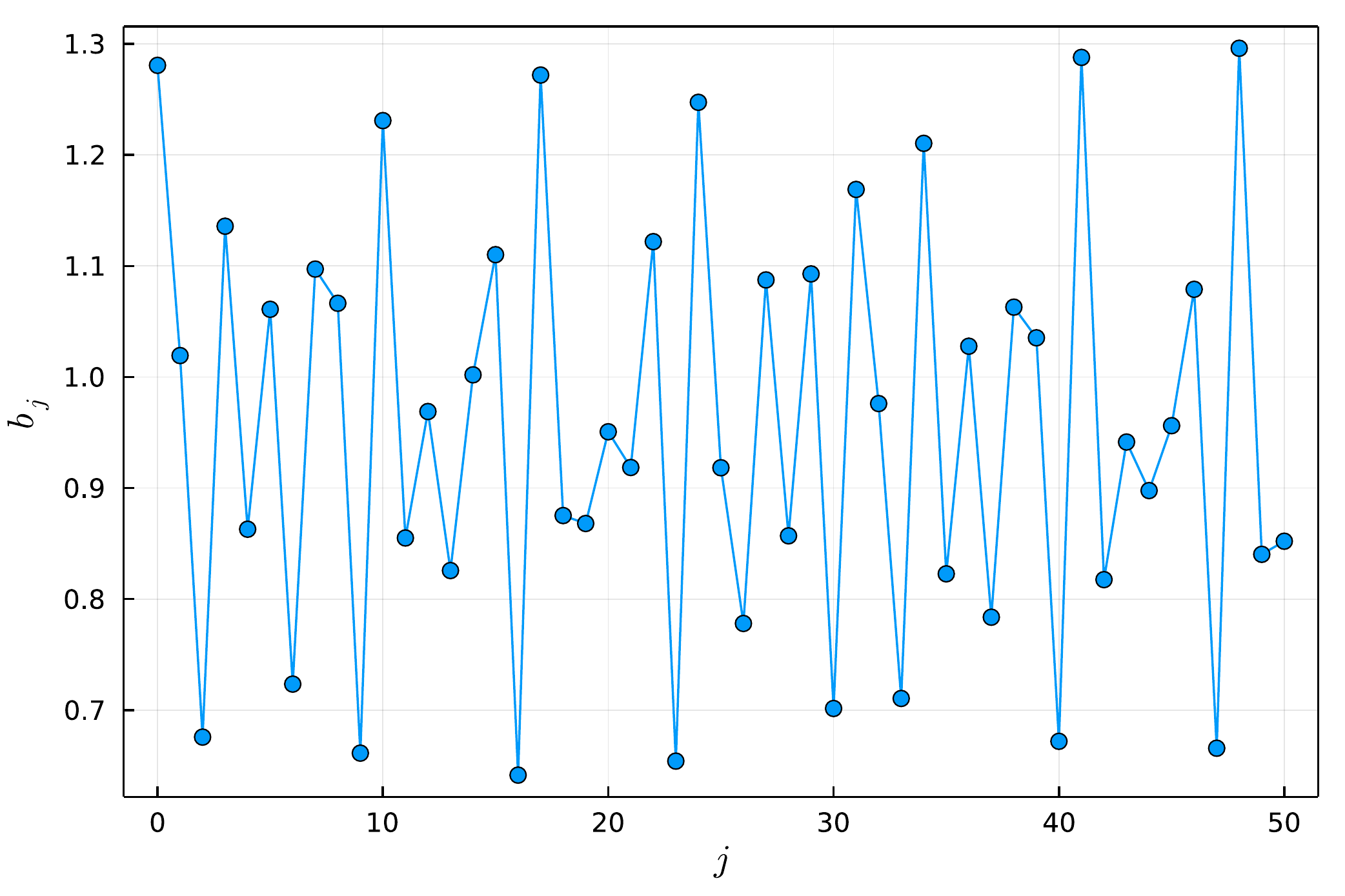}
	\end{subfigure}
	\caption{The first 51 recurrence coefficients for a Chebyshev-$W$-like weight on $[0.1, 1.1]\cup[2, 3]\cup[3.5, 4]$ computed using 120 collocation points on each $C_j$ and 20 on each interval.}
	\label{3int0}
\end{figure}
We evolve the Toda lattice numerically for times ranging from $t=0$ to $t=5.7$ in the same manner as before and plot the results in Figure \ref{toda3int}. 
\begin{figure}
	\centering
	\begin{subfigure}{0.95\linewidth}
		\centering
		\includegraphics[width=\linewidth]{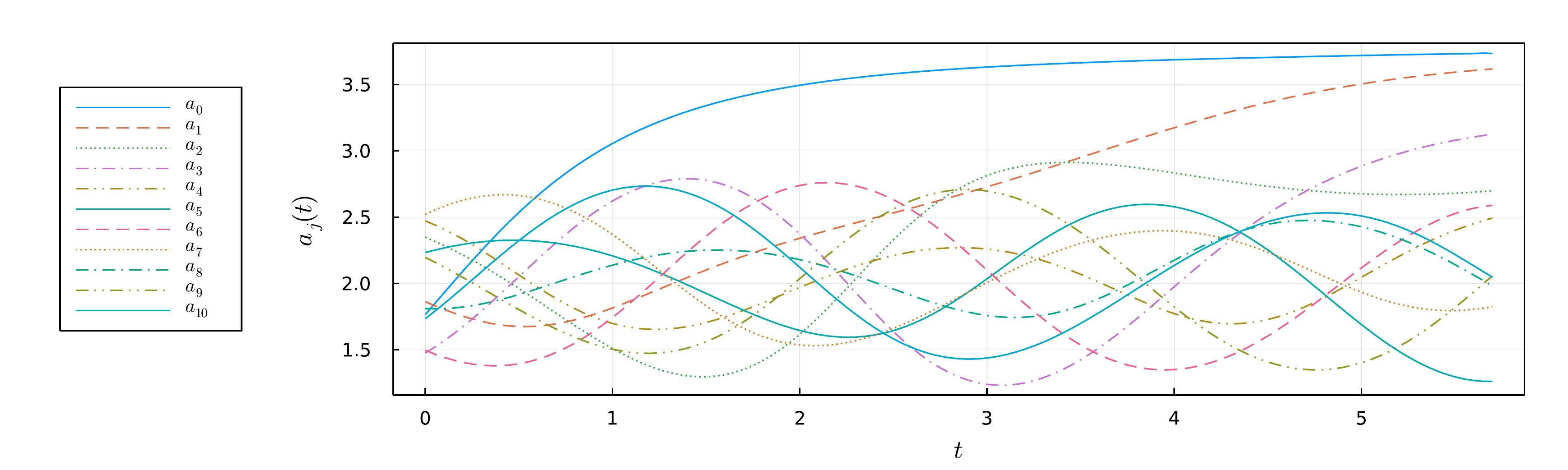}
	\end{subfigure}
	\begin{subfigure}{0.95\linewidth}
		\centering
		\includegraphics[width=\linewidth]{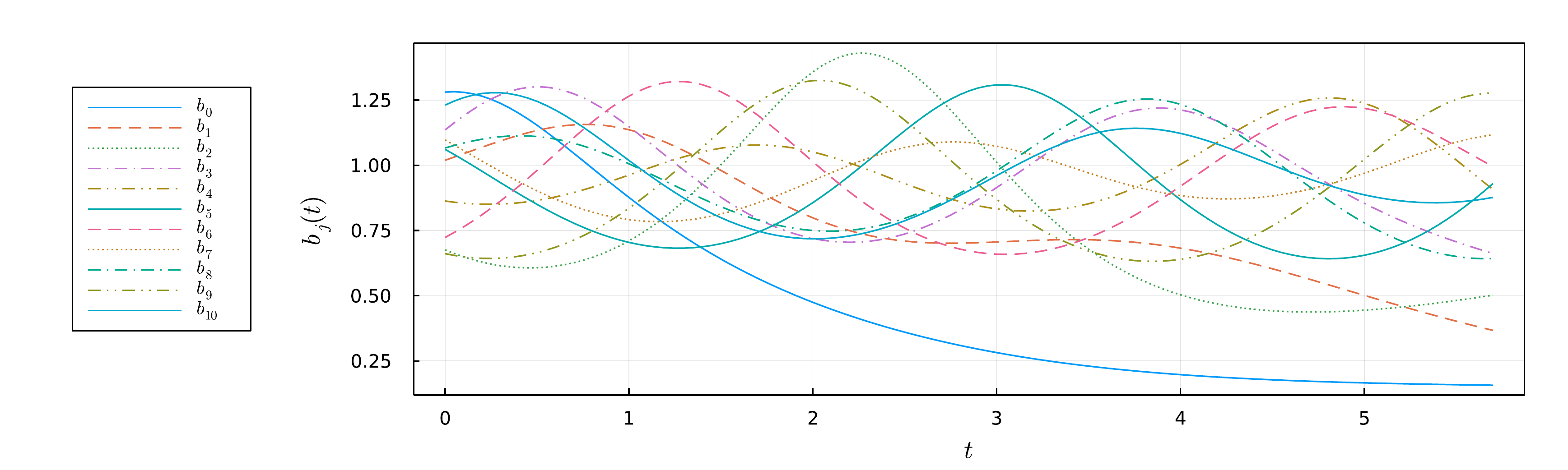}
	\end{subfigure}
	\caption{The evolution of the first 11 tridiagonal elements of the solution of the Toda lattice $\bX(t)$ with initial condition corresponding to a Chebyshev-$W$-like weight on $[0.1, 1.1]\cup[2, 3]\cup[3.5, 4]$ from time $t=0$ to time $t=5.7$.}
	\label{toda3int}
\end{figure}
\subsection{Building an approximation to $1/x$}
As a final example, we demonstrate that our procedure can be used to approximate functions on disconnected domains. Specifically, we consider an approximation to the function $1/x$. If the orthonormal polynomials corresponding to the normalization of the weight \eqref{weight} are given by $p_0,p_1,\ldots$, then a $p_j$-series expansion for $1/x$, for $x \in \Sigma$, is given by
\begin{equation}\label{1/xform}
\begin{aligned}
\frac{1}{x}&=\sum_{j=0}^{\infty}\frac{1}{\eta}\left\langle p_j,\frac{1}{\diamond}\right\rangle p_j(x)=\sum_{j=0}^{\infty}\frac{1}{\eta}\left(\int_\Sigma\frac{p_j(z)}{z}w(z)\df z \right)p_j(x)=
\sum_{j=0}^{\infty}\frac{2\pi\im}{\eta}\mathcal{C}_\Sigma\left[p_jw\right](0)p_j(x),
\end{aligned}
\end{equation}
where $\Sigma=\bigcup_{j=1}^{g+1}[a_j,b_j]$,  and $\eta=\int_\Sigma w(x)\df x$. Allowing for unnormalized weight functions mandates the inclusion of the $\frac{1}{\eta}$ constant that would not be present if $w$ was a probability measure or if $p_0,p_1,\ldots$ corresponded to the unnormalized weight. Our definition of $p_0,p_1,\ldots$ allows for the Cauchy integrals of these polynomials to be computed without also computing $\eta$. Thus, we need only to evaluate the Cauchy integral of our polynomials at zero to build an approximation. To do this, we recall that the (1,2) entry of the true solution \eqref{exactsol} to the original Riemann--Hilbert problem \eqref{RHPOG} is the Cauchy integral of the monic orthogonal polynomial, meaning that 
\begin{equation*}
\mathcal{C}_\Sigma\left[p_nw\right](z)=\gamma_n\left(\bY_n(z)\right)_{12},
\end{equation*}
where $\gamma_n$ is the coefficient on the $x^n$ term in $p_n(x)$. \cite{deift_2000} gives that
\begin{equation*}
	\gamma_n=\prod_{j=0}^{n-1}b_j^{-1},
\end{equation*}
with $\gamma_0=1$, so this term can be easily computed if we are already computing the recurrence coefficients. To obtain this from a computed solution $\Tilde \bS_n(z)$, we undo our series of transformations. If we assume that $z\notin\Sigma$ and furthermore that $z\notin \bar D_j$ where $D_j$ is the interior of $C_j$ for any $j$, then we need not worry about undoing the lensing step. In fact, the second column of $\bY_n$ is unaffected by the lensing step, so this need not be accounted for when computing $\mathcal{C}_\Sigma\left[p_nw\right](z)$, regardless of whether $z\in \bar D_j$. Since the remaining steps are diagonal, we can compute
\begin{equation}\label{cauchyints}
\mathcal{C}_\Sigma\left[p_nw\right](z)=\left(\prod_{j=0}^{n-1}\frac{1}{b_j\mathfrak{c}}\right)\left(\Tilde \bS_n(z)\right)_{12}\ex^{\mathfrak{h}_n(z)-n\mathfrak{g}(z)}.
\end{equation}
For our purposes, we only need to evaluate this at $z=0$, so this can easily be done in parallel with computing the recurrence coefficients by computing $\mathfrak{g}(0)$ once at the beginning and computing $\mathfrak{h}_n(0)$ each time we solve the Riemann--Hilbert problem \eqref{RHPfinal}.

Using \eqref{1/xform} and \eqref{cauchyints}, we build an approximation to $1/x$ on $\Sigma=[-4, -3]\cup[-2, -1]\cup[2, 3]$ as well as $\Sigma=[1, 2.3]\cup[3, 4]\cup[4.5, 6.1]$ where each Riemann--Hilbert problem is numerically solved with 120 collocation points on each $C_j$ and 20 points on each interval. To check accuracy, we consider an evenly spaced grid by 0.01 $x_1,x_2,\ldots$ on each domain $\Sigma$ and consider the max-norm error 
\begin{equation*}
\max_k\left|
\sum_{j=0}^{N-1}\frac{2\pi\im}{\eta}\mathcal{C}_\Sigma\left[p_jw\right](0)p_j(x_k)-\frac{1}{x_k}\right|,
\end{equation*}
where $N$ is the number of polynomials that we use. In Figure \ref{1/xapprox}, we plot this computed error for various numbers of terms in the expansion \eqref{1/xform}.
\begin{figure}
	\centering
	\begin{subfigure}{0.495\linewidth}
		\centering
		\includegraphics[width=\linewidth]{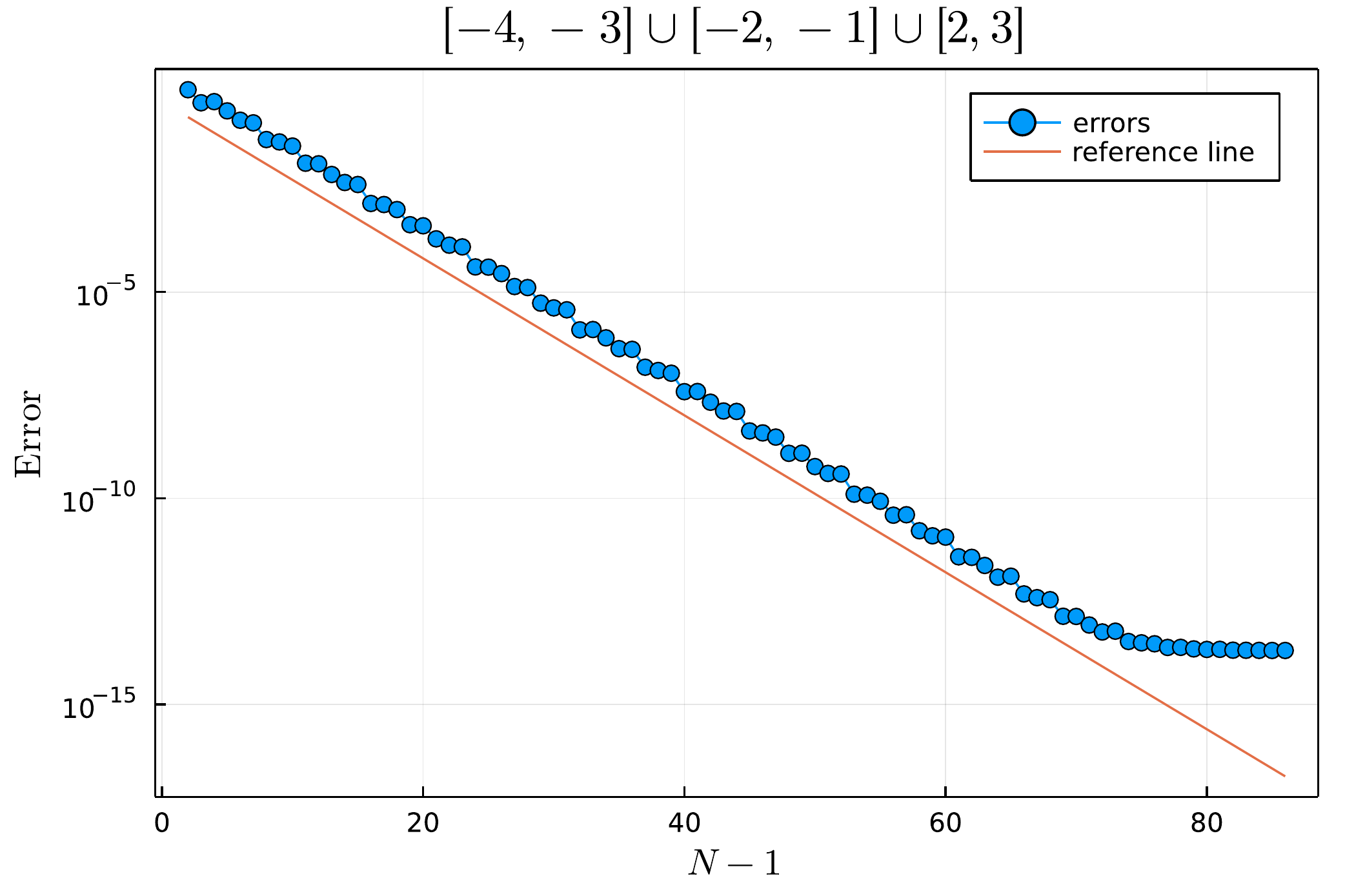}
	\end{subfigure}
	\begin{subfigure}{0.495\linewidth}
		\centering
		\includegraphics[width=\linewidth]{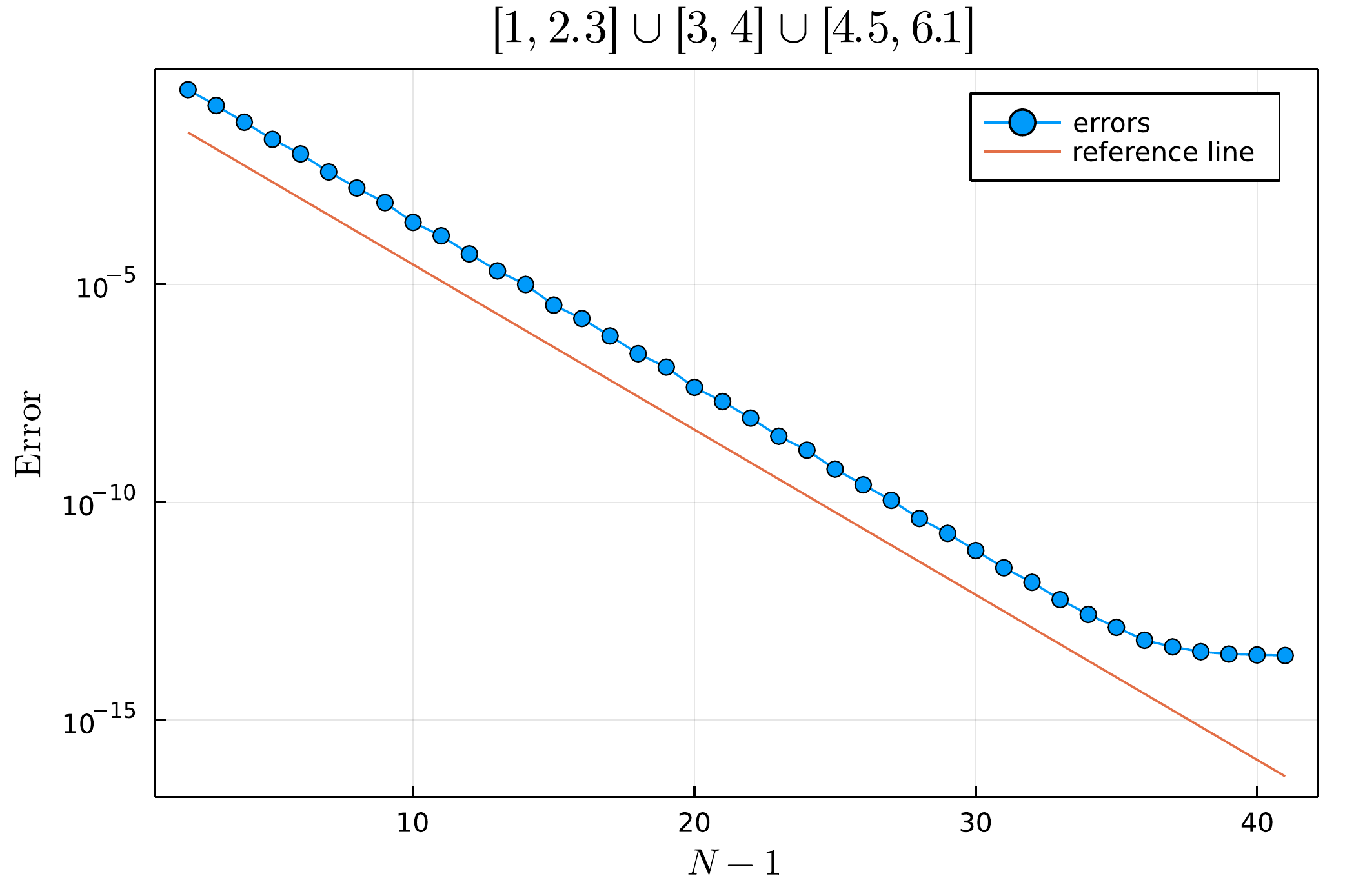}
	\end{subfigure}
	\caption{The max norm error in a Chebyshev-$\TT$-like expansion to approximate $1/x$ on $\Sigma=[-4, -3]\cup[-2, -1]\cup[2, 3]$ (left) and $\Sigma=[1, 2.3]\cup[3, 4]\cup[4.5, 6.1]$ (right) where $N$ denotes the number of polynomials used. Much faster convergence is observed when $0$ is not contained in the gaps between intervals. The rate of convergence is well-approximated by the reference line $\left|e^{-\mathfrak{g}(0)N}\right|$.}
	\label{1/xapprox}
\end{figure}

It is also possible for us to determine the rate of convergence for our orthogonal polynomial approximation. From \cite{Ding2022}, we have that if $0$ is outside the deformation region of our Riemann--Hilbert problem,
\begin{equation*}
2\pi\im\mathcal{C}_\Sigma\left[p_nw\right](0)=\delta_n\ex^{-n\mathfrak{g}(0)},
\end{equation*}
where $\delta_n$ depends on $\prod_{j=0}^{n-1}\frac{1}{b_j\mathfrak{c}}$,and other constants determined by the asymptotic behavior of $\bY_n$ at infinity \cite{Ding2022}. By \cite[Equations B.3, B.12]{Ding2022}, 
\begin{equation*}
\mathfrak{c}^{n}\prod_{j=0}^{n-1}b_j=\OO(1),
\end{equation*}
as $n\to\infty$. This implies that $\delta_n$ is bounded from above, so we have that
\begin{equation*}
\left\|
\sum_{j=0}^{N-1}\frac{2\pi\im}{\eta}\mathcal{C}_\Sigma\left[p_jw\right](0)p_j-\frac{1}{\diamond}\right\|^2_{L_w^2(\Sigma)}=c_1\sum_{j=N}^\infty\delta_j^2\ex^{-2j\re \mathfrak g(0)}\leq c\frac{\ex^{-2N\re \mathfrak{g}(0)}}{1-\ex^{-2\re \mathfrak{g}(0)}},
\end{equation*}
where $c$ and $c_1$ are constants. This implies that our approximation converges uniformly, at worst, at a geometric rate given by $
|\ex^{-\mathfrak{g}(0)}|$. We believe $\delta_n$ to also be bounded below which would imply that convergence is truly given by this rate. This means that in practice, one can compute $\mathfrak{g}(0)$ and the error at an initial approximation that will then allow them to estimate the number of terms necessary to reach a desired accuracy.

\section{Conclusion and future work}
This work presents a novel numerical method for computing orthogonal polynomials that are orthogonal on multiple, disjoint intervals, where analytical formulae are currently unavailable. By leveraging the Fokas--Its–-Kitaev Riemann–-Hilbert representation of these polynomials, our approach achieves $\OO(N)$ complexity for computing the first $N$ recurrence coefficients. Moreover, our method enables efficient pointwise evaluation of the polynomials and their Cauchy transforms across the entire complex plane.

The efficiency of our method is currently limited for low degree polynomials by the jump conditions on the circles $C_j$ around each interval; however, there exist functions $h_j$ such that the analytic continuations of $w_j, w_k$ satisfy $w_j(z) = w_k(z)$ for all $j,k$ and individual circles around each interval $[a_j,b_j]$ could be replaced with a large circle around $[a_1,b_{g+1}]$, and then removed entirely. In such a regime, our method is competitive (and better asymptotically) when compared with the $\OO(N^2)$ method presented in \cite{Gragg1984, gautschi}. In future work, we will explore these classes of weight functions and apply their respective orthogonal polynomials to indefinite linear systems and matrix function approximation in the vein of \cite{Saad}.

Finally, our approach potentially gives rise to new techniques for computing Riemann theta functions. The known asymptotic formulae for the recurrence coefficients of the orthogonal polynomials that we consider consist of the ratio of theta functions composed with the Abel map \cite{Ding2022, YATTSELEV201573}. In our method, the asymptotics of the recurrence coefficients are essentially governed by the auxiliary function $\mathfrak h_n$. Since we can efficiently evaluate $\mathfrak h_n$ in the complex plane, this could potentially enable the efficient evaluation of theta functions. We will also explore this further in future work.


\bibliographystyle{plain}
\bibliography{refs}
\end{document}